\documentclass[leqno, 10pt]{amsart}
\usepackage{a4wide}
\usepackage{amsmath, amssymb, latexsym, mathrsfs, slashed, tikz-cd, mathtools, mathabx, shuffle, tikz-cd}

 \newtheorem{theorem}{Theorem}[section]
 \newtheorem{lemma}[theorem]{Lemma}
 \newtheorem{proposition}[theorem]{Proposition}
 \newtheorem{corollary}[theorem]{Corollary}

 \newtheorem*{theorem*}{Theorem}
\newtheorem*{proposition*}{Proposition}
\newtheorem*{lemma*}{Lemma}

\theoremstyle{definition}
\newtheorem{definition}[theorem]{Definition}

 \theoremstyle{remark} 
 \newtheorem{example}[theorem]{Example}
 \newtheorem{remark}[theorem]{Remark}

  \newtheorem*{claim*}{Claim}

\newcommand{\op}[1]{\operatorname{#1}}

\newcommand{\acou}[2]{\ensuremath{\left\langle #1 , #2 \right\rangle}}

\def\XXint#1#2#3{{\setbox0=\hbox{$#1{#2#3}{\int}$}
\vcenter{\hbox{$#2#3$}}\kern-.5\wd0}}

\newcommand{\C}{\ensuremath{\mathbb{C}}} 
 
\newcommand{\Q}{\ensuremath{\mathbb{Q}}}

\newcommand{\Z}{\ensuremath{\mathbb{Z}}}

\newcommand{\Ca}[1]{\ensuremath{\mathcal{#1}}}
\newcommand{\cA}{\Ca{A}}

\newcommand{\sE}{\mathscr{E}}

\newcommand{\Hom}{\op{Hom}}

\newcommand{\ran}{\op{ran}}

\newcommand{\bt}{\bullet}

\newcommand{\HC}{\op{HC}}

\newcommand{\HP}{\op{HP}}

\newcommand{\wbC}{\widebar{C}} 
\newcommand{\wb}[1]{\widebar{#1}}

\newcommand{\dl}{\partial} 

\newcommand{\nnatural}{{\natural\natural}}

\numberwithin{equation}{section}

\begin{document}

\title{Periodicity and Cyclic Homology. Para-$S$-Modules and Perturbation Lemmas}
 \author{Rapha\"el Ponge}
 \address{Department of Mathematical Sciences, Seoul National University, 1 Gwanak-ro, Gwanak-gu, Seoul 08826, Republic of Korea}
\curraddr{School of Mathematics, Sichuan University, No.~24 South Section 1, Yihuan Road, Chengdu, PR China, 610065}
 \email{ponge.math@icloud.com}

 \thanks{Research supported by Basic Research Grant 2016R1D1A1B01015971
 of National Research Foundation of Korea (South Korea).}
 
  \keywords{Cyclic homology, perturbation lemmas}

\subjclass[2010]{19D55}

\begin{abstract}
In this paper, we introduce a paracyclic version of $S$-modules. These new objects are called \emph{para-$S$-modules}. Paracyclic modules and parachain complexes give rise to para-$S$-modules much in the same way as cyclic modules and mixed complexes give rise to $S$-modules. More generally, para-$S$-modules provide us with a natural framework to get analogues for paracyclic modules and parachain complexes of various constructions and equivalence results for cyclic modules or mixed complexes. The datum of a para-$S$-module does not provide us with a chain complex, and so notions of homology and quasi-isomorphisms do not make sense. We establish some generalizations for para-$S$-modules and parachain complexes of the basic perturbation lemma of differential homological algebra. These generalizations provide us with general recipes for converting deformation retracts of Hoschschild chain complexes into deformation retracts of para-$S$-modules. By using ideas of Kassel this then allows us to get comparison results between the various para-$S$-modules associated with para-precyclic modules, and between them and Connes' cyclic chain complex. These comparison results lead us to alternative descriptions of Connes' periodicity operator. This has some applications in periodic cyclic homology. We also describe the counterparts of these results in cyclic cohomology. In particular, we obtain an explicit way to convert a periodic $(b,B)$-cocycle into a cohomologous periodic cyclic cocycle.
\end{abstract}

\maketitle

\section{Introduction} 
In the terminology of Jones-Kassel~\cite{JK:KT89, Ka:Crelle90} an  $S$-module is  the datum of a chain complex $(C_\bt,d)$ of modules over some ring $k$ together with a degree~$-2$ chain map $S:C_\bt \rightarrow C_{\bt -2}$. This notion has its roots in the seminal note of Connes~\cite{Co:CRAS83}. It naturally comes out in bivariant cyclic theory~\cite{Co:CRAS83, CQ:JAMS95, JK:KT89, Ka:Crelle90, Ni:AM93}. Incidentally, it encapsulates various approaches to cyclic homology. In particular, we obtain $S$-modules from the cyclic complex and $(b,B)$-bicomplex of Connes~\cite{Co:MFO81,  Co:CRAS83, Co:IHES85}, from any mixed  mixed complex~\cite{Bu:CM86, Ka:JAlg87}, and from the $CC$-bicomplex of Connes~\cite{Co:CRAS83} and Tsygan~\cite{Ts:UMN83} with the operators $(b,-b', 1-\tau, N)$. 

A para-$S$-module is similar to an $S$-module where the chain map property $[d,S]=0$ is preserved, but the condition $d^2=0$ is relaxed into 
\begin{equation}
 d^2=(1-T)S,
 \label{eq:Intro.para-S-condition}
\end{equation}
where $T:C_\bt \rightarrow C_{\bt}$ is a $k$-linear isomorphism which is compatible with both $d$ and $S$ (see Section~\ref{sec:Para-S-Mod} for the precise definition). When $T=1$ we recover the definition of an $S$-module. Para-$S$-modules can be also thought of as the unperiodic version of the para-complexes of~\cite{Vo:JIMJ07}. Indeed, periodizing a para-$S$-module with respect to the $S$-operator precisely gives rise to a para-complex (see Section~\ref{sec:Para-S-Mod}). 

The aim of this paper is to lay down the main ground of a paracyclic version of $S$-modules~\cite{Co:CRAS83, JK:KT89, Ka:Crelle90}. We call these objects \emph{para-$S$-modules}. In particular, para-$S$-modules are naturally associated with paracyclic modules and parachain complexes much like $S$-modules are associated with cyclic modules and mixed complexes. 

In order to obtain a general apparatus for establishing equivalence results for para-$S$-modules, we generalize to this setting the  basic perturbation lemma of differential homological algebra~\cite{Br:Messina64, EM:AM47, Gu:IJM72,  Sh:IHES62}. As we shall see in this paper, this allows us to extend to the paracyclic category various of well known constructions and equivalence results in cyclic homology. In particular, at the level of cyclic cohomology this will lead us to an explicit transformation to convert $(b,B)$-cocycles into cohomologous \emph{periodic} cocycles.

In the follow-up paper~\cite{Po:EZ} the generalized perturbation theory of this paper is used extensively to obtain a constructive version of the Eilenberg-Zilber theorem for bi-paracyclic modules (see also~\cite{KR:CMB04}). This allows us to get explicit cup and cap products for parachain complexes associated with paracyclic modules. In particular, these constructions are expected to have applications in Hopf cyclic cohomology. 

The main results of this paper and of the follow-up paper~\cite{Po:EZ} are important ingredients in the construction of explicit quasi-isomorphisms that compute the cyclic homology of crossed-product algebras associated with actions of (discrete) groups~\cite{Po:CRAS17a, Po:CRAS17b, Po:JGP18}.

\subsection*{Examples of para-$S$-modules} 
A first set of examples of para-$S$-modules consists of para-$S$-modules associated with the parachain complexes of Getzler-Jones~\cite{GJ:Crelle93} (see Section~\ref{sec:parachain}). The  cyclic complex $(C^\natural_\bt, b+Bu^{-1})$ of a mixed complex $(C_\bt, b,B)$ (see~\cite{Co:MFO81, Co:CRAS83, Co:IHES85, Bu:CM86, Ka:JAlg87}) carries a natural $S$-module structure given by the projection $u^{-1}: C^\natural_\bt \rightarrow C^\natural_{\bt-2}$. The para-$S$-module of a parachain complex is defined similarly. Thus, although for a parachain complex we don't get a chain complex in general, we still get a para-$S$-module.   

Getzler-Jones~\cite{GJ:Crelle93} also observed that any paracyclic module gives rise to a parachain complex in essentially the same way as a cyclic module gives rise to a mixed complex. This construction actually makes sense for any {$H$-unital para-precyclic module} (see Proposition~\ref{prop:parachain.paracomplex}). In our terminology, a \emph{para-precyclic module} is just a paracyclic module without degeneracies (see Section~\ref{sec:para-precyclic} for the precise definition).  Following~\cite{Ka:Crelle90, Wo:AM89} a precyclic module is called \emph{$H$-unital} when its bar complex is contractible. We use a similar notion for para-precyclic modules (see Section~\ref{sec:paracyclic} for the precise definition). In any case, with any $H$-unital para-precyclic module $C$ is associated a para-$S$-module $C^\natural$ as the para-$S$-module of its parachain complex. In the cyclic case we recover the total complex of Connes' $(b,B)$-bicomplex of a cyclic module~\cite{Co:CRAS83, Co:IHES85}.  

The cyclic homology of a cyclic module $C$ can be also defined by using the total chain complex $(C^\nnatural_\bt, \dl +\delta)$ of the $CC$-bicomplex 
This chain complex is turned into an $S$-module by using the shift $u^{-2}: C^\nnatural_\bt\rightarrow C^\nnatural_{\bt-2}$. This construction actually makes sense in full generality for precyclic modules. In particular, this allows us to define the cyclic homology of non-unital algebras. Given any para-precyclic module $C$, the $C^\nnatural$-construction does not produce a chain complex in general anymore. However, we observe that it still gives rise to a para-$S$-module $C^\nnatural$ (see Proposition~\ref{prop:para-precyclic.para-S-module}).  Examples of para-precyclic $k$-modules include twisted precyclic modules of non-unital $k$-algebras (see Section~\ref{sec:paracyclic-modules}). Thus, this construction allows us to associate para-$S$-modules with these twisted precyclic modules. 

\subsection*{$S$-homotopy equivalences}  Because of the relaxation of the condition $d^2=0$ into~(\ref{eq:Intro.para-S-condition}), 
the notions of  homology and quasi-isomorphisms do not make sense for general para-$S$-module. However, notions of chain maps, chain homotopies, and chain homotopy equivalences do make sense even when $d^2\neq 0$. In the setting of para-$S$-modules it is natural to require some compatibility with the $(S,T)$-operators. This leads us to define \emph{$S$-maps} and \emph{$S$-homotopies} as chain maps and chain homotopies that are compatible with the $(S,T)$-operators. 
Using  these notions we define \emph{$S$-homotopy equivalences} and \emph{$S$-deformation retracts} of para-$S$-modules (see Section~\ref{sec:Para-S-Mod}). 
Furthermore, any $S$-homotopy equivalence between para-$S$-modules gives rise to a homotopy equivalence between the corresponding periodic para-complexes (see Proposition~\ref{prop:Para-S-Mod.periodic-homotopy}). 

Chain homotopy equivalences are stronger notions of equivalence and enjoy better functorial properties than quasi-isomorphisms. Therefore, even between $S$-modules, it is quite significant to have chain homotopy equivalences. 

\subsection*{Quasi-$S$-modules} 
It is natural to look at para-$S$-modules that are equivalent to $S$-modules. We single out a sub-class of para-$S$-modules, called \emph{quasi-$S$-modules}, which admit $S$-deformation retracts to $S$-modules (see Section~\ref{sec:Para-S-Mod} for their precise definition). This leads us to define \emph{quasi-mixed complexes} and \emph{quasi-precyclic modules} as the parachain complexes and para-precyclic modules whose para-$S$-modules are quasi-$S$-modules. In fact, for these notions the sole splitting $C_\bt=\ker (1-T)\oplus (C_\bt/ \ran (1-T))$ ensures us to get quasi-$S$-modules (see Proposition~\ref{prop:parachain.quasi-mixed-quasi-S} and Proposition~\ref{prop:para-precyclic.Cnn-quasi}). In particular, we recover the quasi-cyclic modules of~\cite{KK:Crelle14, KM:HHA18} (\emph{cf}.~Section~\ref{sec:paracyclic-modules}). 

Examples of quasi-(pre)cyclic modules include the $r$-(pre)cyclic modules of Feigin-Tsygan~\cite[Appendix]{FT:LNM87a}, the twisted (pre)cyclic modules associated with periodic automorphisms and the twisted group cyclic modules associated with finite order central elements (see Section~\ref{sec:paracyclic-modules}). More generally, we obtain a quasi-mixed complex or a quasi-(pre)cyclic module as soon the $T$-operator solves a polynomial equation $Q(T)=0$ with $Q(1)=1$ (see Lemma~\ref{lem:Parachain.Q(T)-quasi} and Lemma~\ref{lem:paracyclic.Q(T)-quasi-precyclic}). 

Given any quasi-mixed complex $C=(C_\bt, b,B)$, we observe that a simple modification of the $B$-differential allows us to get a mixed complex $C=(C_\bt, b,\tilde{B})$  
whose $S$-module is $S$-homotopy equivalent to the para-$S$-module $C^\natural$ (see Proposition~\ref{prop:Parachain.tC-mixed-complex}). 
Similarly, when $C$ is a quasi-precyclic module, we can modify the construction of the para-$S$-module $C^\nnatural$ to get an $S$-module $\widetilde{C}^\nnatural$ which is $S$-homotopy equivalent to $C^\nnatural$ (Proposition~\ref{prop:para-precyclic.CC-quasi-precyclic}). In fact, this $S$-module appears of the total chain complex of a chain bicomplex, which is constructed like the $CC$-bicomplex only by modifying the $N$-operator. In particular, for $r$-(pre)cyclic modules we recover the  bicomplex of Feigin-Tsygan~\cite[Appendix]{FT:LNM87a}. 

\subsection*{Perturbation lemmas} An important tool in cyclic homology is Connes' long exact sequence, which relates cyclic homology to Hochschild homology~\cite{Co:MFO81, Co:CRAS83, Co:IHES85}. In particular, a map of mixed complexes (and more generally an $S$-map) is a quasi-isomorphism at the cyclic and periodic levels as soon as it is a quasi-isomorphism at the Hochschild level. As observed by Kassel~\cite{Ka:Crelle90}, in the setting of mixed complexes the basic perturbation lemma~\cite{Br:Messina64, EM:AM47, Gu:IJM72,  Sh:IHES62} often allows us to convert a deformation retract of Hochschild complexes into an $S$-deformation retract of cyclic complexes. In particular, a number of  well known quasi-isomorphisms in cyclic homology can be reformulated as $S$-deformation retracts of $S$-modules (see~\cite{Ba:Preprint98, Ka:Crelle90}). 

We seek for extending Kassel's approach to parachain complexes and para-(pre)cyclic modules. To this end we generalize the basic perturbation lemma to ``para-twin'' complexes (Lemma~\ref{lem:perturbation-lemma-special}). By a para-twin complex we just mean a graded $k$-module together with a pair of $k$-linear maps of degree~$-1$, which are not assumed to be differentials (see Section~\ref{sec:Perturbation} for the precise definition).

In this generalized setting the input data include a mere pair of left/right chain homotopy inverses which need not be chain maps or anti-sided inverses, whereas for the basic perturbation lemma the input data actually involve a true deformation retract provided by a pair of chain maps. However, as for the basic perturbation lemma, the chain homotopy equivalences are required to be ``special'' (see Section for the precise meaning). Unlike with the basic perturbation lemma this assumption is essential 
(\emph{cf.}\ Remark~\ref{rmk:Perturbation.special}). Nevertheless, under suitable conditions this requirement can be relaxed (see Lemma~\ref{lem:perturbation.Delta-zero}). 

Specializing these generalized perturbation lemmas to parachain complexes provides us with general recipes for converting deformation retracts of Hoschschild chain complexes into $S$-deformation retracts of para-$S$-modules (see~Lemma~\ref{lem:Perturbation.parachain-complexes}, Lemma~\ref{lem:Lifting.co-extension-deformation-retract-fBg=B}, Lemma~\ref{lem:Lifting.co-extension-deformation-retract-f-parachain}, and Lemma~\ref{lem:Lifting.co-extension-deformation-retract-f-mixed}). In particular, this provides us with a substitute for Connes' long exact sequence in the framework of parachain complexes. 

\subsection*{Comparing $C^\nnatural$ and $C^\natural$} For any  $H$-unital precyclic $k$-module $C$, we have three cyclic chain complexes: the cyclic complex $C^\lambda$, the $C^\natural$-complex of its mixed complex, and the total complex $C^\nnatural$ of the $CC$-bicomplex . The last two chain complexes are quasi-isomorphic, and, when $k\supset \Q$, they are both quasi-isomorphic to the $C^\lambda$-complex (see~\cite{Co:CRAS83, Co:IHES85, LQ:CMH84, Ts:UMN83}). By using suitable versions of the basic perturbation lemma, Kassel~\cite{Ka:Crelle90} further exhibited a deformation retract of $C^\nnatural$ to $C^\natural$ and, when $k\supset \Q$, he also constructed a deformation retract of $C^\nnatural$ to $C^\lambda$. 

We seek for extending Kassel's results to the paracyclic setting by using the generalized perturbation theory of this paper.  We start by showing that, when $C$ is an $H$-unital para-precyclic $k$-module, the para-$S$-module $C^\natural$ is an $S$-deformation retract of $C^\nnatural$ (Proposition~\ref{prop:CnnCn.deformation-retract}). In particular, as pointed out in~\cite{Ka:Crelle90}, the $B$-operator naturally re-appears from this process. In the precyclic case, we recover the deformation retract of~\cite{Ka:Crelle90}. 

\subsection*{Comparing $C^\nnatural$ and $C^\lambda$} We also compare the para-$S$-module $C^\nnatural$ to the cyclic complex $C^\lambda$ when $k\supset \Q$. The latter actually makes sense for any para-precyclic module. Given any para-precyclic $k$-module $C$ (with $k\supset \Q$), we show that $C^\lambda$ is a deformation retract of $C^\nnatural_T$, where $C_T$ is the pre-cyclic $k$-module obtained by mod-outing $C$ by $\ran(1-T)$ (see Proposition~\ref{prop:Cn-Cl.CTnn}). 

A difference with the approach of~\cite{Ka:Crelle90} lies on the use of the generalized perturbation theory of this paper. This  allows us to construct a $k$-linear map 
$\nu_0^\nnatural: C_\bt \rightarrow C_\bt^\nnatural$,  which is ``almost'' a chain homotopy inverse of the canonical projection 
$\pi_0^\nnatural: C_\bt^\nnatural \rightarrow C_\bt$. Namely, this is a chain homotopy left-inverse, and a right-inverse modulo $\ran(1-\tau)$, as well as 
a chain map modulo $\ran(1-T)$ (see Lemma~\ref{lem:Cn-Cl.Cnn-C}). It then descends to a chain map $\wb{\nu}^\nnatural: C_\bt^\lambda \rightarrow C^\nnatural_{T,\bt}$, which is a chain homotopy left-inverse and a right-inverse of the canonical chain map $\wb{\pi}^\nnatural: C^\nnatural_{T,\bt} \rightarrow C_\bt^\lambda$. 
We thus obtain a deformation retract $C^\nnatural_T$ to $C^\lambda$. Incidentally, our approach avoid using the cyclic relation $T=1$, which is used in~\cite{Ka:Crelle90}, but is not available in general in the para-precyclic setting.

When $C$ is quasi-precyclic the map $\nu_0^\nnatural$ also gives rise to a right-inverse and chain homotopy left-inverse $\nu^\nnatural: C_\bt^\lambda \rightarrow C^\nnatural_{\bt}$ of the canonical chain map $\pi^\nnatural: C^\nnatural_{\bt} \rightarrow C_\bt^\lambda$, and so get a deformation retract of $C^\nnatural$ to $C^\lambda$ (see Proposition~\ref{prop:Cn-Cl.CTnn}). 

\subsection*{Comparing $C^\natural$ and $C^\lambda$} 
When $C$ is a $H$-unital pre-paracyclic $k$-module with $k\supset \Q$, we can combine the previous comparison results to compare the para-$S$-module $C^\natural$ to the cyclic complex $C^\lambda$. More precisely, we obtain a $k$-linear map $\nu_0^\natural: C_\bt \rightarrow C_\bt^\natural$ which, in the same way as the map $\nu^\nnatural_0$ above, is an almost chain homotopy inverse of the canonical projection $\pi_0^\natural: C_\bt^\natural \rightarrow C_\bt$ (see Lemma~\ref{lem:Cn-Cl.nu0}). It then descends to a chain map $\wb{\nu}^\natural: C^\lambda \rightarrow  C^\natural_{T,\bt}$, which a right-inverse and chain homotopy left-inverse of the canonical map  $\wb{\pi}^\natural: C^\natural_{T,\bt} \rightarrow C^\lambda$. This provides us with a deformation retract of $C^\natural_T$ to $C^\lambda$ (Proposition~\ref{prop:CnCl.CTn}).  

When $C$ is quasi-cyclic the map  $\nu_0^\natural$ also gives rise to a right-inverse and chain homotopy left-inverse $\nu^\natural: C^\lambda \rightarrow  C^\natural_{\bt}$ of the canonical chain map $\pi^\natural: C^\natural_{\bt} \rightarrow C_\bt^\lambda$, and so this gives a deformation retract of $C^\natural$ to $C^\lambda$ (Proposition~\ref{prop:CnCl.Cn-quasi}). 

\subsection*{Periodicity operator} 
When $C$ is a precyclic $k$-module with  $k\supset \Q$,  Kassel~\cite{Ka:Crelle90} used the deformation retract of $C^\nnatural$ to $C^\lambda$ to get an alternative description of periodicity operator of Connes~\cite{Co:MFO81, Co:IHES85} in cyclic homology. 
We seek for extending Kassel's approach to the pre-paracyclic setting. One feature of Kassel's approach is the use of a nilpotent chain homotopy. Such a chain homotopy need not be available in the paracyclic setting (see Remark~\ref{rmk:Perturbation.special} on this point). This issue is bypassed by using the almost chain homotopy inverse $\nu_0^\nnatural$ described above. Namely, let $S_0:C_\bt \rightarrow C_{\bt-2}$ be the $k$-linear map  defined by 
\begin{equation*}
 S_0(x)= \pi_0^\nnatural \big(u^{-2} \nu_0^\nnatural (x) \big), \qquad x\in C_\bt. 
\end{equation*}
Then $S_0$ descends to a chain map  $S:C_\bt^\lambda \rightarrow C_{\bt-2}^\lambda$ that turns the cyclic complex $C^\lambda$ into an $S$-module and with respect to which the chain homotopy inverse $\wb{\nu}^\nnatural:C_\bt^\lambda \rightarrow C^\nnatural_{T,\bt}$ is an $S$-map and the canonical chain map $\pi^\nnatural: C_\bt^\nnatural \rightarrow C^\lambda_\bt$ is an $S$-map up to chain homotopy (Proposition~\ref{prop:S.para-precyclic}). There is a unique such map (\emph{loc.\ cit.}). When $C$ is quasi-precyclic it can be also shown that the chain homotopy inverse $\nu^\nnatural:C_\bt^\lambda \rightarrow C^\nnatural_{\bt}$ is an $S$-map (\emph{loc.\ cit.}).  

We also compute the $S$-operator and shows it actually agrees with Connes' periodicity operator in the precyclic case (Proposition~\ref{prop:S.formulas}). This equality occurs at the level of chains, rather than at the homology level (compare~\cite{Lo:CH, Ka:Crelle90}). In addition,  we obtain a new, and somewhat concise, formula for the $S$-operator~(\emph{loc.\ cit.}).

When $C$ is $H$-unital, it can be shown that the $S$-operator is induced on $C_\bt^\lambda$ by the operator $x\rightarrow \pi_0^\natural (u^{-1} \nu_0^\natural(x))$ (Proposition~\ref{prop:S.nun}).  The chain homotopy inverse $\wb{\nu}^\natural: C^\lambda_\bt \rightarrow  C_{T,\bt}^\natural$ is an $S$-map and the canonical chain map $\pi^\natural:C^\natural_\bt \rightarrow C^\lambda_\bt$ is an $S$-map up to chain homotopy (\emph{loc.\ cit.}). When $C$ is quasi-precyclic  (and $H$-unital) the chain homotopy inverse $\nu^\natural: C^\lambda_\bt \rightarrow  C_\bt^\natural$ is an $S$-map as well (\emph{loc.\ cit.}).

The description of the periodicity operator for $H$-unital precyclic modules above has some interesting implications in periodic cyclic homology. First, Connes~\cite{Co:IHES85} proved a formula that relates explicitly the periodicity operator to the $(b,B)$-differentials in cyclic cohomology. By using the description the $S$-operator in terms of the map $\nu_ 0^\natural$  allows us to a new proof of a dual version of Connes' formula (see Proposition~\ref{prop:S.Connes-formula}). 

Second, any $S$-map between $S$-modules that is a quasi-isomorphism gives rise to a quasi-isomorphism in periodic cyclic homology (see Section~\ref{sec:Para-S-Mod}). Thus, the very fact that the quasi-isomorphism $\wb{\nu}^\natural: C^\lambda_\bt  \rightarrow  C_{T,\bt}^\natural$ is an $S$-map ensures us it extends to a quasi-isomorphism 
$\wb{\nu}^\sharp: C^{\lambda, \sharp}_\bt \rightarrow  C_{T,\bt}^\sharp$, where $C_{T}^\sharp$ is the periodic complex of the $H$-unital precyclic module $C_T$ and $C^{\lambda, \sharp}_\bt$ is the inverse limit of the  modules $C^\lambda_\bt$ under the periodicity operator (see Proposition~\ref{prop:periodic-homology.quasi-isom}).

\subsection*{Applications in cyclic cohomology} 
By duality, the above results have interesting counterparts in cyclic cohomology and periodic cyclic cohomology.  First, we obtain an explicit way to convert a $(b,B)$-cocycle into a cohomologous cyclic cocycle (see Proposition~\ref{prop:cohomology.cyclic} and Corollary~\ref{cor:cohomology.cyclic}). 

Second, we exhibit an explicit quasi-isomorphism that realizes Connes' isomorphism between periodic cyclic cohomology and the direct limit of cyclic cohomology under the periodicity operator (Proposition~\ref{prop:periodic-cohomology.HP}). More precisely, this quasi-isomorphism allows us to invert the natural map in cohomology provided by the inclusion of cyclic cochain complex into the periodic cyclic cochain complex. This might have some applications to the representation of the Connes-Chern character by $(b.B)$-cocycles (\emph{cf}.\ Remark~\ref{rmk:S.Connes-Chern}). 
 
\subsection*{Organization of the paper} The rest of the paper is organized as follows. In Section~\ref{sec:Para-S-Mod}, we present the main definitions regarding para-$S$-modules, quasi-$S$-modules, and $S$-homotopy equivalences, as well as some of their properties.  In Section~\ref{sec:parachain}, we present the construction of the $S$-module of a parachain complex by following~\cite{GJ:Crelle93} and show that we obtain a quasi-$S$-module in the case of a quasi-mixed complex. 
In Section~\ref{sec:paracyclic-modules}, after recalling basic facts on paracyclic modules, we introduce para-precyclic modules and  quasi-precyclic modules. In Section~\ref{sec:paracyclic}, by elaborating on the considerations of~\cite{GJ:Crelle93, Ka:Crelle90} we explain how to associate a parachain complex with any $H$-unital para-precyclic module. 
In Section~\ref{sec:para-precyclic}, we construct the para-$S$-module $C^\nnatural$ of an arbitrary para-precyclic module $C$ and show that we obtain a quasi-$S$-module when $C$ is quasi-precyclic. 
In Section~\ref{sec:Perturbation}, we establish some generalizations of the basic perturbation lemma to para-twin complexes and parachain complexes. In Section~\ref{sec:Cn-Cnn}, we show that when $C$ is an $H$-unital para-precyclic module, the para-$S$-module $C^\natural$ is an $S$-deformation retract of $C^\nnatural$. In Section~\ref{sec:Cnn-Cl}, we show that, when $C$ is a para-precyclic module and the ground ring contains $\Q$, we compare the para-$S$-module $C^\nnatural$ to the cyclic complex $C^\lambda$. In Section~\ref{sec:Cn-Cl}, we show that, when $C$ is an $H$-unital para-precyclic module and the ground ring contains $\Q$, we also can compare the para-$S$-module $C^\natural$ to $C^\lambda$. In Section~\ref{sec:S}, we use these comparison results to describe Connes' periodicity operator and gives some applications in periodic cyclic homology. Finally, in Section~\ref{sec:cohomology} we describe the counterparts of these results in cyclic cohomology and periodic cyclic cohomology. 

\subsubsection*{Notation} Throughout this paper we denote by $k$ a ring with unit which is not assumed to be commutative. Unless otherwise mentioned in the last section, by a $k$-module we shall mean a \emph{left} $k$-module. 

\subsubsection*{Acknowledgements} The author would like to thank Alain Connes, Ulrich Kr\"ahmer, Henri Moscovici, Bahram Rangipour, and Christian Voigt for discussions related to the subject matter of this paper. He also thanks McGill University, University of New South Wales and University of Qu\'ebec at Montr\'eal for their hospitality during the preparation of this paper.  

\section{Para-$S$-Modules}\label{sec:Para-S-Mod} 
In this section, we introduce para-$S$-modules as the natural ``para'' version of the $S$-modules of Jones-Kassel~\cite{JK:KT89, Ka:Crelle90}. We shall also look at  chain homotopy equivalences and deformation retracts between para-$S$-modules. This will provide us with a natural notion of equivalence between para-$S$-modules. 
 
\subsection{$S$-modules}
In the terminology of Jones-Kassel~\cite{JK:KT89, Ka:Crelle90} an \emph{$S$-module} is given by chain complex of $k$-modules $(C_\bt, d)$ and a chain map $S:C_{\bt}\rightarrow C_{\bt-2}$. This notion has its roots in the seminal note of Connes~\cite{Co:CRAS83} and it naturally comes into play in bivariant cyclic homology~\cite{Co:CRAS83, CQ:JAMS95, JK:KT89, Ka:Crelle90, Ni:AM93}. Incidentally, it encapsulates various approaches to cyclic homology. In particular, examples of $S$-modules include the following:
\begin{itemize}
\item The cyclic complex of a (pre-)cyclic module of Connes~\cite{Co:MFO81, Co:CRAS83, Co:IHES85} (see also Section~\ref{sec:Cnn-Cl}). 

\item The total complex of the $(b,B)$-bicomplex of Connes~\cite{Co:MFO81, Co:CRAS83, Co:IHES85}, and more generally the cyclic complex of a mixed complex in the sense of Burghelea~\cite{Bu:CM86} and Kassel~\cite{Ka:JAlg87} (see also Section~\ref{sec:parachain}).  

\item The total complex of the $CC$-bicomplex with the operators $(b,-b',1-\tau, N)$ of Connes~\cite[Section~4]{Co:CRAS83} and Tsygan~\cite[Proposition~1]{Ts:UMN83} 
(see also Section~\ref{sec:para-precyclic}). 

\end{itemize}
We will describe generalizations for these examples in subsequent sections. Following is a further example of an $S$-module. 

\begin{example}
Given a group $\Gamma$, let us denote by $k\Gamma$ its group ring over $k$. We let $C(\Gamma)=(C_\bt(\Gamma),\partial)$ be its standard complex of $k\Gamma$-modules, where $C_m(\Gamma)=k\Gamma^{m+1}$ and the boundary $\partial: C_\bt(\Gamma)\rightarrow C_{\bt+1}(\Gamma)$ is given by 
\begin{equation}
 \partial(\gamma_0, \ldots, \gamma_m)= \sum_{0\leq j \leq m} (-1)^j (\gamma_0, \ldots, \widehat{\gamma_j}, \ldots, \gamma_m), \qquad \gamma_j\in \Gamma, 
 \label{eq:Para-S-Mod.group-differential}
\end{equation}
where $\widehat{\cdot}$ denotes omission. Let $u$ be a group 2-cocycle, i.e., a $\Gamma$-equivariant map $u: C_2(\Gamma) \rightarrow k$ such that $u\circ \partial=0$. The cap product $u\frown -: C_{\bt}(\Gamma) \rightarrow C_{\bt-2}(\Gamma)$ is given by
\begin{equation*}
 u\frown (\gamma_0,\ldots, \gamma_m)= u(\gamma_0, \gamma_1, \gamma_2)(\gamma_2, \ldots, \gamma_m), \qquad \gamma_j\in \Gamma. 
\end{equation*}
This is a chain map, and so $(C_\bt(\Gamma), \partial, u\frown -)$ is an $S$-module. This kind of $S$-module naturally appears in the description of the cyclic homology of the group ring $k\Gamma$ (see~\cite{Bu:CMH85, Ji:KT95, Ni:InvM90}). 
\end{example}

Any $S$-module $C=(C_\bt, d,S)$ gives rise to a chain complex $(C_\bt^\sharp, d)$, where $C_m^\sharp = {\varprojlim}_S C_{m+2q}$, $m\geq 0$, and ${\varprojlim}_S$ is the limit of the inverse system defined by the operators  $S:C_{m+2q+2}\rightarrow C_{m+2q}$, $q\geq 0$. Namely, 
\begin{equation}
 C_m^\sharp = \left\{ (x_{m+2q})_{q \geq 0}; \ Sx_{m+2q+2}= x_{m+2q} \ \forall q \geq 0\right\}.
  \label{eq:Para-S-Mod.Csharp2}  
\end{equation}
Moreover, as $S$ is a chain map,  the differential $d$ is compatible with the $S$-operator, and so it gives rise to a $k$-linear differential $d:C_\bt^\sharp \rightarrow C_{\bt-1}^\sharp$. 
 
We have an identification $C_{m}^\sharp \simeq C_{m+2}^\sharp$ given by the shift   
$(x_{m},x_{m+2}, \ldots)  \rightarrow  (x_{m+2},x_{m+4}, \ldots)$ and its inverse $(x_{m+2}, \ldots) \rightarrow  (Sx_{m+2}, x_{m+2}, \ldots)$. Thus, if we set $m=2p+i$ with $i\in \{0,1\}$, then we have a natural identification $C_m^\sharp \simeq C_i^\sharp$. Under this identification the chain complex $(C_\bt^\sharp, d)$ is $2$-periodic and is naturally identified with the $\Z_2$-graded chain complex $C^\sharp:=(C_0^\sharp\oplus C_1^\sharp, d)$, where $d$ is regarded as an \emph{odd}  $k$-linear operator on $C_0^\sharp\oplus C_1^\sharp$, i.e., it maps $C_0^\sharp$ (resp., $C^\sharp_1$) to $C_1^\sharp$ (resp., $C^\sharp_0$). (This uses the identification $C_2^\sharp\simeq C^\sharp_0$ described above.) 

We call the $\Z_2$-graded complex $C^\sharp$ the \emph{periodic chain complex} of the $S$-module $C$.  If we denote by $H_\bt(C)$ (resp., $H_\bt(C^\sharp)$) the homology of $C$ (resp., $C^\sharp$), then the operator $S$ induces a $k$-linear map $S:H_\bt(C)\rightarrow H_{\bt-2}(C)$, and we have an exact sequence, 
\begin{equation}
 0 \longrightarrow {\varprojlim}^1_{S} H_\bt(C) \longrightarrow H_\bt(C^\sharp)  \longrightarrow {\varprojlim}_{S} H_\bt(C) \longrightarrow 0,  
 \label{eq:Para-S-Mod.exact-sequence-periodicity} 
\end{equation}
where ${\varprojlim}^1$ is the first derived functor of the inverse limit functor $\varprojlim$ (see, e.g., \cite{Lo:CH, We:IHA}). 

In the terminology of Kassel~\cite{Ka:JAlg87, Ka:Crelle90}, given $S$-modules $C=(C_\bt, d, S)$ and $\wbC=(\wbC_\bt, d,S)$, an \emph{$S$-map} is any chain map $f:C_\bt \rightarrow \wbC_{\bt}$ that is compatible with the $S$-operators. Any $S$-map $f:C_\bt \rightarrow \wbC_{\bt}$ gives rise to a chain map $f^\sharp: C_\bt^\sharp \rightarrow \wbC_\bt^\sharp$ between the corresponding periodic complexes. More generally, any $S$-compatible $k$-linear map $f:C_\bt \rightarrow \wbC_{\bt+d}$ of degree $d$, $d\in \Z$, gives rise to $k$-linear map between the periodic complexes which is even or odd according to the parity of $d$. By using the functoriality of the exact sequence~(\ref{eq:Para-S-Mod.exact-sequence-periodicity}) and the 5-lemma we obtain the following result. 

\begin{proposition}\label{prop:Para-S-Mod.quasi-isomorphism-periodic}
 Let $f:C_\bt\rightarrow \wb{C}_\bt$ be an $S$-map and a quasi-isomorphism. Then the chain map $f^\sharp: C_\bt^\sharp \rightarrow \wbC_\bt^\sharp$ is a quasi-isomorphism as well. 
\end{proposition}

\subsection{Para $S$-modules} 
The paracyclic category of Feigin-Tsygan~\cite{FT:LNM87a} and Getzler-Jones~\cite{GJ:Crelle93} (see also~\cite{DK:CMH85, FL:TAMS92}) appears in various computations of cyclic homology groups, including the cyclic homology of group rings and crossed-product algebras associated with group actions on $k$-algebras (see, e.g., \cite{Bu:CMH85, Cr:KT99, FT:LNM87b, GJ:Crelle93, Ma:BCP86, Po:CRAS17a}). Getzler-Jones also introduced parachain complexes as the ``para'' version of mixed complexes (see also~\cite{DK:CMH85}). In this paper, we shall use the following ``para'' generalization of $S$-modules. 

\begin{definition}\label{def:Para-S-Module}
 A \emph{para-$S$-module} is given by  a system $(C_\bt, d, S,T)$, where $C_m$, $m\geq 0$, are $k$-modules, $d:C_\bt \rightarrow C_{\bt-1}$ and $S:C_\bt\rightarrow C_{\bt-2}$ are $k$-linear maps, and $T:C_\bt \rightarrow C_\bt$ is a $k$-linear isomorphism in such a way that
\begin{equation}
 d^2=S(1-T)  \qquad \text{and} \qquad [d,S]=[d,T]=[S,T]=0. 
 \label{eq:Para-S-modules-relations}
\end{equation}
\end{definition}

\begin{remark}
When $T=1$ we recover the definition of an $S$-module. In particular, any $S$-module is a para-$S$-module. 
\end{remark}

\begin{remark}
As we shall see later, para-$S$-modules provides us with a natural setting for the following constructions: 
\begin{itemize}
\item A version for parachain complexes of the cyclic complex of a mixed complex. This construction goes back to Getzler-Jones~\cite{GJ:Crelle93} (see Section~\ref{sec:parachain}). 

\item A version for para-precyclic $k$-modules of (the total complex of) the $CC$-bicomplex (see~Section~\ref{sec:para-precyclic}). 
\end{itemize}
As we shall see, although these constructions do not provide us with chain complexes, they do give rise to para-$S$-modules.  
\end{remark}

Given para-$S$-modules $C=(C_\bt, d,S,T)$ and $\wbC=(\wbC_\bt, d,S,T)$, we shall say that a $k$-linear map $f:C_\bt \rightarrow \wbC_\bt$ is a \emph{chain map} when it is compatible with the $d$-operators. We shall say that this is a para-$S$-module map, or simply an \emph{$S$-map}, when it is further compatible with the operators $(S,T)$. When $T=1$ we recover the usual notion of an $S$-map between $S$-modules~(\cite{JK:KT89, Ka:JAlg87}). 

As $(C_\bt, d)$ and $(\wbC_\bt, d)$ need not be chain complexes, we cannot speak about quasi-isomorphisms between them. However, notions of chain homotopy equivalences and deformation retracts do make sense in the same way as with usual chain  complexes. More precisely, we shall say that two chain maps $f_2:C_\bt \rightarrow \wbC_\bt$ and $f_2:C_\bt \rightarrow \wbC_\bt$ are \emph{chain homotopic} when there is a $k$-linear map $\varphi: C_\bt \rightarrow \wbC_{\bt+1}$ such that  
\begin{equation}
 f_1-f_2=d\varphi + \varphi d. 
 \label{eq:Para-S-Modules.homotopy}
\end{equation}
Given chain maps $f:C_\bt \rightarrow \wbC_\bt$ and $g:\wbC_\bt \rightarrow C_\bt$, we then say that $g$ is a homotopy inverse of $f$ when $gf$ and $fg$ are chain homotopy equivalent to the identity maps of $C_\bt$ and $\wbC_\bt$. We similarly define notions of chain homotopy left-inverse and right-inverse. 

\begin{definition}
 A \emph{chain homotopy equivalence} between $C$ and $\wbC$ is given by chain maps $f:C_\bt \rightarrow \wbC_\bt$ and $g:\wbC_\bt \rightarrow C_\bt$ 
 that are chain homotopy inverses of each other. We have a \emph{deformation retract} of $C$ to $\wbC$ when we can choose $f$ and $g$ so that $fg=1$. 
\end{definition}

In what follows we will be interested in chain homotopies that are compatible with the $(S,T)$-operators. We shall say that $S$-maps $f:C_\bt \rightarrow \wbC_\bt$ and $f':C_\bt \rightarrow \wbC_\bt$ are \emph{$S$-homotopic} when they are chain homotopic and the homotopy~(\ref{eq:Para-S-Modules.homotopy}) can be realized by means of a chain homotopy $\varphi: C_\bt \rightarrow \wbC_{\bt+1}$ that is compatible with the $(S,T)$-operators. This allows us to give sense to notions of $S$-homotopy inverses of $S$-maps. 

\begin{definition}
An \emph{$S$-homotopy equivalence} between $C$ and $\wbC$ is given by chain maps $f:C_\bt \rightarrow \wbC_\bt$ and $g:\wbC_\bt \rightarrow C_\bt$ 
 that are $S$-homotopy inverses of each other. We have an \emph{$S$-deformation retract} of $C$ to $\wbC$ when we can choose $f$ and $g$ so that $fg=1$. 
\end{definition}

\subsection{The Periodic para-complex of a para-$S$-module} 
Let $C=(C_\bt, d,S,T)$ be a para-$S$-module. In the same way as with $S$-modules, the operator $S$ allows us to form a system $(C_0^\sharp \oplus C_1^\sharp, d)$, 
where the $\Z_2$-graded $k$-module $C_0^\sharp \oplus C_1^\sharp$ is defined as in~(\ref{eq:Para-S-Mod.Csharp2}) and the $k$-linear map $d:C_\bt^\sharp \rightarrow C_{\bt-1}^\sharp$ is induced from the ``para-differential'' $d:C_\bt \rightarrow C_{\bt-1}$. The compatibility of $T$ with $S$ also implies that it gives rise to an invertible (even) $k$-linear map $T:C_\bt^\sharp \rightarrow C_\bt^\sharp$. Moreover, given any sequence $(x_0,x_2, \ldots)\in C_0^\sharp$, we have 
\begin{equation*}
 d^2(x_0,x_2, \ldots)=d(dx_2,dx_4, \ldots)= (d^2x_2,d^2x_4, \ldots). 
\end{equation*}
As $d^2x_{2q+2}=(1-T)Sx_{2q+2}=(1-T)x_{2q}$, we see that $d^2=(1-T)$ on $C_0^\sharp$. Likewise, $d^2=(1-T)$ on $C_1^\sharp$. As $T:C_\bt^\sharp \rightarrow C_\bt^\sharp$ is invertible, this shows that the triple 
$(C_\bt^\sharp , d, T)$ is a para-complex in the sense of Voigt~\cite{Vo:JIMJ07}. 

\begin{definition}
 The para-complex $C^\sharp:=(C_\bt^\sharp, d, T)$ is called the \emph{periodic para-complex} of the para-$S$-module 
 $C$. 
\end{definition}

Given para-$S$-modules $C=(C_\bt, d,S,T)$ and $\wb{C}=(\wb{C}_\bt, d, S, T)$, any $S$-map $f:C_\bt \rightarrow \wb{C}_\bt$ gives rise to a $T$-compatible chain map $f^\sharp: C_\bt^\sharp \rightarrow \wb{C}^\sharp_\bt$.  More generally, any $k$-linear map $f:C_\bt \rightarrow \wb{C}_{\bt+d}$, $d \in \Z$, which is compatible with the operators $S$ and $T$ gives to a $T$-compatible $k$-linear map $f^\sharp$ between periodic chains. We get an even or odd map according to the parity of $d$. 
 
 In what follows we shall call \emph{$T$-homotopy} any chain homotopy between para-complexes that is realized by means of $T$-compatible $k$-linear maps. 
Any $S$-homotopy between $C$ and $\wb{C}$ gives rise a $T$-homotopy between $C_\bt^\sharp$ and $\wb{C}_{\bt}^\sharp$. Therefore, any $S$-homotopy equivalence between paras-$S$-modules gives rise to a $T$-homotopy equivalence between the corresponding periodic para-complexes. More precisely, we have the following result. 
 
\begin{proposition}\label{prop:Para-S-Mod.periodic-homotopy}
Let $f:C_\bt \rightarrow \wb{C}_ \bt$ and $g:\wb{C}\rightarrow C_\bt$ be $S$-maps such that
\begin{equation*}
 fg=1+d\psi +\psi d, \qquad gf=1+d\varphi +\varphi d, 
\end{equation*}
where $\varphi: C_\bt \rightarrow C_{\bt+1}$ and  $\psi: \wb{C}_\bt \rightarrow \wb{C}_{\bt+1}$ are $T$-compatible $k$-linear maps. 
Then,  we have 
\begin{equation*}
 f^\sharp g^\sharp=1+d\psi^\sharp +\psi^\sharp d, \qquad g^\sharp f^\sharp=1+d\varphi^\sharp +\varphi^\sharp d. 
\end{equation*}
In particular, this gives a $T$-homopotopy equivalence between $C^\sharp$ and $\wb{C}^\sharp$. We get a deformation retract when $fg=1$. 
\end{proposition}
 
\begin{remark}
 By Proposition~\ref{prop:Para-S-Mod.quasi-isomorphism-periodic}, when $C$ and $\wb{C}$ are both $S$-modules, any quasi-isomorphism provided by an $S$-map $f:C_\bt \rightarrow \wb{C}_\bt$ gives rise to quasi-isomorphism $f^\sharp: C_\bt^\sharp \rightarrow \wb{C}_\bt^\sharp$. Proposition~\ref{prop:Para-S-Mod.periodic-homotopy} provides us with some substitute for this result in the framework of para-$S$-modules. 
\end{remark}

\subsection{Property (DR) and quasi-$S$-modules}  Given any para-$S$-module $C=(C_\bt, d,S,T)$ we set 
\begin{equation}
 C^T_\bt=\ker(1-T), \qquad R^T_\bt=\ran(1-T), \qquad C_{T,\bt}=C_\bt\slash R^T_\bt. 
 \label{eq:Para-S-Mod.CTRT}
\end{equation}
The compatibility of $T$ with the operators $d$ and $S$ implies that these operators induce operators on $C_\bt^T$ and $R^T_\bt$, so that we obtain para-$S$-modules $C^T:=(C^T_\bt, d,S,T)$ and $R^T:=(R^T_\bt, d,S,T)$. In fact, as $T=1$ on $C^T_\bt$, we see that $C^T$ is actually an $S$-module. 
 
The compatibility of $T$ with the operators $d$ and $S$ also implies that these operators descend to operators on $C_{T,\bt}$. As $T=1$ on $C_{T,\bt}$ we obtain another $S$-module $C_T:=(C_{T,\bt}, d,S)$. 
Note that the canonical projection $\pi_T: C_\bt \rightarrow C_{T,\bt}$ is an $S$-map.

We further observe that, given para-$S$-modules $C=(C_\bt, d,S,T)$ and $\wbC=(\wbC_\bt, d,S,T)$, any $S$-map $f:C_\bt \rightarrow \wbC_\bt$ induces an $S$-map $f:C_\bt^T \rightarrow \wbC_\bt^T$ and it descends to a unique $S$-map $f:C_{T,\bt} \rightarrow \wbC_{T,\bt}$. More generally, any $S$-homotopy equivalence between $C$ and $\wbC$ gives rise to $S$-homotopy equivalences between the $S$-modules $C^T$ and $\wbC^T$ and between $C_T$ and $\wbC_T$. 

We are especially interested in para-$S$-modules with the following property. 

\begin{definition}
 A para-$S$-module $C$ has property (DR) when the canonical projection $\pi_T: C_\bt \rightarrow C_{T,\bt}$ gives rise to an $S$-deformation retract (i.e., it has a right-inverse which an $S$-homotopy left-inverse).
 \end{definition}

In what follows, given any $x\in C_\bt$, it will be convenient to set $\wb{x}=\pi_T(x)$.  We shall also say that a para-$S$-module  is \emph{$S$-contractible} when the identity map is $S$-homotopic to $0$. 

We have the following characterization of property (DR).

\begin{proposition}\label{eq:Para-S-Module.Property-DR}
 Let $C=(C_\bt, d, S,T)$ be a para-$S$-module. Then $C$ has property (DR) if and only the following two conditions are satisfied:
 \begin{enumerate}
    \item[(i)] The  para-$S$-module $R^T$ is $S$-contractible. 
    
    \item[(ii)] There is a sub-$S$-module $C'\subset C^T$ such that $C_\bt=C'_\bt\oplus R^T_\bt$. 
\end{enumerate}
\end{proposition}
\begin{proof}
 Suppose that $C$ has property (DR). This means there are a chain map $\iota:C_{T,\bt}\rightarrow C_\bt$ and an $(S,T)$-compatible $k$-linear map $h:C_\bt \rightarrow C_{\bt+1}$ such that $\pi_T \iota=1$ and $\iota \pi_T=1+dh+hd$.  The $(S,T)$-compatibility of $h$ implies that it induces an $(S,T)$-compatible $k$-linear map $h:R^T_\bt \rightarrow R^T_{\bt+1}$. Moreover, the equality $1+dh+dh=\iota \pi_T$ implies that $1+dh+dh=0$ on $R^T_\bt=\ker \pi_T$. Thus, the identity map on $R^T_\bt$ is $S$-homotopic to 0, i.e., $R^T$ is $S$-contractible.
 
 Set $\pi =\iota \pi_T$ and $C_\bt'=\ran \pi$. The $T$-compatibility of $\iota$ means that $T\iota=\iota$, and so $T\pi=\pi$. This means that $C'_\bt=\ran \pi$ is contained in $C^T_\bt$. In addition, the operators $d$ and $S$ acts on $C'_\bt$. Thus, we get a sub-$S$-module $C'=(C'_\bt, d, S, T)$ of $C^T$. 
 As $\pi_T \iota=1$, we have $\pi^2=\iota(\pi \iota)\pi_T=\pi$, and so $\pi$ is a projection. This implies that $C_\bt=C_\bt'\oplus \ker \pi$. As $\pi =\iota \pi_T$ it is immediate that $\ker \pi$ contains $\ker \pi_T$. Moreover, as $\pi_T\pi=(\pi_T\iota \pi) \pi_T=\pi_T$, we also see that $\ker \pi\subset \ker \pi_T$. Thus, $\ker \pi$ agrees with $\ker \pi_T=R^T_\bt$, and hence $C_\bt=C_\bt'\oplus R^T_\bt$. 
 
Conversely, suppose that the condition (i) and (ii) are satisfied. The latter condition means there is a a sub-$S$-module $C'\subset C^T$ such that $C_\bt=C'_\bt\oplus R^T_\bt$. Let $\pi':C_\bt \rightarrow C_\bt$ be the projection onto $C'_\bt$ associated with this projection. This is an $S$-map, since $C_\bt=C'_\bt\oplus R^T_\bt$ is a splitting of para-$S$-modules. This splitting also implies that the canonical projection $\pi_T:C_\bt \rightarrow C_{T,\bt}$ induces a $k$-module isomorphism $\pi_T': C_\bt '\rightarrow C_{T,\bt}$. This is an $S$-map, and so we get an $S$-module isomorphism from $C'$ onto $C_T$. Under this isomorphism the inclusion of $C'_\bt$ to $C_\bt$ corresponds to the $S$-map $\iota_T':C_{T,\bt} \rightarrow C_\bt$ given by
\begin{equation*}
 \iota_T'( \wb{x}) =  \iota_T'\big( \overline{\pi'(x)}\big) = \pi'(x), \qquad x\in C_\bt. 
\end{equation*}
This means that $\iota_T'\pi_T=\pi'$. This gives $\pi_T \iota_T'\pi_T=\pi_T\pi'=\pi_T$, and hence $\pi_T \iota_T'=1$ on $C_{T,\bt}$. 

By assumption $R^T$ is $S$-contractible, and so there is an $(S,T)$-compatible $k$-linear map $\beta:R^T_\bt\rightarrow R^T_{\bt+1}$ such that $1=d\beta+\beta d$. Let $h:C_{\bt}\rightarrow C_{\bt+1}$ be the $k$-linear map defined by
\begin{equation}
 h(x) =-\beta\left(x-\pi'(x)\right), \qquad x \in C_\bt. 
 \label{eq:Para-S-Module.homotopy-h}
\end{equation}
This map is compatible with the operators $S$ and $T$. Moreover, as $\pi'$ is an $S$-map, given any $x\in C_\bt$, we see that $dh(x)+h(dx)$ is equal to
\begin{equation*}
 -d\beta\left(x-\pi'(x)\right)- \beta \left(dx-\pi'(dx)\right)=-(d\beta+\beta d)\left(x-\pi'(x)\right)= -\left(x-\pi'(x)\right). 
\end{equation*}
It then follows that $1+dh+hd=\pi'=\iota_T'\pi_T$. Together with the equality $\pi_T\iota_T'=1$ above this shows that we have an $S$-deformation retract of $C$ to $C_T$. The proof is complete. 
\end{proof}

\begin{remark}\label{rmk:Para-S-Module.homotopy-special-C'} 
As we can see from~(\ref{eq:Para-S-Module.homotopy-h}), the range of $h$ is contained in $R^T_\bt=\ker \pi_T$ and its kernel contains on $C_\bt'=\ran \iota_T'$. It then follows that $\pi_Th=0$ and $h\iota_T'=0$. 
\end{remark}

Proposition~\ref{eq:Para-S-Module.Property-DR} is the main impetus for introducing the following class of para-$S$-modules. 

\begin{definition}\label{def:Para-S.quasi-S-mod}
 A \emph{quasi-$S$-module} is a para-$S$-module $C=(C^T_\bt, d,S)$ such that $R^T$ is $S$-contractible and we have the $k$-module splitting, 
\begin{equation}
 C_\bt = C^T_\bt \oplus R^T_\bt. 
 \label{eq:Para-S-Mod.quasi-splitting}
\end{equation}
\end{definition}

\begin{remark}
 As we shall see, in many instances the $S$-contractibility of $R^T$ is ensured by the sole splitting~(\ref{eq:Para-S-Mod.quasi-splitting}). 
\end{remark}

When it occurs the splitting~(\ref{eq:Para-S-Mod.quasi-splitting}) is a splitting of para-$S$-modules. In such a case we shall let $\pi^T : C_\bt \rightarrow C_\bt$ be the projection on $C_T$ associated with this splitting, and, given any $x\in C_\bt$, we set $x^T=\pi^T(x)$. 

\begin{lemma}\label{lem:Para-S-Mod.Q-module}
 Suppose there is a polynomial $Q(X)\in k[X]$ such that $Q(T)(T-1)=0$ and $Q(1)=1$. Then $C$ is a quasi-$S$-module and we have $\pi^T=Q(T)$. 
\end{lemma}
\begin{proof}
The assumption that $Q(T)(1-T)=(1-T)Q(T)=0$ ensures us that $R^T_\bt\subset \ker Q(T)$ and $\ran Q(T)\subset C^T_\bt$. The fact that $Q(T)T=Q(T)$ implies that $Q(T)T^j =Q(T)$ for all $j\geq 0$, and so $Q(T)P(T)=Q(T)P(T)=P(1)Q(T)$ for every polynomial $P(X)\in k[X]$.  In particular, taking $P(X)=Q(X)$ gives $Q(T)^2=Q(1)Q(T)=Q(T)$. Thus, $Q(T)$ is an idempotent, and so we have the direct sum of $k$-modules,    
\begin{equation*}
 C_\bt=\ran Q(T) \oplus \ker Q(T). 
\end{equation*}
 
To complete the proof we just need to show that $\ker Q(T)\subset R^T_\bt$ and $C^T_\bt\subset \ran Q(T)$. As $Q(1)=1$ there is $\tilde{Q}(X)\in k[X]$ such that $Q(X)-1=(X-1)\tilde{Q}(X)$. If $x\in \ker Q(T)$, then we have $x=(1-Q(T))x=(1-T)\tilde{Q}(T)x\in \ran (1-T)$. Thus, $\ker Q(T)\subset R^T_\bt$. In addition, if $x\in C^T_\bt$, then $Q(T)x-x=\tilde{Q}(T)(T-1)x=0$, and so $x=Q(T)x\in \ran Q(T)$. This shows that $C^T_\bt \subset \ran Q(T)$. The proof is complete.   
 \end{proof}

\begin{example}\label{ex:Para-S-Mod.Tr=1}
 Suppose that $T^r=1$ with $r\geq 1$ such that $r^{-1}\in k$ (this happens for instance when $k\supset \Q$). Then Lemma~\ref{lem:Para-S-Mod.Q-module} holds with $Q(X)=r^{-1}(1+X+\cdots + X^{r-1})$. In this case, we have
 \begin{equation*}
 x^T= \frac{1}{r}(x+Tx+\cdots +T^{r-1}x) \qquad \forall x\in C_\bt. 
\end{equation*}
\end{example}

Let $C=(C_\bt,d,S,T)$ be a quasi-$S$-module. As the conditions~(i)--(ii) of Proposition~\ref{eq:Para-S-Module.Property-DR} are satisfied we know that $C$ has property (DR). Moreover, the proof of Proposition~\ref{eq:Para-S-Module.Property-DR} above provides us with an explicit $S$-deformation retract of $C$ to $C^T$ and $C_T$ by using taking the projection $\pi'$ to be $\pi^T$. More precisely, let $\iota_T:C_{T,\bt}\rightarrow C_{\bt}$ be the $k$-linear map given by
\begin{equation}
  \iota_T(\wb{x})=x^T, \qquad x\in C_\bt. 
  \label{eq:Para-S-Mod.iotaT}
\end{equation}
We also let $\beta:R^T_\bt\rightarrow R^T_{\bt+1}$ be an $(S,T)$-compatible contracting homotopy, i.e., $1=d\beta+\beta d$ on $R^T_\bt$. We  define the $k$-linear map 
$h:C_\bt \rightarrow C_{\bt+1}$ defined by
\begin{equation}
 h(x)=-\beta\left(x-x^T\right), \qquad x\in C_\bt. 
  \label{eq:Para-S-Mod.hT}
\end{equation}
Proposition~\ref{prop:Para-S-Mod.periodic-homotopy} and the second part of the proof of Proposition~\ref{eq:Para-S-Module.Property-DR} then give the following result. 

\begin{proposition}\label{prop:Para-S-Mod.quasi-S-Mod}
 Suppose that $(C_\bt,d,S,T)$ is a quasi-$S$-module. 
 \begin{enumerate}
 \item The canonical projection $\pi^T:C_\bt \rightarrow C_{T,\bt}$ induces an $S$-module isomorphism from $C^T$ onto $C_T$.
 
 \item The map $\iota_T:C_{T,\bt}\rightarrow C_{\bt}$ is an $S$-map such that
 \begin{equation}
 \pi_T \iota_T=1, \qquad \pi^T=\iota_T\pi_T=1+dh +hd.
 \label{eq:Para-S-Module.Quasi-S-DR} 
\end{equation}
 This provides us with $S$-deformation retracts of $C$ to $C_T$ and $C^T$ and $T$-deformation retracts of $C^\sharp$ to $C_T^\sharp$ and $C^{T,\sharp}$.  
 \end{enumerate}
\end{proposition}

\begin{remark}
 The $S$-deformation retract of $C^\natural$ to $C^{T,\natural}$ is simply given by the inclusion $\iota^T: C^T_\bt \hookrightarrow C_\bt$ and the projection $x\rightarrow x^T$ seen as an $S$-map $\tilde{\pi}^T: C_\bt \rightarrow C^T_\bt$. There are the $S$-maps $\iota_T$ and $\pi_T$ under the isomorphism $C^T_\bt\simeq C_{T,\bt}$. In fact, it is immediate that $\tilde{\pi}^T\iota^T=1$ on $C^T_\bt$ and by  using~(\ref{eq:Para-S-Module.Quasi-S-DR}) we get $\iota^T\tilde{\pi}^T=\pi^T=1+dh +hd$. 
\end{remark}

\begin{remark}
As in Remark~\ref{rmk:Para-S-Module.homotopy-special-C'} the chain homotopy $h$ is such that $\pi_Th=0$ and $\iota_Th=0$.  
\end{remark}

\section{The Para-$S$-Module of a Parachain Complex}\label{sec:parachain}
In this section, we review the construction of the para-$S$-module of a parachain complex due to Getzler-Jones~\cite{GJ:Crelle93}. Although, Getzler-Jones didn't  not consider para-$S$-modules in their article, their construction actually yields a para-$S$-module. We shall also introduce quasi-mixed complexes as the parachain complexes that give rise to quasi-$S$-modules. 

\subsection{The Para-$S$-module $C^\natural$} 
Let $C=(C_\bt, b, B)$ be a parachain complex of $k$-modules in the sense of Getzler-Jones~\cite{GJ:Crelle93} (see also~\cite{DK:CMH85}). 
Recall this means that $(C_\bt, b)$ is a chain complex of $k$-modules and we have an extra degree~1 differential $B:C_\bt \rightarrow C_{\bt+1}$, $B^2=0$,  such that $T:=1-(bB+Bb)$ is a $k$-linear isomorphism. When $T=1$ we recover the definition of a mixed complex~\cite{Bu:CM86, Ka:JAlg87}. Even when $T\neq 1$, the map $T$ commutes with the differentials $(b,B)$, and so this is an automorphism of the parachain complex  $C$. 

A number of parachain complexes arise from paracylic modules, and more generally $H$-unital para-precyclic modules  (\emph{cf}.~Section~\ref{sec:paracyclic}). An example of a different kind of parachain complex is the parachain complex of equivariant noncommutative differential forms of a (non-unital) locally convex algebra acted on by a locally compact group  (see Voigt~\cite{Vo:JIMJ07}). 

As it follows from~\cite{GJ:Crelle93} the datum of a parachain complex $C=(C_\bt, b, B)$ allows us to construct a para-$S$-module as follows. 

Throughout this paper we let $u$ be an indeterminate variable and denote by $k[u]$ the polynomial algebra over $k$ that it generates. We shall regard the tensor  product $ k[u] \otimes C_\bt$ as a $k[u]$-module. Given any $m,p\geq 0$ and $x\in C_m$,  we denote by $xu^p$ the tensor product $u^p\otimes x$. We then have
\begin{equation*}
 C_m\otimes k[u]= \bigoplus_{p\geq 0} C_m u^p= C_m \oplus C_m u \oplus  C_m u^2 \oplus \cdots, 
\end{equation*}
where we have set $C_m u^p=\{xu^p; x\in C_m\}$. In addition, we let $u^{-1}: C_\bt \otimes k[u] \rightarrow C_\bt \otimes k[u]$ be the $k$-linear map defined by 
\begin{equation*}
 u^{-1} (xu^0)=0, \qquad u^{-1}(xu^{p})=xu^{p-1}, \quad p \geq 1. 
\end{equation*}
We also denote by $u^{-j}$, $j\geq 2$, the operator $(u^{-1})^{j}$. Note that $\ker u^{-j}= C_\bt u^0 \oplus \cdots \oplus C_\bt u^{j-1}$. 

The para-$S$-module associated with $(C_\bt, b,B)$ is $C^\natural:=(C^\natural_\bt, b+Bu^{-1}, u^{-1}, T)$, where the $k$-module of $m$-chains $C_m^\natural$, $m\geq 0$, is defined by
\begin{equation}
 C^\natural_m = \bigoplus_{2p+q=m} C_q u^p = C_m u^0 \oplus C_{m-2}u \oplus \cdots. 
 \label{eq:PSP.Cnat}
\end{equation}
Here $u^{-1}$ maps $C^\natural_m$ to $ C_{m-2} \oplus C_{m-4}u \oplus \cdots = C_{m-2}^\natural$. Moreover, we have 
\begin{equation*}
 (b+Bu^{-1})^2= (bB+Bb)u^{-1}=(1-T)u^{-1}.
\end{equation*}
In particular, when $C$ is a mixed complex, i.e., $T=1$, we recover the usual cyclic complex of a mixed complex~\cite{Bu:CM86, Co:MFO81, Co:CRAS83, Co:IHES85, Ka:JAlg87}.  

As the $S$-operator $u^{-1}$ simply shifts $(x_m, x_{m-2}, \ldots)$ to $(x_{m-2}, x_{m-4}, \ldots)$, the inverse limit~(\ref{eq:Para-S-Mod.Csharp2}) is naturally identified with the direct product, 
\begin{equation}
 C_i^\sharp = \prod_{q\geq 0} C_{2q+i}, \qquad i=0,1. 
 \label{eq:Parachain.periodic}
\end{equation}
Namely, the identification between $\varprojlim_{u^{-1}}C_{i+2q}^\sharp$ and $\prod_{q\geq 0} C_{2q+i}$ is given by $(x_{i+2q})\rightarrow (\pi_0^\natural (x_{i+2q}))$, where $\pi_0^\natural:C_\bt^\natural \rightarrow C_\bt$ is the projection onto the zeroth degree summand $C_\bt u^0=C_\bt$. Under this identification, the differential of $C^\sharp$ is just $b+B$. Note that $(b+B)^2=bB+Bb=1-T$. 

\begin{definition}
The para-complex $C^\sharp:=(C_\bt^\sharp, b+B, T)$  is called the \emph{periodic para-complex} of the parachain complex $C$. 
\end{definition}

Given parachain complexes $(C_\bt, b,B)$ and $(\wbC_\bt, b,B)$, an $S$-map $f: C_\bt^\natural \rightarrow \wbC_\bt^\natural$ is a chain map of the form, 
\begin{equation*}
 f = f^{(0)} +f^{(1)}u^{-1}+f^{(2)}u^{-2}+\cdots, 
\end{equation*}
where the sum is finite and the $k$-linear maps $f^{(j)}:C_{\bt}\rightarrow \wbC_{\bt+2j}$, $j\geq 0$, are compatible with the $T$-operator. The chain map condition means that 
\begin{equation}
 bf^{(0)} = f^{(0)}b, \qquad b f^{(j)} + Bf^{(j-1)} = f^{(j)}b + f^{(j-1)}B, \quad j\geq 1. 
 \label{eq:Parachain.S-map}
\end{equation}
For instance, any map of parachain complexes $f:C_\bt \rightarrow \wbC_\bt$ is an $S$-map with $f^{(0)}=f$ and $f^{(j)}=0$ for $j\geq 1$. When $T=1$ we recover the usual notion of $S$-map between cyclic complexes of mixed complexes in the sense of~\cite{Ka:JAlg87}. 

Note also that~(\ref{eq:Parachain.S-map}) implies that $f^{(0)}$ is an ordinary chain map from $(C_\bt, b)$ to $(\wb{C}_\bt, b)$. In addition, the corresponding chain map $f^\sharp: C_\bt^\sharp \rightarrow  \wb{C}_\bt^\sharp$ is given by
\begin{equation*}
 f^\sharp \left( (x_{2q+i})_{q\geq 0}\right) = \big( f^{(0)}(x_{2q+i})+f^{(1)}(x_{2q-2+i})+\cdots + f^{(q)}(x_i)\big)_{q\geq 0}. 
\end{equation*}

\subsection{Quasi-mixed complexes} 
We shall now look at parachain complexes $C$ for which the para-$S$-modules $C^\natural$ are quasi-$S$-modules. Let $C=(C_\bt, b,B)$ be a parachain complex of $k$-modules. As in~(\ref{eq:Para-S-Mod.CTRT}) we form the $k$-modules $C^T_\bt=\ker(1-T)$, $R^T_\bt=\ran(1-T)$ and $C_{T,\bt}=C_\bt \slash R^T_\bt$. As the differentials $b$ and $B$ commute with the operator $T=1-(bB+bB)$, they induce differentials on $C^T_\bt$ and $R^T_\bt$ and descend to differentials on $C_{T,\bt}$. We thus get parachain complexes $C^T:=(C^T_\bt, b,B)$, $R^T:=(R^T_\bt, b,B)$, and $C_T:=(C_{T,\bt},b,B)$. Furthermore, as $bB+bB=1-T=0$ on $C^T_\bt$ and $C_{T,\bt}$, we see that $C^T$ and $C_T$ are actually mixed complexes. 

\begin{definition}
A parachain complex $C=(C_\bt,b,B)$ such that $C_\bt = C^T_\bt \oplus R^T_\bt$ is called a  \emph{quasi-mixed complex}.  
\end{definition}

By using Lemma~\ref{lem:Para-S-Mod.Q-module} we obtain the following criterion.  

\begin{lemma}\label{lem:Parachain.Q(T)-quasi}
Suppose there is $Q(X)\in k[X]$ such that $Q(T)(T-1)=0$ and $Q(1)=1$. Then $C$ is a quasi-mixed complex.
\end{lemma}

\begin{example}
In the same way as in Example~\ref{ex:Para-S-Mod.Tr=1}, if $T^r=1$ for some $r\geq 1$ such that $r^{-1}\in k$, then the assumption of Lemma~\ref{lem:Parachain.Q(T)-quasi} is satisfied with $Q(X)=r^{-1}(1+X+\cdots +X^{r-1})$. Therefore, in this case $C$ is a quasi-mixed complex. Examples of such quasi-mixed complexes are provided by the parachain complexes associated with $r$-cyclic modules (see Example~\ref{ex:Paracylic.r-cyclic} below). 
\end{example}

Assume that $C$ is a quasi-mixed complex. We note that the splitting $C_\bt = C^T_\bt \oplus R^T_\bt$ is actually a splitting of parachain complexes, and so we get a splitting of para-$S$-modules, 
\begin{equation}
 C^\natural_\bt = C^{T,\natural}_\bt \oplus R^{T,\natural}_\bt. 
 \label{eq:paracyclic.splitting-Cn-quasi}
\end{equation}
We also observe that  the splitting $C_\bt = C^T_\bt \oplus R^T_\bt$ implies that $1-T$ induces on $R^T_\bt$ a $k$-linear map which is both one-to-one and onto, and so we get a $k$-linear isomorphism of $R^T_\bt$. As it is compatible with the differentiels $(b,B)$ we get a parachain complex automorphism. 
We then let $\beta:R^T_\bt \rightarrow R^T_{\bt+1}$ be the $k$-module defined by
\begin{equation*}
 \beta =(1-T)^{-1}B=B(1-T)^{-1}. 
\end{equation*}

\begin{lemma}\label{lem:Parachain.contractibility}
 On $R^T_\bt$ we have
\begin{equation*}
 b\beta +\beta b=1, \qquad B\beta=\beta B=0. 
\end{equation*}
\end{lemma}
\begin{proof}
 It is immediate that $ B\beta=B^2(1-T)^{-1}=0$, and likewise $\beta B=0$. Moreover, the equality $bB+bB=1-T$ implies that on $R^T$ we have 
 \begin{equation*}
 b\beta+\beta b =(bB+Bb)(1-T)^{-1}=(1-T)(1-T)^{-1}=1.
\end{equation*}
 The proof is complete. 
\end{proof}

\begin{remark}
 The fact that $\beta$ is a contracting homotopy was observed by Hadfield-Kr\"ahmer~\cite[Proposition~2.1]{HK:KT05} in the case of the parachain complex of a twisted cyclic module. 
\end{remark}

The above lemma implies that on $R^{T,\natural}_\bt$ we have 
\begin{equation*}
 (b+Bu^{-1})\beta + \beta (b+Bu^{-1})= b\beta+\beta b=1. 
\end{equation*}
The contracting homotopy $\beta:R^{T,\natural}_\bt\rightarrow R^{T,\natural}_{\bt+1}$ is compatible with the operators $S$ and $T$. Thus, this is an $S$-homotopy, and so  the para-$S$-module $R^{T}$ is $S$-contractible. Combining this with the splitting~(\ref{eq:paracyclic.splitting-Cn-quasi}) we then arrive at the following result. 

\begin{proposition}\label{prop:parachain.quasi-mixed-quasi-S}
 If $C$ is a quasi-mixed complex, then $C^\natural$ is a quasi-$S$-module.  
\end{proposition}

Proposition~\ref{prop:parachain.quasi-mixed-quasi-S} allows us to apply Proposition~\ref{prop:Para-S-Mod.quasi-S-Mod}. More precisely, let $\pi_T:C_\bt\rightarrow C_{T,\bt}$ be the canonical projection of $C_\bt$ onto $C_{T,\bt}$ and $\pi^T:C_\bt \rightarrow C_\bt$ the projection on $C^T_\bt$ associated with the splitting $C_\bt = C^T_\bt \oplus R^T_\bt$. They are both maps of parachain complexes. As above, given any $x\in C_\bt$, we set $\wb{x}=\pi_T(x)$ and $x^T=\pi^T(x)$. The projection $\pi_T$ induces an isomorphism of mixed complexes from $C^T$ onto $C_T$. Under this isomorphism the inclusion of $C^T$ into $C_\bt$ corresponds to the parachain complex embedding $\iota_T:C_{T,\bt}\rightarrow C_\bt$ given by
\begin{equation*}
 \iota_T\left(\wb{x}\right)=x^T, \qquad x\in C_\bt. 
\end{equation*}
In addition, we let $h:C_\bt\rightarrow C_{\bt+1}$ be the $k$-linear map defined by
\begin{equation}
 h(x)=-B(1-T)^{-1}\left(x-x^T\right), \qquad x \in C_\bt. 
 \label{eq:paracyclic.h}
\end{equation}
This is a $T$-compatible map. It gives the chain homotopy~(\ref{eq:Para-S-Mod.hT}) on $C_\bt^\natural$ associated with the contracting homotopy $\beta=-B(1-T)^{-1}$ of $R^{T,\natural}$. Therefore, by using Proposition~\ref{prop:Para-S-Mod.quasi-S-Mod} we obtain the following result. 

\begin{proposition}\label{prop:Parachain.quasi-mixed-DR}
 Suppose that $C$ is a quasi-mixed complex. 
 \begin{enumerate}
 \item The canonical projection $\pi_T:C_\bt \rightarrow C_{T,\bt}$ induces a mixed complex isomorphism from $C^T$ onto $C_T$. 
 
 \item The map $\iota_T:C_\bt \rightarrow C_\bt$ given by~(\ref{eq:Para-S-Mod.iotaT}) is a parachain complex embedding. 
 
 \item We have
          \begin{equation*}
          \pi_T\iota_T=1, \qquad \iota_T\pi_T=\pi^T=1 + \left(b+Bu^{-1}\right)h + h  \left(b+Bu^{-1}\right). 
           \end{equation*}
           This provides us with $S$-deformation retracts of $C^\natural$ to $C_T^\natural$ and $C^{T,\natural}$ and $T$-deformation retracts of $C^\sharp$ to $C_T^\sharp$ and $C^{T,\sharp}$. 
\end{enumerate}
\end{proposition}

\begin{remark}
 In the same way as in Remark~\ref{rmk:Para-S-Module.homotopy-special-C'}, we have $\pi_Th=0$ and $h\iota_T=0$. Moreover, for all $x\in C_\bt$, we have 
 \begin{equation*}
 h^2(x)=B(1-T)^{-1}B(1-T)^{-1}(x-x^T)=(1-T)^{-1}B^2(1-T)^{-1}(x-x^T)=0.
\end{equation*}
That is, $h^2=0$. Therefore, we see that the chain homotopy $h$ is special in the sense of~(\ref{eq:Perturbation.special}). 
\end{remark}

Finally, when $C$ is a quasi-mixed complex we also can modify the $B$-operator to get an actual mixed complex whose cyclic $S$-module  is $S$-homotopy equivalent to $C^\natural$. Namely, we have the following result. 

\begin{proposition}\label{prop:Parachain.tC-mixed-complex}
 Suppose that $C=(C_\bt, b,B)$ is a quasi-mixed complex and set $\tilde{B}=\pi^TB$. 
 \begin{enumerate}
    \item $\tilde{C}:=(C_\bt, b,\tilde{B})$ is a mixed complex. 
    
    \item The projection $\pi^T:C_\bt \rightarrow C_\bt$ gives rise to an $S$-homotopy equivalence between $C^\natural$ and $\tilde{C}^\natural$ and a $T$-homotopy equivalence between their periodic para-complexes. 
\end{enumerate}
\end{proposition}
\begin{proof}
As the projection $\pi^T$ is a parachain complex map we have $\tilde{B}^2=\pi^T B^2=0$. As the range of $\pi^T$ is $C^T_\bt=\ker (1-T)$ we also have
$b\tilde{B}+\tilde{B}b=(bB+Bb)\pi^T=(1-T)\pi^T=0$. We thus get a mixed complex $\tilde{C}:=(C_\bt, b,\tilde{B})$.  Moreover, as $[b,\pi_T]=0$ and $\pi^T \tilde{B}=\tilde{B}\pi^T=\pi^T B =B\pi^T$, we see that  $\pi^T:C_\bt \rightarrow C_\bt$ is a parachain complex map from $C$ (resp., $\tilde{C}$) to $\tilde{C}$ (resp., $C$). 

By Proposition~\ref{prop:Parachain.quasi-mixed-DR} we have $(\pi^T)^2=\pi^T= 1 + (b+Bu^{-1})h + h  (b+Bu^{-1})$. In fact, it follows from Lemma~\ref{lem:Parachain.contractibility} and the definition of $h$ that $Bh=hB=0$. Thus,  
$\tilde{B}h=\pi^T Bh=0$ and $ h\tilde{B}=hB\pi^T=0$, and so we also have $(\pi^T)^2= 1 + (b+\tilde{B}u^{-1})h + h  (b+\tilde{B}u^{-1})$. This shows that the 
$S$-maps $\pi^T: C_\bt^\natural \rightarrow \tilde{C}_\bt^\natural$ and $S$-maps $\pi^T: C_\bt^\natural \rightarrow \tilde{C}_\bt^\natural$ are $S$-homotopy inverses of each other, and hence $C^\natural$ and $\tilde{C}^\natural$ are $S$-homotopy equivalent. Furthermore, by using Proposition~\ref{prop:Para-S-Mod.periodic-homotopy} we get a $T$-homotopy equivalence between the periodic para-complexes  $C^\sharp$ and $\tilde{C}^\sharp$. The proof is complete. 
\end{proof}

\section{Paracyclic and Para-Precyclic Modules}\label{sec:paracyclic-modules}
In this section, we review the main definitions and examples regarding paracyclic and para-precyclic modules. 

\subsection{Paracyclic modules} 
Recall that a \emph{simplicial $k$-module} $C=(C_\bt, d_\bt, s_\bt)$ is given by $k$-modules $C_m$, $m\geq 0$, together with $k$-linear maps $d_j:C_m \rightarrow C_{m+1}$, $j=0,\ldots, m$, called faces, and $k$-linear maps $s_j:C_{m}\rightarrow C_{m+1}$, $j=0,\ldots, m$, called degeneracies, which satisfy the following relations: 
\begin{equation}
 d_id_j=\left\{ 
\begin{array}{ll}
 d_{j-1}d_i, & i\leq j-1,\\
 d_jd_{i+1} & i\geq j,
\end{array}\right.
\label{eq:simplicial.dd}
\end{equation}

\begin{equation}
 s_is_j=\left\{ 
 \begin{array}{ll}
 s_{j+1}s_i, & i\leq j,\\
 s_js_{i-1}, & i\geq j+1,
\end{array}\right.
\label{eq:simplicial.ss}
\end{equation}

\begin{equation}
 d_is_j = \left\{ 
 \begin{array}{cl}
 s_{j-1}d_i, & i\leq j-1,\\
 1, & i=j,j+1,\\
 s_jd_{i-1}, & i\geq j+2.  
\end{array}\right.
\label{eq:simplicial.ds}
\end{equation}

A \emph{paracyclic $k$-module}~\cite{FT:LNM87a, GJ:Crelle93} (see also~\cite{DK:CMH85, FL:TAMS92}) is a simplicial $k$-module $(C_\bt, d_\bt, s_\bt)$ together with an invertible $k$-linear map $t:C_\bt \rightarrow C_{\bt}$ such that
\begin{equation}
\left\{ 
 \begin{array}{ll} 
 td_i = d_{i+1}t, &0\leq i \leq m-1,\\ 
 d_{m}=d_0t,
 \end{array}\right.
 \label{eq:paracyclic.td}
 \end{equation}
\begin{equation}
\left\{ 
 \begin{array}{ll} 
 ts_j=s_{j+1}t, &0\leq j \leq m-1,\\ 
 ts_m=s_{-1}.
 \end{array}\right.
  \label{eq:paracyclic.ts}
 \end{equation}
Here $s_{-1}:C_\bt \rightarrow C_{\bt+1}$ is the \emph{extra-degeneracy}, i.e.,
\begin{equation*}
 s_{-1}:=t^{-1}s_0t. 
\end{equation*}
We obtain a cyclic $k$-module in the sense of Connes~\cite{Co:CRAS83} when $t^{m+1}=1$ on $C_m$ for all $m\geq 0$. 

A paracyclic module structure is uniquely determined by the end-face $d=d_m$, the extra degeneracy $s=s_{-1}$ and the paracyclic operator $t=ds$. Indeed, it follows from~(\ref{eq:paracyclic.td}) and~(\ref{eq:paracyclic.ts}) that we have
\begin{equation*}
 d_i=t^{i-m}dt^{m-i}, \qquad s_i=t^{i+1} st^{-(i+1)}, \qquad i=0,\ldots, m. 
\end{equation*}
Combining this with the relations $d_m=d_0t$ and $s_{-1}=ts_m$ we further see that $d=d_0t=t^{-m}dt^{m+1}$ and $s=ts_m=t^{m+2}st^{-(m+1)}$. Thus, even in the non-cyclic case,  the operator $T:=t^{m+1}$ is invertible and commutes with the structural operators $(d,s,t)$. That is, this is a paracylic-module automorphism. 

We also record the following relations satisfied by the extra-degeneracy: 
\begin{equation}
 s_{-1}s_j=s_{j+1}s_{-1}, \qquad -1\leq j \leq m, 
  \label{eq:paracyclic.s1s}
\end{equation}
\begin{equation}
 d_is_{-1} = \left\{ 
  \begin{array}{ll} 
1 & i=0,\\
s_{-1}d_{i-1} & 0\leq j \leq m-1,\\ 
 t& i=m+1.
 \end{array}\right.
   \label{eq:paracyclic.ds1}
\end{equation}

\begin{example}[\cite{HK:KT05, KMT:JGP03}] \label{ex:paracyclic.twisted-cyclic}
Let $\cA$ be a unital $k$-algebra and $\sigma:\cA\rightarrow \cA$ a unital algebra automorphism. The corresponding twisted cyclic $k$-module $C^\sigma(\cA)$ is the paracyclic $k$-module $(C_\bt(\cA), d_\sigma, s, t_\sigma)$, where $C_m(\cA)=\cA^{\otimes (m+1)}$, $m\geq 0$, and the structural operators  $(d_\sigma, s, t_\sigma)$ are given by
\begin{align}
 d_\sigma(a^0\otimes \cdots \otimes a^m) & = \left ( \sigma(a^m)a^0\right) \otimes a^1 \otimes \cdots \otimes a^{m-1},
 \label{eq:paracyclic.twisted-dtau}\\
 s(a^0\otimes \cdots \otimes a^m) & = 1\otimes a^0\otimes \cdots \otimes a^m,\\
 t_\sigma (a^0\otimes \cdots \otimes a^m) & = a^m\otimes a^0\otimes \cdots \otimes a^{m-1}, \qquad a^j\in \cA. 
 \label{eq:paracyclic.twisted-ttau}
\end{align}
When $\sigma=1$ we recover the usual cyclic module of a unital algebra~\cite{Co:CRAS83, Co:IHES85}. In general, the operator $T=t_\sigma^{m+1}$ is just the $k$-module isomorphism $\sigma^{\otimes(m+1)}$ on $C_m(\cA)=\cA^{\otimes (m+1)}$. 
\end{example}

\begin{example}\label{ex:paracyclic.twisted-group}
 Let $\Gamma$ be a group and $\phi$ a normal element of $\Gamma$. When $k$ has a unit these data give rise to the paracyclic $k$-module $C^\phi(\Gamma)=(C_\bt(\Gamma), d,s_\phi, t_\phi)$, where $C_m(\Gamma)=k\Gamma^{m+1}$, $m\geq 0$, and the structural operators  $(d,s_\phi, t_\phi)$ are given by
\begin{align}
 d(\gamma_0, \ldots, \gamma_m) & = (\gamma_0, \ldots, \gamma_{m-1}),
 \label{eq:paracyclic.twisted-group-d}\\
 s_\phi(\gamma_0, \ldots, \gamma_m) & = \left(\phi^{-1}\gamma_m, \gamma_0, \ldots, \gamma_{m}\right),\\
  t_\phi(\gamma_0, \ldots, \gamma_m) & = \left(\phi^{-1}\gamma_m, \gamma_0, \ldots, \gamma_{m-1}\right),\qquad \gamma_j\in \Gamma.
  \label{eq:paracyclic.twisted-group-tphi}
\end{align}
When $\phi=1$ we recover the standard cyclic $k$-module of $\Gamma$ (see, e.g., \cite{Lo:CH}). In general, the operator $T=t_\phi^{m+1}$ on $C_m(\Gamma)$ arises from the diagonal left-action of $\phi^{-1}$ on $\Gamma^{m+1}$. Namely, we have 
\begin{equation*}
 T(\gamma_0, \ldots, \gamma_m)=\left(\phi^{-1}\gamma_0, \ldots, \phi^{-1}\gamma_m\right),\qquad \gamma_j\in \Gamma.
\end{equation*}
In any case, the $b$-differential of $C^\phi(\Gamma)$ is always the group differential~(\ref{eq:Para-S-Mod.group-differential}). 
\end{example}

\begin{remark}
 Paracyclic modules also naturally appear in Hopf cyclic theory (see~\cite{CM:CMP98, Cr:JPAA02, HKRS:CRAS04, Ka:KT05}).  
\end{remark}

As in~(\ref{eq:Para-S-Mod.CTRT}) we introduce the $k$-modules $C^T_\bt =\ker (1-T)$, $R^T_\bt=\ran (1-T)$, and $C_{T,\bt}=C_\bt \slash R^T_\bt$. As the structural operators $(d,s,t)$ commute with the operator $T$, they induce structural operators on the $k$-modules $C^T_\bt$ and $R^T_\bt$ and they descend to structural operators on $C_{T,\bt}$. We thus obtain paracyclic $k$-modules $C^T:=(C^T_\bt, d,s,t)$, $R^T:=(R^T_\bt, d,s,t)$ and $C_T=(C_{T,_\bt}, d,s,t)$. In addition, as $t^{m+1}=T=1$ on $C^T_m$ and $C_{T,m}$ we see that $C^T$ and $C_T$ are both cyclic $k$-modules. 

\begin{definition}[\cite{KK:Crelle14}] 
A quasi-cyclic $k$-module is a paracyclic $k$-module $C=(C_\bt, d,s,t)$ such that $C_\bt=C^T_\bt\oplus R^T_\bt$. 
\end{definition}

\begin{remark}
 If $C$ is a quasi-cyclic $k$-module, then the splitting $C_\bt=C^T_\bt\oplus R^T_\bt$ is a splitting of paracyclic $k$-modules.  
\end{remark}

In the same way as with Lemma~\ref{lem:Parachain.Q(T)-quasi} we have the following result. 

\begin{lemma}\label{lem:paracyclic.Q(T)-quasi-cyclic}
 Suppose that $C=(C_\bt, d,s,t)$ is a paracyclic $k$-module for which there is a polynomial $Q(X)\in k[X]$ such that $Q(T)(T-1)=0$ and $Q(1)=1$. Then $C$ is a quasi-cyclic $k$-module. 
\end{lemma}

\begin{example}[Feigin-Tsygan~\cite{FT:LNM87a}; see also~\cite{BHM:InvM93}]\label{ex:Paracylic.r-cyclic}
In the terminology of~\cite{FT:LNM87a}, given $r\geq 1$, an $r$-cyclic $k$-module is a paracyclic $k$-module $C=(C_\bt, d,s,t)$ for which $T^r=1$. In the same way as in Example~\ref{ex:Para-S-Mod.Tr=1}, when $r^{-1}\in k$ the assumptions of Lemma~\ref{lem:paracyclic.Q(T)-quasi-cyclic} are satisfied, and so in this case any $r$-cyclic $k$-module is a quasi-cyclic $k$-module. Instances of $r$-cyclic $k$-modules include the following:
\begin{itemize}
\item The twisted cyclic $k$-module $C^\sigma(\cA)$ of Example~\ref{ex:paracyclic.twisted-cyclic}, when the automorphism $\sigma$ is such that $\sigma^r=1$. 

\item The paracyclic $k$-module $C^\phi(\Gamma)$ of Example~\ref{ex:paracyclic.twisted-group}, when $\phi$ has order $r$. 
\end{itemize}
\end{example}

\begin{remark}
 We refer to~\cite[Example 3.10]{HK:JKT09} for an example of a paracyclic module which is not a quasi-cyclic module. 
\end{remark}

\subsection{Para-precylic Modules} 
Recall that a pre-simplicial $k$-module $C=(C_\bt, d_\bt)$ is given by $k$-modules $C_m$, $m\geq 0$, and $k$-linear maps $d_j:C_m \rightarrow C_{m+1}$, $j=0,\ldots, m$, satisfying the face relations~(\ref{eq:simplicial.dd}). In this paper, we shall use the following definition of a para-precyclic $k$-module. 

\begin{definition}[\cite{Ni:InvM93}] \label{def:para-precyclic.para-precyclic}
 A \emph{para-precyclic $k$-module} $C=(C_\bt, d_\bt, t)$ is a pre-simplicial $k$-module $(C_\bt, d_\bt)$ together with an invertible $k$-linear map $t:C_\bt \rightarrow C_\bt$ satisfying the relations~(\ref{eq:paracyclic.td}). 
\end{definition}

\begin{remark}
 Para-precyclic modules are also called quasi-cyclic modules in~\cite{Ni:InvM93}.  
\end{remark}

When the paracyclic operator $t$ further satisfies the cyclic condition $t^{m+1}=1$ on $C_m$ we recover the definition of a precyclic $k$-module~\cite{Lo:CH}. In the same way as with a paracyclic module structure, a para-precyclic module structure is uniquely determined by the end face $d=d_m$ on $C_m$ and the paracyclic operator $t$. Moreover, the operator $T=t_m^{m+1}$ commutes with the structural operators $(d,t)$, and so this a para-precyclic module automorphism.

\begin{example}\label{ex:para-precyclic.twisted-cyclic}
 Let $\cA$ be a $k$-algebra and $\sigma:\cA\rightarrow \cA$ an algebra automorphism. In the unital case we can form the twisted cyclic module $C^\sigma(\cA)$ of Example~\ref{ex:paracyclic.twisted-cyclic}. In general, we can form the para-precyclic $k$-module $C^\sigma(\cA)=(C_\bt(\cA), d_\sigma, t_\sigma)$, where $C_m=\cA^{\otimes (m+1)}$, $m\geq 0$, and the structural operators $(d_\sigma, t_\sigma)$ are defined as in~(\ref{eq:paracyclic.twisted-dtau}) and~(\ref{eq:paracyclic.twisted-ttau}). We obtain the usual precyclic module of an algebra when $\sigma=1$ (see, e.g.,~\cite{Lo:CH}). 
\end{example}

Let $C=(C_\bt, d,t)$ be a para-precyclic $k$-module. In the same way as with paracyclic $k$-modules, the fact that the structural operators $(d,t)$ commute with $T$ ensures us that they define a pre-paracylic $k$-module structure on 
$R^T_\bt=\ran (1-T)$ and precyclic $k$-module structures on $C^T_\bt=\ker(1-T)$ and $C_{T,\bt}=C_\bt \slash R_\bt^T$. We denote by $R^T$, $C^T$ and $C_T$ the corresponding pre-(para)cyclic $k$-modules. 

\begin{definition}\label{def:paracyclic.quasi-precyclic}
A \emph{quasi-precyclic $k$-module} is a para-precyclic $k$-module $C=(C_\bt, d,t)$ such that $C_\bt=C^T_\bt\oplus R^T_\bt$. 
\end{definition}

\begin{remark}\label{rmk:para-precyclic.splitting-quasi-precyclic}
 If $C$ is a quasi-precyclic $k$-module, then the splitting $C_\bt=C^T_\bt\oplus R^T_\bt$ is a splitting of para-precyclic $k$-modules.  
\end{remark}

We have the following version of Lemma~\ref{lem:paracyclic.Q(T)-quasi-cyclic}. 

\begin{lemma}\label{lem:paracyclic.Q(T)-quasi-precyclic}
 Suppose that $C=(C_\bt, d,t)$ is a para-precyclic $k$-module for which there is a polynomial $Q(X)\in k[X]$ such that $Q(T)(T-1)=0$ and $Q(1)=1$. Then $C$ is a 
 quasi-precyclic $k$-module. 
\end{lemma}

\begin{example}
Given $r\geq 1$, an $r$-precyclic $k$-module is a paracyclic $k$-module $C=(C_\bt, d,t)$ for which $T^r=1$. In the same way as in Example~\ref{ex:Paracylic.r-cyclic}, any $r$-precyclic $k$-module is a quasi-precyclic $k$-module when $r^{-1}\in k$. For instance,  the twisted precyclic $k$-module $C^\sigma(\cA)$ of Example~\ref{ex:para-precyclic.twisted-cyclic} is $r$-precyclic when the automorphism $\sigma$ is such that $\sigma^r=1$. 
\end{example}

\section{The Para-$S$-Module of an $H$-Unital Para-Precyclic Module}\label{sec:paracyclic}
In this section, by elaborating on~\cite{GJ:Crelle93} we explain how to associate a parachain complex with any paracyclic module and, more generally, with any $H$-unital para-precyclic module. As a result, this allows us to associate a para-$S$-module with any $H$-unital para-precyclic module. We shall also exhibit some of the properties of these para-$S$-modules. 

Any cyclic $k$-module gives rise to a mixed complex~\cite{Co:MFO81, Co:CRAS83, Co:IHES85}. In the terminology of~\cite{Co:MFO81, Co:CRAS83, Co:IHES85} the corresponding $C^\natural$-complex is precisely the total complex of Connes' $(b,B)$-bicomplex. As it turns out, this construction can be extended to precyclic modules that are $H$-unital, i.e., precyclic modules whose bar complexes are contractible (see, e.g., \cite{Ka:Crelle90, Wo:AM89}).  

Getzler-Jones~\cite{GJ:Crelle93} observed that the construction of the mixed complex of a cyclic module can be carried out \emph{mutatis mutandis} to associate a parachain complex with any paracyclic module. In fact, in the paracyclic setting, we only need to be a bit  careful with the choice of the contracting homotopy for the bar complex. Moreover, as we shall explain, this construction holds in greater generality for any $H$-unital para-precyclic module.

Let $C=(C_\bt, d,s,t)$ be a paracyclic $k$-module.  As $C$ is a simplicial $k$-module we have differentials $b:C_\bt \rightarrow C_{\bt-1}$ and $b':C_\bt \rightarrow C_{\bt-1}$ given by 
\begin{equation}
 b= \sum_{0 \leq j \leq m} (-1)^j d_j \qquad \text{and} \qquad  b'= \sum_{0 \leq j \leq m-1} (-1)^j d_j \qquad \text{on $C_m$}.
 \label{eq:parachain-paracyclic.bb'} 
\end{equation}
The relations~(\ref{eq:paracyclic.ds1}) imply that the bar complex $(C_\bt, b')$ is acyclic and a contracting homotopy is provided by the extra degeneracy $s$, i.e., we have
\begin{equation}
 b's + sb'=1 \qquad \text{on $C_\bt$}. 
 \label{eq:parachain-paracyclic.contracting-homotopy}
\end{equation}
Moreover, the extra degeneracy $s$ is compatible with the isomorphism $T=\tau^{m+1}$. 

The chain complexes $(C_\bt, b)$ and $(C_\bt, b')$ make sense for any pre-simplicial $k$-module. In particular, they make sense for any para-precyclic $k$-module. However, for a general para-precyclic $k$-module, the bar complex $(C_\bt, b')$ need not be contractible. 

\begin{definition}
A para-precyclic $k$-module $C=(C_\bt, d,t)$ is called \emph{$H$-unital} when its bar complex $(C_\bt, b')$ admits a $T$-compatible contracting homotopy. 
\end{definition}

An $H$-unital paracyclic $k$-module is thus given by a system $(C_\bt, d,s,t)$, where $(d,t)$ defines a para-precyclic $k$-module structure on $C_\bt$ and $s$ is a $T$-contractible homotopy satisfying~(\ref{eq:parachain-paracyclic.contracting-homotopy}). In particular, any paracyclic $k$-module is an $H$-unital para-precyclic $k$-module. 

\begin{example}
At least when $k$ is a division ring, for any $k$-algebra with a local unit the precyclic $k$-module $C(\cA)$ is $H$-unital (see~\cite{Lo:CH, Wo:AM89}). As a result, given any automorphism $\sigma:A\rightarrow A$ of finite order, the corresponding twisted precyclic $k$-module $C^\sigma(A)$ of Example~\ref{ex:para-precyclic.twisted-cyclic} is $H$-unital. 
\end{example}

Let $C=(C_\bt, d, t)$ be a para-precyclic $k$-module. The $k$-linear maps $\tau:C_\bt \rightarrow C_\bt$ and $N:C_\bt \rightarrow C_\bt$ are defined by 
\begin{equation}
 \tau=(-1)^m t \qquad \text{and} \qquad  N= 1+\tau + \cdots +\tau^m \qquad \text{on $C_m$}.
 \label{eq:parachain-paracyclic.tau-N} 
\end{equation}
Note that
\begin{equation*}
 (1-\tau)N=N(1-\tau)=1-T. 
\end{equation*}
Moreover, in the same way as with cyclic modules~\cite{Co:CRAS83, Co:IHES85, Ts:UMN83} (see also~\cite{Lo:CH}) we have 
\begin{equation}
 b(1-\tau)=(1-\tau)b'  \qquad \text{and} \qquad Nb=b'N. 
 \label{eq:parachain-paracyclic.bNtau}
\end{equation}

Suppose now that we have a contracting homotopy $s':C_\bt \rightarrow C_{\bt+1}$ satisfying~(\ref{eq:parachain-paracyclic.contracting-homotopy}). The $k$-linear map 
$B:C_\bt \rightarrow C_{\bt+1}$ is defined by 
\begin{equation}
 B=(1-\tau) s'N.
 \label{eq:parachain-paracyclic.B} 
\end{equation}
When $C$ is cyclic and $s'$ is the extra degeneracy we recover the usual $B$-operator~\cite{Co:MFO81, Co:CRAS83, Co:IHES85}. Using~(\ref{eq:parachain-paracyclic.bNtau}) and the fact that $b's'+s'b'=1$ we see that $bB+Bb$ is equal to 
\begin{equation*}
 b(1-\tau)s'N+(1-\tau)s'Nb  = (1-\tau)(b's'+s'b')N=(1-\tau)N=1-T. 
\end{equation*}
 We also have 
 \begin{equation*}
 B^2= (1-\tau) s'N(1-\tau) s'N=(1-\tau)s'(1-T)s'N. 
\end{equation*}
Therefore, we arrive at the following result. 

\begin{proposition}[see also~\cite{GJ:Crelle93, HK:JKT10}]\label{prop:parachain.paracomplex} 
If the contracting homotopy $s'$ is such that $s'(1-T)s'=0$, then $(C_\bt, b,B)$ is a parachain complex such that $bB+Bb=1-T$.
\end{proposition}

\begin{remark}
 When $C$ is a cyclic module, i.e., $T=1$, the condition $s'(1-T)s'=0$ is always satisfied, and so we can choose any contracting homotopy to define the operator $B$. In particular, we may take $s'$ to be the extra degeneracy $s$, in which case we recover the usual mixed complex of a cyclic module~\cite{Co:CRAS83, Co:IHES85}. 
\end{remark}

Let $C=(C_\bt, d,s,t)$ be an $H$-unital pre-paracylic $k$-module. If $T\neq 1$, then $s(1-T)s=(1-T)s^2$ need not be zero (there seems to be an oversight of this fact in~\cite{GJ:Crelle93}; see also~\cite{HK:JKT10} on this point). However, we can replace $s$ by the $k$-linear map $s':C_\bt \rightarrow C_{\bt+1}$ given by
\begin{equation*}
 s'=sb's. 
\end{equation*}
This is a contracting homotopy, since we have 
\begin{equation*}
 b's'+s'b'=b'sb's+sb'sb'=(1-sb')b's+sb'(1-b's)=b's+sb'=1. 
\end{equation*}
In addition, we have 
\begin{equation*}
 (s')^2=sb's s b's=s(1-sb')(1-b's)s=s(1-b's-sb')s=0. 
\end{equation*}
As $T$ commutes with $s$ and $b'$, and hence with $s'$, we deduce that $s'(1-T)s'=(1-T)(s')^2=0$. Therefore, by using $s'=sb's$ to define the operator $B$ we obtain a parachain complex.

\begin{definition}[compare~\cite{GJ:Crelle93}] \label{def:paracyclic.parachain-complex}
Given any $H$-unital para-precyclic $k$-module $C=(C_\bt, d,s,t)$ its parachain complex is $(C_\bt, b,B)$, where 
$b$ is defined as in~(\ref{eq:parachain-paracyclic.bb'}) and $B$ is given by~(\ref{eq:parachain-paracyclic.B}) with $s':=sb's$. The corresponding para-complex $C^\sharp:=(C_\bt^\sharp, b+B, T)$ is called the \emph{periodic para-complex} of $C$. 
\end{definition}

\begin{remark}
 When $C$ is a cyclic $k$-module, we get a mixed complex which is different from the usual mixed complex of a cyclic $k$-module, since for the latter the 
 $B$-differential is defined by using the extra degeneracy $s$ instead of the special homotopy $s'=b'sb'$. As it turns out, the two mixed complexes are isomorphic (see Proposition~\ref{prop:Parachain.isomorphism} below). 
\end{remark}

Let $C=(C_\bt, d,s,t)$ be an $H$-unital quasi-precylic $k$-module, so that  $C_\bt =C^T_\bt \oplus R^T_\bt$. This implies that the parachain complex $(C,b,B)$ is a quasi-mixed complex. As the contracting homotopy $s$ is compatible with the $T$-operator, it induces contracting homotopies for the bar complexes of $C^T$ and $R^T$ and it descends to a contracting homotopy for the bar complex of $C_{T}$. Thus, the para-precyclic $k$-module $R^T$ and the precyclic $k$-modules $C^T$ and $C_T$ are all $H$-unital. Moreover, their parachain complexes are precisely the parachain complexes $C^T, R^T, C_T$ associated with the parachain complex $(C_\bt,b,B)$.  In addition,  the canonical projection $\pi_T:C_\bt \rightarrow C_{T}^\bt$ and the projection $\pi^T:C_\bt \rightarrow C_\bt$ on $C^T$ defined by the splitting $C_\bt =C^T_\bt \oplus R^T_\bt$ are both maps of $H$-unital para-precyclic $k$-modules. Therefore, by using Proposition~\ref{prop:parachain.quasi-mixed-quasi-S} we arrive at the following result.

\begin{proposition}[see also~{\cite[Proposition~2.1]{HK:KT05}}] \label{eq:paracyclic.quasi-cyclic-DR}
Assume that $C$ is a quasi-cyclic (resp., $H$-unital quasi-precyclic) $k$-module. 
\begin{enumerate}
 \item The canonical projection $\pi_T:C_\bt\rightarrow C_{T,\bt}$ induces an isomorphism of cyclic (resp., precyclic) $k$-modules from $C^{T}$ onto $C_T$. 
 
 \item The embedding $\iota_T: C_{T,\bt} \rightarrow C^T_\bt$ given by~(\ref{eq:Para-S-Mod.iotaT}) is an embedding of paracyclic (resp., para-precyclic) $k$-modules.

 \item  We have
          \begin{equation*}
          \pi_T\iota_T=1, \qquad \iota_T\pi_T=\pi^T=1 + \left(b+Bu^{-1}\right)h + h  \left(b+Bu^{-1}\right). 
           \end{equation*}
This provides us with $S$-deformation retracts of $C^\natural$ to $C_T^\natural$ and $C^{T,\natural}$, and with $T$-deformation retracts of $C^\sharp$ to $C_T^\sharp$ and $C^{T,\sharp}$.  
\end{enumerate}
\end{proposition}

As $(C_\bt, b,B)$ is a quasi-mixed complex, by Proposition~\ref{prop:Parachain.tC-mixed-complex} we also have a  mixed complex $\tilde{C}=(C_\bt, b,\tilde{B})$ whose $S$-module is $S$-homotopy equivalent to $C^\natural$. Here $\tilde{B}=B\pi^T =(1-\tau)s'N\pi^T$, and so $\tilde{C}$ is defined in the same way as $(C_\bt, b,B)$, upon replacing $N$ by $\tilde{N}:=N\pi^T$ in the formula~(\ref{eq:parachain-paracyclic.B}) for $B$. 

When $C$ is an $r$-precyclic $k$-module and $r^{-1}\in k$, we know by Lemma~\ref{lem:Parachain.Q(T)-quasi} that $C$ is quasi-precyclic and $\pi^T=r^{-1}(1+\cdots + T^{r-1})$. On $C_m$ we then have
\begin{equation}
r \tilde{N}= \frac{1}{r}N(1+\cdots + T^{r-1}) = \biggl(\sum_{0\leq j \leq m} \tau^j \biggr) \biggl( \sum_{0\leq l \leq r-1} \tau^{(m+1)l}\biggr) = 
\sum_{j=0}^{r(m+1)-1} \tau^j. 
\label{eq:Paracyclic.FT-N}
\end{equation}
 Thus, $r\tilde{N}$ agrees with the $N$-operator for $r$-cyclic $k$-modules of Feigin-Tsygan~\cite[Appendix]{FT:LNM87a}.

\subsection{Dependence on the contracting homotopy $s'$} 
As it turns out, the choice of the contracting homotopy $s'$ is hardly important, since, as we are going to see, all the para-$S$-modules that we get are isomorphic. More precisely, let $C=(C_\bt, d,s,t)$ be an $H$-unital pre-paracylic $k$-module and suppose we are given another contracting homotopy  $\hat{s}:C_\bt \rightarrow C_{\bt +1}$ such  that $\hat{s}(1-T)\hat{s}=0$. (In particular, when $C$ a cyclic module we may take $\hat{s}$ to be the extra degeneracy $s$.) We thus get another para-$S$-module $(C^\natural, b+\hat{B}u^{-1}, u^{-1}, T)$, where $\hat{B}=(1-\tau)\hat{s}N$.  Let $f:C_\bt^\natural \rightarrow C_\bt^\natural$ be the $k$-linear map defined by 
\begin{equation}
 f=1 +(1-\tau) \hat{s} s' Nu^{-1}.
 \label{eq:parachain-paracyclic.f} 
\end{equation}
This maps commutes with $u^{-1}$ and with the operator $T$. This is also a graded $k$-linear isomorphism with inverse, 
\begin{equation*}
 f^{-1}= 1 + \sum_{j \geq 1} \left[ (1-\tau) \hat{s} s' N\right]^j u^{-j}. 
\end{equation*}

\begin{lemma}\label{lem:parachain-paracyclic.f-chain-map}
 We have 
\begin{equation*}
 f\left(b+Bu^{-1}\right)= \big(b+\hat{B}u^{-1}\big)f. 
\end{equation*}
\end{lemma}
\begin{proof}
 The operator $f(b+Bu^{-1})$ is equal to 
\begin{equation*}
 \left( 1 + (1-\tau) \hat{s} s' Nu^{-1}\right) \left(b+Bu^{-1}\right)= b + \left[ B+ (1-\tau) \hat{s} s'N b\right] u^{-1} + (1-\tau) \hat{s} s'NBu^{-2}. 
\end{equation*}
By using the relation $Nb=b'N$ we get
\begin{equation*}
B+ (1-\tau) \hat{s} s'N b= (1-\tau) s'N + (1-\tau) \hat{s} s'b'N= (1-\tau)( s' +\hat{s} s'b')N. 
\end{equation*}
As $s'(1-T)s'=0$ we also obtain
\begin{equation*}
 (1-\tau) \hat{s}s'NB=(1-\tau) \hat{s}s'N(1-\tau) s'N= (1-\tau) \hat{s}s'(1-T)s'N=0. 
\end{equation*}
Thus, 
\begin{equation}
  f\left(b+Bu^{-1}\right)= b + (1-\tau)( s' +\hat{s} s'b')Nu^{-1}.
  \label{eq:parachain-paracyclic.fd}  
\end{equation}

Similarly, the operator $(b+\hat{B}u^{-1})f$ is equal to 
\begin{equation*}
 \left(b+\hat{B}u^{-1}\right)  \left( 1 + (1-\tau) \hat{s} s' Nu^{-1}\right) = b + \left[ b(1-\tau) \hat{s} s' N +  \hat{B} \right] u^{-1}+ \hat{B}(1-\tau) \hat{s} s' Nu^{-2}. 
\end{equation*}
 Thanks to the relation $b(1-\tau)=b'(1-\tau)$ we obtain
 \begin{equation*}
 b(1-\tau) \hat{s} s' N +  \hat{B}= (1-\tau)b' \hat{s}s'N+(1-\tau)\hat{s} N= (1-\tau)\left[ b' \hat{s}s' + \hat{s}\right] N. 
\end{equation*}
Using the fact that $\hat{s}(1-T)\hat{s}=0$ we also get
\begin{equation*}
  \hat{B}(1-\tau) \hat{s} s' N = (1-\tau)\hat{s} N(1-\tau) \hat{s} s'N = (1-\tau)\hat{s} (1-T)  \hat{s} s'N=0. 
\end{equation*}
Thus, 
\begin{equation}
 \left(b+\hat{B}u^{-1}\right)f = b + (1-\tau)\left[ b' \hat{s}s' + \hat{s}\right] Nu^{-1}.
 \label{eq:parachain-paracyclic.df} 
\end{equation}

Combining~(\ref{eq:parachain-paracyclic.fd}) and~(\ref{eq:parachain-paracyclic.df}) gives 
\begin{equation*}
  f\left(b+Bu^{-1}\right)-  \left(b+\hat{B}u^{-1}\right)f= (1-\tau)\left[s'- \hat{s} + \hat{s} s'b'-b' \hat{s}s' \right] Nu^{-1}. 
\end{equation*}
As $\hat{s} s'b'-b' \hat{s}s' = \hat{s}(1-b's')-(1-\hat{s}b')s'=\hat{s}-s'$, we deduce that $f(b+Bu^{-1})-(b+\hat{B}u^{-1})f=0$. This proves the result.  
\end{proof}

Lemma~\ref{lem:parachain-paracyclic.f-chain-map} asserts that $f$ is a chain map from $(C^\natural_\bt, b+Bu^{-1})$ to $(C^\natural_\bt, b+\hat{B}u^{-1})$. As this is an isomorphism which commutes with $u^{-1}$ and $T$, we deduce that this is an isomorphism of para-$S$-modules. Therefore, we have proved the following result. 

\begin{proposition}\label{prop:Parachain.isomorphism}
 The para-$S$-modules  $(C^\natural_\bt, b+Bu^{-1}, u^{-1},T)$ and $(C^\natural_\bt, b+\hat{B}u^{-1}, u^{-1},T)$ are isomorphic. An explicit isomorphism is provided by~(\ref{eq:parachain-paracyclic.f}). 
\end{proposition}

\section{The Para-$S$-Module $C^\nnatural$}\label{sec:para-precyclic}
An alternative way to define the cyclic homology of a cyclic $k$-module is to use the $CC$-bicomplex~\cite{Co:CRAS83, Ts:UMN83}. 
From this point of view the cyclic homology is given by the homology of the total complex of this bicomplex. In fact, this bicomplex makes sense for any precyclic $k$-module, including precyclic modules of non-unital algebras. 

In this section, we shall extend this approach to the para-(pre)cyclic setting. Although we will not get a bicomplex, we will still be able to construct a para-$S$-module. We shall work in the setting of para-precyclic $k$-modules.

\subsection{Construction of $C^\nnatural$} 
Let $C=(C_\bt, d,t)$ be a para-precyclic $k$-module. The $k$-module $C^\nnatural_m$, $m\geq 0$, is defined by 
\begin{equation}
 C^\nnatural_m= \bigoplus_{p+q=m} C_qu^p =C_m u^0 \oplus C_{m-1}u \oplus C_{m-2}u^2 \oplus \cdots. 
 \label{eq:Cnnat.Cnnat}
\end{equation}
Note that $u^{-1}$ maps $C_m^\nnatural$ to $C_{m-1}u^0 \oplus C_{m-2}u \oplus \cdots =C_{m-1}^\nnatural$, and so $u^{-2}$ maps $C_m^\nnatural$ to 
$C_{m-2}^\nnatural$.  

We define the operator $\dl: C_\bt^\nnatural \rightarrow C_{\bt-1}^\nnatural$ by 
\begin{equation}
 \dl \left(xu^{2p+1}\right) = (1-\tau)x u^{2p}, \qquad  \dl \left(xu^{2p}\right) = \left\{ 
\begin{array}{cl}
 0  &  \text{if $p=0$},     \\
  Nxu^{2p-1}&   \text{if $p\geq 1$},
\end{array} \right. 
\label{eq:CCparacyclic.dl}
\end{equation}
where $\tau$ and $N$ are given by~(\ref{eq:parachain-paracyclic.tau-N}). We also define the operator $\delta: C_\bt^\nnatural \rightarrow C_{\bt-1}^\nnatural$ by 
\begin{equation}
 \delta\left( xu^{2p}\right) = bx u^{2p}, \qquad  \delta\left( xu^{2p+1}\right)=-b'xu^{2p+1}. 
 \label{eq:CCparacyclic.delta}
\end{equation}

\begin{lemma}[see also~\cite{Co:CRAS83, LQ:CMH84, Ts:UMN83}]\label{lem:CCparacyclic.square}
 We have 
 \begin{equation*}
 \dl^2 = (1-T)u^{-2}  \qquad \text{and} \qquad  \dl \delta +\delta \dl=\delta^2=0. 
\end{equation*}
\end{lemma}
\begin{proof}
It is immediate that $\delta^2=0$, since $b^2=(b')^2=0$. Let $x\in C_\bt$. Note that $\dl=Nu^{-1}$ on $C_\bt u^{2p}$ and $\dl=(1-t)u^{-1}$ on $C_\bt u^{2p+1}$. Thus, 
\begin{equation*}
 \dl^2\left( xu^{2p}\right)=\dl Nu^{-1}\left( xu^{2p}\right) = (1-\tau)Nu^{-2} \left( xu^{2p}\right)  =(1-T)u^{-2}\left( xu^{2p}\right).  
\end{equation*}
 Likewise, we have
 \begin{equation*}
  \dl^2 \left( xu^{2p+1}\right)=\dl (1-\tau)u^{-1}  \left( xu^{2p+1}\right) = N(1-\tau)u^{-2}  \left( xu^{2p+1}\right)  =(1-T)u^{-2} \left( xu^{2p+1}\right). 
\end{equation*}
Therefore, we see that $ \dl^2 = (1-T)u^{-2}$.  
 
 We have $(\dl \delta +\delta \dl)(xu^0)= \dl (bx u^0)=0$. If $p\geq 1$, then by using the relation $Nb=b'N$ we get
 \begin{equation*}
 (\dl \delta +\delta \dl)\left( xu^{2p}\right)= \dl \left( bxu^{2p}\right)+ \delta \left( Nxu^{2p-1}\right)= ( Nbx-b'Nx)u^{2p-1}=0. 
\end{equation*}
 If $p\geq 0$, then the equality $b(1-\tau)=(1-\tau)b'$ implies that we have
 \begin{equation*}
 (\dl \delta +\delta \dl)\left( xu^{2p}\right)= - \dl\left( b'xu^{2p+1}\right)+ \delta \left( (1-\tau)xu^{2p}\right) =\left[-b'(1-\tau)x+b(1-\tau)x\right]u^{2p}=0.
\end{equation*}
It then follows that $\dl \delta +\delta \dl=0$. The proof is complete. 
\end{proof}

It follows from Lemma~\ref{lem:CCparacyclic.square} that we have 
\begin{equation*}
 (\dl + \delta)^2=\dl^2 +\dl \delta +\delta\dl +\delta^2=(1-T)u^{-2}. 
\end{equation*}
As $u^{-1}$ and $T$ both commute with $\dl$ and $\delta$ we arrive at the following result. 

\begin{proposition}\label{prop:para-precyclic.para-S-module}
$C^\nnatural:= (C_\bt^\nnatural, \dl+\delta, u^{-2},T)$ is a para-$S$-module. 
\end{proposition}

\begin{remark}\label{def:para-precyclic.cyclic-bicomplex}
 If $C$ is a cyclic $k$-module, i.e., $T=1$, then $\dl^2=0$ and $(\dl+\delta)^2=0$, and so $C^\nnatural$ is an $S$-module. In particular, $(C_\bt^\nnatural, \dl+\delta)$ is a chain complex. In fact, this is the total complex of the bicomplex $(CC_{\bt,\bt}, \dl, \hat{\delta})$, where 
\begin{equation*}
 CC_{p,q}= C_q u^p, \qquad \hat{\delta}\left(xu^{2p}\right)=bxu^{2p}, \qquad \hat{\delta}\left(xu^{2p+1}\right)=b'xu^{2p+1}.
\end{equation*}
This bicomplex is precisely the $CC$-bicomplex.  
\end{remark}

\begin{remark}
When $C$ is an $H$-unital paracyclic $k$-module we obtain two para-$S$-modules $C^\natural$ and $C^\nnatural$. We will see in Section~\ref{sec:Cn-Cnn} that the former is a deformation retract of the latter.
\end{remark}

\subsection{Quasi-precyclic modules} 
Supposed now that $C=(C_\bt, d,t)$ is a quasi-precyclic $k$-module in the sense of Definition~\ref{def:paracyclic.quasi-precyclic}. As mentioned in Remark~\ref{rmk:para-precyclic.splitting-quasi-precyclic} the splitting $C_\bt=C^T_\bt\oplus R^T_\bt$ is a splitting of para-precyclic $k$-modules. In particular, we get a splitting of para-$S$-modules, 
\begin{equation}
 C^\nnatural_\bt = C^{T,\nnatural}_\bt \oplus R^{T,\nnatural}_\bt. 
 \label{eq:para-precyclic.splitting-Cnn-quasi}
\end{equation}

As mentioned in Section~\ref{sec:parachain} the splitting $C_\bt=C^T_\bt\oplus R^T_\bt$ also ensures us that $1-T$ induces an invertible $k$-linear map on $R^T_\bt$. Let $\beta^\nnatural: R^{T,\nnatural}_\bt \rightarrow R^{T,\nnatural}_{\bt+1}$ be the $k$-linear map defined by
\begin{equation*}
 \beta^\nnatural(xu^{2p})=N(1-T)^{-1}xu^{2p+1}, \qquad \beta^\nnatural(xu^{2p+1})=0, \qquad x\in R^T_\bt. 
\end{equation*}

\begin{lemma}
 On $R^{T,\nnatural}_\bt$ we have 
\begin{equation*}
  \dl \beta^\nnatural + \beta^\nnatural \dl =1, \qquad \delta \beta^\nnatural + \beta^\nnatural \delta=0. 
\end{equation*}
\end{lemma}
\begin{proof}
 Let $x\in R^T_\bt$. We have 
 \begin{align*}
  \left(\dl \beta^\nnatural + \beta^\nnatural \dl \right)(xu^{2p}) & = \dl \left[ N(1-T)^{-1} xu^{2p+1}\right] + \beta^\nnatural  \left[Nxu^{2p-1}\right] \\ 
  & = (1-\tau)N(1-T)^{-1}x u^{2p}.
\end{align*}
As $(1-\tau)N(1-T)^{-1}=(1-T)(1-T)^{-1}=1$ on $R^T_\bt$, we deduce that $(\dl \beta^\nnatural + \beta^\nnatural \dl)(xu^{2p})= xu^{2p}$. Likewise, $\left(\dl \beta^\nnatural + \beta^\nnatural \dl \right)(xu^{2p+1})$ is equal to  
\begin{equation*}
 \beta^\nnatural  \left[(1-\tau)xu^{2p}\right] = N(1-T)^{-1}(1-\tau)x u^{2p+1}= xu^{2p+1}.
\end{equation*}
Therefore, we see that $ \dl \beta^\nnatural + \beta^\nnatural \dl =1$ on  $R^{T,\nnatural}_\bt$. 

In addition, we have 
 \begin{align*}
  \left(\delta \beta^\nnatural + \beta^\nnatural \delta \right)(xu^{2p}) & = \delta \left[ N(1-T)^{-1} xu^{2p+1}\right] + \beta^\nnatural  \left[bxu^{2p}\right] \\ 
  & = -b'N(1-T)^{-1}x u^{2p+1}+N(1-T)^{-1}bx^{2p+1}\\
  & = (-b'N+Nb)(1-T)^{-1}x u^{2p+1}\\
  &=0.  
\end{align*}
We also have $\left(\delta \beta^\nnatural + \beta^\nnatural \delta\right)(xu^{2p+1})= 
-\beta^\nnatural [b'xu^{2p+1}]=0$. 
Therefore, we see that $\delta \beta^\nnatural + \beta^\nnatural \delta=0$ on all $R^{T, \nnatural}_\bt$. The proof is complete. 
\end{proof}

The above lemma implies that $\beta^\nnatural$ is a contracting homotopy for $(R^{T,\nnatural}_\bt, \dl+\delta)$. By construction $\beta^\nnatural$ is compatible with the $T$-operator. Moreover, we have $\beta^\nnatural u^{-2}=u^{-2}\beta^\nnatural $. Therefore, $\beta^\nnatural$ is an $S$-contracting homotopy, and so  $R^{T,\nnatural}$ is an $S$-contractible para-$S$-module. Combining this with the splitting~(\ref{eq:para-precyclic.splitting-Cnn-quasi}) we then arrive at the following statement. 

\begin{proposition}\label{prop:para-precyclic.Cnn-quasi}
 If $C$ is a quasi-precyclic $k$-module, then $C^\nnatural$ is a quasi-$S$-module. 
\end{proposition}

Proposition~\ref{prop:para-precyclic.Cnn-quasi} allows us to apply Proposition~\ref{prop:Para-S-Mod.quasi-S-Mod} as follows. Let $\pi_T:C_\bt\rightarrow C_{T,\bt}$ be the canonical projection of $C_\bt$ onto $C_{T,\bt}$ and $\pi^T:C_\bt \rightarrow C_\bt$ the projection on $C^T_\bt$ associated with the splitting $C_\bt=C^T_\bt\oplus R^T_\bt$.
They are both maps of paracyclic $k$-modules. As above, $\pi_T$ induces an isomorphism of para-precyclic $k$-modules from $C^T$ onto $C_T$, and, under this isomorphism, the inclusion of $C^T$ into $C_\bt$ corresponds to the embedding $\iota_T:C_{T,\bt}\rightarrow C_\bt$ given by~(\ref{eq:Para-S-Mod.iotaT}). This embedding is actually a para-precyclic module embedding. We also let $h:C_\bt\rightarrow C_{\bt+1}$ be the $k$-linear map defined by
\begin{equation}
 h^\nnatural(xu^{2p})=-N(1-T)^{-1}\left(x-x^T\right)u^{2p+1}, \qquad  h^\nnatural(xu^{2p+1})=0, \qquad x\in C_\bt.
 \label{eq:para-precyclic.homotopy-hnn} 
\end{equation}
This is a $T$-compatible map. This is also the chain homotopy~(\ref{eq:Para-S-Mod.hT}) associated with the contracting homotopy $\beta^\nnatural$. Therefore, by using Proposition~\ref{prop:Para-S-Mod.quasi-S-Mod} we obtain the following result.

\begin{proposition}\label{prop:para-precyclic.CTCT-Cnn}
 Assume that $C$ is a quasi-precyclic $k$-module. 
 \begin{enumerate}
 \item The canonical projection $\pi_T:C_\bt\rightarrow C_{T,\bt}$ induces a precyclic $k$-module isomorphism from $C^T$ onto $C_T$. 
 
 \item The map $\iota_T:C_{T,\bt} \rightarrow C_\bt$ given by~(\ref{eq:Para-S-Mod.iotaT}) is a para-precyclic $k$-module embedding. 
 
 \item We have
           \begin{equation*}
                      \pi_T\iota_T=1,\qquad \pi^T=\iota_T\pi_T=1 + (\dl +\delta)h^\nnatural + h^\nnatural (\dl +\delta). 
            \end{equation*}
           This provides us with $S$-deformation retracts of $C^\nnatural$ to $C^\nnatural_T$ and $C^{T,\nnatural}$.
\end{enumerate}
\end{proposition}

\begin{remark}\label{rmk:para-precyclic.hnn-special}
 As we can see from~(\ref{eq:para-precyclic.homotopy-hnn}), the range of $h^\nnatural$ is contained in $\oplus_{p\geq 0} R_\bt^Tu^{2p+1}$ and its kernel contains $\oplus_{p\geq 0} C_\bt^T u^{2p}$ and $\oplus_{p\geq 0} C_\bt u^{2p+1}$. It then follows that $\pi_T h^\nnatural=0$, $h^\nnatural \iota_T=0$ and $(h^\nnatural)^2=0$, i.e., the chain homotopy $h^\nnatural$ is special in the sense of~(\ref{eq:Perturbation.special}). 
\end{remark}

Given any $r$-cyclic $k$-module $C$, Feigin-Tsygan~\cite[Appendix]{FT:LNM87a} defined a chain bicomplex, which is like the $CC$-bicomplex 
$(CC_{\bt,\bt}, \dl, \hat{\delta})$ (\emph{cf.}~Remark~\ref{def:para-precyclic.cyclic-bicomplex}). The only difference is with the formula~(\ref{eq:CCparacyclic.dl}) for the differential $\dl$, where the operator $N$ is replaced by the operator $\sum_{j=0}^{r(m+1)-1} \tau^j$. As we shall now see, the Feigin-Tsygan bicomplex is a special case of a general construction for quasi-precyclic modules. 

Suppose that $C$ is a quasi-precyclic $k$-module. As mentioned above the projection $\pi^T:C_\bt \rightarrow C_\bt$ is a para-precyclic $k$-module map. As in Section~\ref{sec:paracyclic} set $\tilde{N}=N\pi^T=\pi^TN$, and let $\tilde{\dl}:C_\bt^\nnatural \rightarrow C_{\bt-1}^\nnatural$ be the $k$-linear map defined as in~(\ref{eq:CCparacyclic.dl}) by using the operator $\tilde{N}$ instead of $N$. Thus, $\tilde{\dl}=\pi^T\dl=\dl\pi^T$ on $C_\bt u^{2p}$ and $\tilde{\dl}=\dl$ on $C_\bt u^{2p+1}$. When  $C$ is  $r$-precyclic and $r^{-1}\in k$, the operator $\tilde{N}$ is $r^{-1}$-times the Feigin-Tsygan operator $\sum_{j=0}^{r(m+1)-1} \tau^j$ on $C_m$ (see  Eq.~(\ref{eq:Paracyclic.FT-N})).  

\begin{proposition}\label{prop:para-precyclic.CC-quasi-precyclic}
 Suppose that $C$ is a quasi-precyclic $k$-module.
\begin{enumerate}
 \item $\widetilde{CC}:=(CC_{\bt, \bt}, \tilde{\dl}, \hat{\delta})$ is a chain bicomplex and $\widetilde{C}^\nnatural:=(C^\nnatural_\bt, \tilde{\dl}+\delta, u^{-2})$ is an $S$-module. 
 
 \item The projection $\pi^T:C_\bt \rightarrow C_\bt$ yields an $S$-homotopy equivalence between $C^\nnatural$ and  $\widetilde{C}^\nnatural$.
\end{enumerate}
\end{proposition}
\begin{proof}
 We know by Lemma~\ref{lem:CCparacyclic.square} that $\dl^2=(1-T)u^{-2}$ and $\delta^2=\dl\delta+ \delta\dl=0$. As mentioned above $\tilde{\dl}=\pi^T\dl=\dl\pi^T$ on $C_\bt u^{2p}$ and $\tilde{\dl}=\dl$ on $C_\bt u^{2p+1}$. Thus, $\tilde{\dl}^2=\pi^T\delta^2=\pi^T(1-T)u^{-2}=0$. Moreover, as $\delta$ maps $C_\bt u^q$ to $C_{\bt-1} u^q$ we have $\tilde{\dl}\delta+ \delta\tilde{\dl}=\pi^T(\dl\delta+ \delta\dl)=0$ on $C_\bt u^{2p}$ and $\tilde{\dl}\delta+ \delta\tilde{\dl}=\dl\delta+ \delta\dl=0$ on $C_\bt u^{2p+1}$. Therefore, we see that $\delta^2=\tilde{\dl}^2=\tilde{\dl}\delta+ \delta\tilde{\dl}=0$, and hence $(\tilde{\dl}+\delta)^2=0$. It then follows that $(CC_{\bt, \bt}, \tilde{\dl}, \hat{\delta})$ is a chain bicomplex, and so its total complex $(C^\nnatural_\bt, \tilde{\dl}+\delta)$ is a chain complex. As $u^{-2}$ is compatible with the operators $(\dl, \delta)$, we then see that 
 $\widetilde{C}^\nnatural:=(C^\nnatural_\bt, \tilde{\dl}+\delta, u^{-2})$ is an $S$-module. 
 
 As $[\delta,\pi^T]=[u^{-2},\pi^T]=0$ and $ \tilde{\dl}\pi^T =  \pi^T\tilde{\dl}=\dl \pi^T= \pi^T\dl$ we see that $\pi^T$ gives rise to an $S$-map from 
 $C^\nnatural$ (resp., $\widetilde{C}^\nnatural$) to $\widetilde{C}^\nnatural$ (resp., $C^\nnatural$). Moreover, by Proposition~\ref{prop:para-precyclic.CTCT-Cnn} we have $(\pi^T)^2=\pi^T=1+(\dl +\delta)h^\nnatural + h^\nnatural(\dl +\delta)$, where the chain homotopy $h^\nnatural :C^\nnatural_\bt \rightarrow C^\nnatural_{\bt+1}$ is given by~(\ref{eq:para-precyclic.homotopy-hnn}). Thus, $\pi^T: C_\bt^\nnatural \rightarrow \widetilde{C}^\nnatural$ is an $S$-homotopy right-inverse of  $\pi^T: \tilde{C}_\bt^\nnatural \rightarrow C^\nnatural$. 
 
 Let $x\in C_\bt$. By definition $h^\nnatural$ vanishes on $C_\bt u^{2p+1}$ and maps $C_\bt u^{2p}$ to $C_\bt u^{2p-1}$. As the images of $xu^{2p}$ by $\dl$ and $\tilde{\dl}$ are contained in $C_\bt u^{2p-1}$, we get
\begin{equation*}
 (\tilde{\dl} h^\nnatural + h^\nnatural\tilde{\dl})(xu^{2p})= \dl h^\nnatural (xu^{2p}) = (\dl h^\nnatural + h^\nnatural\dl)(xu^{2p}). 
\end{equation*}
 Similarly, as $h^\nnatural (xu^{2p+1})=0$ and $\tilde{\dl}(xu^{2p+1}) = \dl(xu^{2p+1})$, we have 
 \begin{equation*}
 (\tilde{\dl} h^\nnatural + h^\nnatural\tilde{\dl})(xu^{2p+1})= \dl h^\nnatural (xu^{2p+1}) = (\dl h^\nnatural + h^\nnatural\dl)(xu^{2p+1}). 
\end{equation*}
Therefore, we see that $\tilde{\dl} h^\nnatural + h^\nnatural\tilde{\dl}=\dl h^\nnatural + h^\nnatural\dl$, and so $(\pi^T)^2=1+(\tilde{\dl} +\delta)h^\nnatural + h^\nnatural(\tilde{\dl} +\delta)$. It then follows that $\pi^T$ gives rise to an $S$-homotopy equivalence between $C^\nnatural$ and  $\widetilde{C}^\nnatural$. 
The proof is complete. 
\end{proof}

\section{Perturbation Lemmas and Co-Extensions}\label{sec:Perturbation}
The basic perturbation lemma~\cite{Br:Messina64, EM:AM47, Gu:IJM72,  Sh:IHES62} is a simple and powerful tool for constructing chain homotopy equivalences in homological algebra  (see, e.g.,~\cite{Ba:Preprint98, Be:AGT14, Br:Messina64, EM:AM47, Gu:IJM72, GL:IJM89, GLS:IJM91, HK:MZ91, Ka:Crelle90, KR:CMB04, LS:MM87, La:SM01, Re:HHA00, Sh:IHES62}). 
In particular, as first observed by Kassel~\cite{Ka:Crelle90}, in the context of cyclic homology it allows us to convert various deformation retracts of Hochschild complexes  into deformation retracts of cyclic complexes (see also~\cite{Ba:Preprint98, KR:CMB04, Po:EZ}). Incidentally, it also provides us with a natural interpretation of the $B$-operator of an $H$-unital para-precyclic module (\emph{cf}.~\cite[Remarque~5.3]{Ka:Crelle90}; see also Section~\ref{sec:Cn-Cnn}). 

In this section, we seek for generalizations of the basic perturbation lemma for para-$S$-modules, and more especially para-$S$-modules associated with parachain complexes. To this end we shall establish a generalized perturbation lemma for \emph{para-twin complexes} (see below for their precise definition). Specializing these generalized perturbation lemmas to para-S-modules of parachain complexes will provide us with general recipes to convert deformation retracts of Hochschild chain complexes into $S$-deformation retracts of para-$S$-modules.

In the next sections, we will present various applications of the generalized perturbation theory of this section. We also refer to~\cite{Po:EZ} for further applications regarding the normalization of the para-$S$-modules of paracyclic and bi-paracyclic modules and a version of the  Eilenberg-Zilber theorem for bi-paracyclic modules.  

\subsection{Basic perturbation lemma for para-twin complexes} 
In what follows by a \emph{para-twin complex} we shall mean a system $(C_\bt, \dl,\delta)$, where $C_m$, $m \geq 0$, are $k$-modules and $\dl: C_\bt \rightarrow C_{\bt-1}$ and  $\delta: C_\bt \rightarrow C_{\bt-1}$ are $k$-linear maps. We call  $(C_\bt, \dl,\delta)$ a \emph{twin complex} when we further have $\dl^2=\delta^2=\dl\delta +\delta \dl=0$. In this case $(C_\bt, \dl)$, $(C_\bt, \delta)$ and $(C_\bt, \dl+\delta)$ are chain complexes. Given two twin complexes  $(C_\bt, \dl,\delta)$ and $(\wbC_\bt, \dl, \delta)$, the basic perturbation lemma~\cite{Br:Messina64, EM:AM47, Gu:IJM72,  Sh:IHES62} is a simple recipe for converting a deformation retract of 
$(C_\bt, \dl)$ to $(\wbC_\bt, \dl)$ into a deformation retract of $(C_\bt, \dl+\delta)$. 

We seek for generalizing the basic perturbation lemma to para-twin complexes. Thus, let $C=(C_\bt, \dl,\delta)$ and $(\wbC_\bt, \dl, \delta)$ be para-twin complexes, and  assume we are given $k$-linear maps $f:C_\bt \rightarrow \wbC_\bt$ and $g:\wbC_\bt \rightarrow C_\bt$ such that
             \begin{equation}
                 gf=1 +\dl \varphi + \varphi \dl, 
                   \label{eq:Perturbation.homotopy-gf}
             \end{equation}
where $\varphi: C_\bt \rightarrow C_{\bt+1}$  is a $k$-linear map so that, for every $m\geq 0$, we have 
\begin{equation}
(\delta \varphi)^j=(\varphi \delta)^j=0 \qquad \text{on $C_m$ for $j\gg 1$}. 
\label{eq:perturbation.delta-vphij}
\end{equation}
These data allow us to define $k$-linear maps $\tilde{\varphi}:C_\bt \rightarrow C_{\bt+1}$, $\tilde{f}: C_\bt \rightarrow C_{\bt}$ and $\tilde{g}: \wbC_\bt \rightarrow \wbC_{\bt}$ by
\begin{gather}
 \tilde{\varphi} = \sum_{j \geq 0} \varphi (\delta \varphi)^j = \sum_{j \geq 0} (\varphi \delta)^j  \varphi ,
 \label{eq:Perturbation.tvarphi} \\
 \tilde{f}= f(1+ \delta \tilde{\varphi}), \qquad \tilde{g}= (1+  \tilde{\varphi}\delta)g. 
  \label{eq:Perturbation.tf}
\end{gather}
Note that we have 
\begin{gather}
  \tilde{f}= \sum_{j\geq 0} f(\delta \varphi)^j = f + f (\delta \varphi) + f (\delta \varphi)^2 + \cdots,
  \label{eq:Perturbation.tf-expansion} \\ 
  \tilde{g}= \sum_{j\geq 0} (\varphi \delta )^j g = g +   (\varphi \delta ) g +  (\varphi \delta )^2 g+ \cdots,
   \label{eq:Perturbation.tg-expansion}
\end{gather}

We also define the $k$-linear map $\tilde{\delta}: \wbC_\bt \rightarrow \wbC_{\bt-1}$ by
\begin{equation}
 \tilde{\delta}= f(\delta +\delta \tilde{\varphi} \delta) g= f \delta g + f \delta \tilde{\varphi} \delta g.
 \label{eq:Perturbation.tdelta0} 
\end{equation}
We observe that
\begin{equation}
 \tilde{\delta}= \tilde{f}\delta g= f \delta \tilde{g}.  
  \label{eq:Perturbation.tdelta}
\end{equation}
In particular, we have 
\begin{equation}
  \tilde{\delta}= \sum_{j \geq 0} f (\delta \varphi)^j \delta g = \sum_{j \geq 0} f \delta (\varphi \delta )^j g= f\delta g + f\delta \varphi \delta g +\cdots .
  \label{eq:Perturbation.tdelta-expansion}
\end{equation}
We also set 
\begin{equation}
 \Delta =  \delta^2+ \dl \delta + \delta \dl.
 \label{eq:perturbation.Delta}
\end{equation}

We shall say that the chain homotopy $\varphi$ in~(\ref{eq:Perturbation.homotopy-gf}) is \emph{special} when the following conditions are satisfied:
 \begin{equation}
 f\varphi =0, \qquad \varphi g=0, \qquad \varphi^2=0.
 \label{eq:Perturbation.special} 
\end{equation}
We are now in a position to obtain the following generalization of the basic perturbation lemma. 

\begin{lemma}\label{lem:perturbation-lemma-special}
 Suppose that the maps $f$ and $g$ are compatible with the operators $\dl$, and the chain homotopy $\varphi$ is special and $\Delta$-compatible. 
Then we have
\begin{gather}
 \tilde{f}(\dl+\delta)= (\dl +\tilde{\delta}) \tilde{f}, 
 \label{eq:Perturbation.tf-chain-map}\\
 \tilde{g}(\dl+\tilde{\delta})= (\dl +\delta) \tilde{g},
 \label{eq:Perturbation.tg-chain-map}\\
  \tilde{g}\tilde{f}= 1+ (\dl +\delta) \tilde{\varphi} + \tilde{\varphi}( \dl+\delta),
   \label{eq:Perturbation.homotopy-tgtf}\\
     \tilde{f}\tilde{g}=fg,
    \label{eq:Perturbation.homotopy-tftg} \\
 \tilde{\delta}^2+ \dl\tilde{\delta} + \tilde{\delta}\dl =f\Delta g.
\label{eq:Perturbation.tdelta-fdeltag}
 \end{gather}
In addition, the chain homotopy $\tilde{\varphi}$ is special. 
\end{lemma}
\begin{proof}
We observe that~(\ref{eq:Perturbation.tvarphi}) implies that 
 \begin{equation}
 \tilde{\varphi} = \varphi +\tilde{\varphi} \delta \varphi=  \varphi + \varphi \delta \tilde{\varphi}. 
 \label{eq:Perturbation.varphi}
\end{equation}
We also have the following equalities, 
 \begin{equation}
 f\Delta \varphi=0, \qquad \varphi \Delta g=0, \qquad \varphi \Delta \varphi=0.
 \label{eq:Perturbation.conditions}
\end{equation}
Indeed, as $\varphi$ is special and compatible with $\Delta$ we have $f\Delta \varphi= f\varphi\Delta =0$. Likewise, $\varphi \Delta g=\Delta\varphi  g=0$ and $\varphi \Delta \varphi= \Delta \varphi^2=0$.

By using~(\ref{eq:Perturbation.homotopy-gf}) and~(\ref{eq:Perturbation.varphi}) we see that $\tilde{g}\tilde{f}$ is equal to 
\begin{align*}
 (1+\tilde{\varphi} \delta)gf (1+\delta \tilde{\varphi})  = &  (1+\tilde{\varphi} \delta)(1+\dl \varphi + \varphi \dl) (1+\delta \tilde{\varphi})\\ 
 = & 1 + \dl (\varphi + \varphi \delta \tilde{\varphi}) + \delta \tilde{\varphi} + (\varphi + \tilde{\varphi} \delta \varphi)\dl + \tilde{\varphi} \delta \\ 
  & + 
 \tilde{\varphi} \delta \dl (\varphi + \varphi \delta \tilde{\varphi}) + (\varphi + \tilde{\varphi} \delta \varphi) \dl \delta \tilde{\varphi} + \tilde{\varphi} \delta^2 \tilde{\varphi} \\
 = & 1+ (\dl + \delta)\tilde{\varphi} +  \tilde{\varphi}(\dl + \delta) + \tilde{\varphi} \Delta \tilde{\varphi}. 
\end{align*}
 We know by~(\ref{eq:Perturbation.conditions}) that  $\varphi \Delta \varphi=0$, and so $\tilde{\varphi} \Delta \tilde{\varphi}= (1+\tilde{\varphi}\delta) \varphi \Delta \varphi (1+\delta \tilde{\varphi}) =0$. It then follows that $ \tilde{g}\tilde{f}= 1+ (\dl + \delta)\tilde{\varphi} +  \tilde{\varphi}(\dl + \delta)$, i.e., we have~(\ref{eq:Perturbation.homotopy-tgtf}). 
 
By assumption $[\dl, f]=0$. Combining this with~(\ref{eq:Perturbation.tdelta}) and (\ref{eq:Perturbation.homotopy-tgtf}) gives
\begin{align*}
(\dl + \tilde{\delta})\tilde{f} &=  \dl f (1+ \delta \tilde{\varphi} ) + f \delta \tilde{g} \tilde{f}\\
\ & = f \dl  (1+ \delta \tilde{\varphi} ) +  f \delta\left[ 1+ (\dl + \delta)\tilde{\varphi} +  \tilde{\varphi}(\dl + \delta)\right]  \\
 &= (f+ f \delta \tilde{\varphi}) \dl + (f+f\delta \tilde{\varphi}) \delta + f(\dl \delta +\delta \dl +\delta^2)\tilde{\varphi}\\
&  = \tilde{f} (\dl +\delta)  + f \Delta \tilde{\varphi}. 
\end{align*}
We know by~(\ref{eq:Perturbation.conditions})  that $f\Delta\varphi=0$. Thus,  $f \Delta \tilde{\varphi}= f\Delta\varphi(1+\delta \tilde{\varphi})=0$, and hence $(\dl + \tilde{\delta})\tilde{f}= \tilde{f} (\dl +\delta)$. 
Similarly, as by assumption $[\dl,g]=0$, we see that $\tilde{g}(\dl + \tilde{\delta})$ is equal to 
\begin{align*}
  (1+  \tilde{\varphi}\delta ) g  \dl + \tilde{g} \tilde{f} \delta g  = & (1+  \tilde{\varphi}\delta )  \dl g +    (1+\dl \varphi + \varphi \dl) \delta g \\
  = & \dl  (g+   \tilde{\varphi}\delta g) + \delta (g+   \tilde{\varphi}\delta g) +  \tilde{\varphi}(\dl \delta +\delta \dl +\delta^2) g\\
 = & ( \dl +\delta) \tilde{g} + \tilde{\varphi}\Delta g. 
\end{align*}
As by~(\ref{eq:Perturbation.conditions}) we have $\varphi\Delta g=0$, we get $\tilde{\varphi}\Delta g=(1+ \tilde{\varphi}\delta)\varphi \Delta g=0$, and so 
$\tilde{g}(\dl + \tilde{\delta})=(\dl + \delta)\tilde{g}$. 

By using~(\ref{eq:Perturbation.tdelta}) and~(\ref{eq:Perturbation.homotopy-tgtf}) we see that $\tilde{\delta}^2$ is equal to
\begin{align*}
  f\delta \tilde{g} \tilde{f} \delta g  =  & f\delta\left[ 1+ (\dl + \delta)\tilde{\varphi} +  \tilde{\varphi}(\dl + \delta)\right]  \delta g \\
=  &  f \delta^2 + f\left( \delta^2+\delta \dl\right) \tilde{\varphi} \delta g + f \delta \tilde{\varphi} \left( \delta^2+\dl \delta \right) g \\ 
= &  f \Delta g + f \Delta \tilde{\varphi} \delta g + f  \delta\tilde{\varphi} \Delta g \
  - f(\dl \delta +\delta \dl) g - f \dl \delta \tilde{\varphi} \delta g  - f  \delta\tilde{\varphi} \delta \dl  g. 
\end{align*}
By using the compatibility of $f$ and $g$ with the $\dl$-operators and the equalities  $f\Delta \tilde{\varphi}=0$ and $\tilde{\varphi} \Delta g=0$ we obtain
\begin{align*}
\tilde{\delta}^2 = & f\Delta g -  \dl f\delta - \delta g \dl -  \dl f\delta \tilde{\varphi} \delta g  - f  \delta\tilde{\varphi} \delta g \dl \\
= & f\Delta g - \dl f(\delta + \delta \tilde{\varphi} \delta) g - f (\delta + \delta \tilde{\varphi} \delta)g \\
= &  f\Delta g -\dl \tilde{\delta}-  \tilde{\delta}\dl. 
\end{align*}
This proves~(\ref{eq:Perturbation.tdelta-fdeltag}). 

It remains to show that $\tilde{f}\tilde{g}=fg$ and $\tilde{\varphi}$ is special. We have $\tilde{\varphi}^2=(1+\tilde{\varphi}\delta)\varphi^2 (1+ \delta \tilde{\varphi})=0$. 
Combining this with~(\ref{eq:Perturbation.tf}) we get
\begin{equation}
\tilde{f}\tilde{g}= f(1+ \delta \tilde{\varphi})(1+\tilde{\varphi}\delta)g= fg + f\tilde{\varphi}\delta g + f \delta \tilde{ \varphi}g.
\label{eq:Perturbation.tftg}
\end{equation}
We also have $f\tilde{\varphi}=f\varphi (1+ \delta \tilde{\varphi})=0$ and $\tilde{\varphi} g = (1+\tilde{\varphi}\delta)\varphi g=0$. Thus, 
\begin{gather*}
 \tilde{f}\tilde{\varphi} = f(1+ \delta \tilde{\varphi}) \tilde{\varphi}= f \tilde{\varphi} + f\delta \tilde{\varphi}^2=0,\\
 \tilde{\varphi} \tilde{g}=  \tilde{\varphi} (1+\tilde{\varphi}\delta)g=  \tilde{\varphi}g +  \tilde{\varphi}^2\delta g=0.
\end{gather*}
Combining this with~(\ref{eq:Perturbation.tftg})  shows that $\tilde{f}\tilde{g}=fg$. As $\tilde{\varphi}^2=0$ this also shows that the chain homotopy $\tilde{\varphi}$ is special. 
The proof is complete. 
\end{proof}

\begin{remark}
When $fg=1$ and $C$ and $\wbC$ are both twin complexes (i.e., $\dl^2=\delta^2=\dl \delta +\delta \dl=0$) we recover the basic perturbation lemma in the version of~\cite{Br:Messina64, Gu:IJM72}.  
\end{remark}

\begin{remark}\label{rmk:Perturbation.fg}
When $fg=1$ and $[\Delta,f]=0$ the equalities~(\ref{eq:Perturbation.homotopy-tftg}) and~(\ref{eq:Perturbation.tdelta-fdeltag}) become
\begin{equation}
 \tilde{f}\tilde{g} = fg =1 \qquad \text{and} \qquad \tilde{\delta}^2+ \dl\tilde{\delta} + \tilde{\delta}\dl=f\Delta g= \Delta fg =\Delta. 
 \label{eq:Perturbation.tdelta-fdeltag-fg=1}
\end{equation}
\end{remark}

\begin{remark}\label{rmk:Perturbation.Delta-0}
The proof of Lemma~\ref{lem:perturbation-lemma-special} shows a little more. A close examination of the proof shows the following: 
\begin{enumerate}
 \item[(a)] In order to get~(\ref{eq:Perturbation.homotopy-tgtf})  we only need~(\ref{eq:Perturbation.conditions}). 
 
 \item[(b)] In order to get~(\ref{eq:Perturbation.tf-chain-map}) (resp., (\ref{eq:Perturbation.tg-chain-map})) we only need~(\ref{eq:Perturbation.conditions}) and the compatibility of the map $f$ (resp., $g$) with the $\dl$-operators. 
 
 \item[(c)] In order to get~(\ref{eq:Perturbation.tdelta-fdeltag}) we  only need~(\ref{eq:Perturbation.conditions}) and the compatibility with the $\dl$-operators of  both maps $f$ and $g$. 
\end{enumerate}
\end{remark}

\begin{remark}\label{rmk:Perturbation.f-tdelta-compatible}
If $[\delta, f]=0$, then $\tilde{f}=f$ and $\tilde{\delta}=f\delta g= \delta f g$. Indeed, as $\varphi$ is special, we have $f \delta \varphi=\delta f \varphi=0$, and so by using~(\ref{eq:Perturbation.tf-expansion}) and~(\ref{eq:Perturbation.tdelta-expansion}) we see that  $\tilde{f}=f$ and $\tilde{\delta}=f\delta g$. In particular, when $[\delta, f]=0$ and $fg=1$ we have $\tilde{\delta}=\delta$. 
\end{remark}

\begin{remark}\label{rmk:Perturbation.special}
The assumption that $\varphi$ is special ensures us that $f\varphi=0$ and $\varphi g=0$. When $fg=1$ these two conditions are not essential. Indeed, if we set $\pi=gf$, then $\pi$ is compatible with the $\dl$-operator on $C_\bt$ and, as $fg=1$, we have $\pi^2=g(fg)f=\pi$. Therefore, if we set $\tilde{\varphi} =(1-\pi)\varphi (1-\pi)$, then we have
\begin{equation*}
 fg -1= (1-\pi)(fg -1) (1-\pi) =  (1-\pi)(\dl\varphi+\varphi \dl)  (1-\pi)= \dl\tilde{\varphi}+\tilde{\varphi} \dl. 
\end{equation*}
 We also have $f(1-\pi)=f-(fg)f=0$ and $(1-\pi)g=g-g(fg)=0$, and so $f\tilde{\varphi}=0$ and $\tilde{\varphi}g=0$. Thus, upon replacing $\varphi$ by $\tilde{\varphi}$ 
 we obtain a chain homotopy such that $f\varphi =0$ and $\varphi g=0$. 

Assume further that $\dl^2=0$ and we are given a chain homotopy such that $f\varphi=0$ and $\varphi g=0$. Under these assumptions the condition $\varphi^2=0$ in~(\ref{eq:Perturbation.special}) becomes inessential. Indeed, we then have $(\dl \varphi)^2=(1-\pi -\varphi \dl)\dl \varphi=\dl \varphi$. Likewise, $( \varphi \dl)^2= \varphi \dl$, and so $ \dl (\varphi \dl \varphi)+  (\varphi \dl \varphi)\dl = \dl \varphi+ \varphi \dl=1-\pi$. We also have $\dl \varphi^2 \dl=(1-\pi -\varphi \dl)(1-\pi -\dl \varphi )=1-\pi - \varphi \dl -\dl \varphi =0$, and hence $(\varphi \dl \varphi)^2=\varphi (\dl \varphi^2 \dl) \varphi=0$. Thus, upon replacing $\varphi$ by $\varphi \dl \varphi$ we obtain a chain homotopy such that $\varphi^2=0$. 

All this shows that, when $fg=1$ and $\dl^2=0$, we can convert $\varphi$ into a special homotopy by replacing it by 
\begin{equation*}
 \hat{\varphi}:=\tilde{\varphi}\dl \tilde{\varphi}=(1-\pi)\varphi \dl (1-\pi)\varphi (1-\pi). 
\end{equation*}
Note that when $f\varphi=0$ we simply have $\hat{\varphi}= \varphi \dl \varphi (1-\pi)$.

When $fg\neq 1$ or $\dl^2=0$ the above transformations need apply. Thus, whereas in the setting of the basic perturbation lemma the special homotopy assumption is not essential, it becomes essential for the generalization provided by Lemma~\ref{lem:perturbation-lemma-special}. 
\end{remark}

Further elaborating on the remarks above we would like to stress that, when $\Delta=0$ and $f$ is compatible with the $\delta$-operators, we actually can relax the assumption on $\varphi$ being special (compare~\cite{Ka:Crelle90}). Namely, we have the following result. 

\begin{lemma}\label{lem:perturbation.Delta-zero}
Suppose that $\Delta=0$ on $C_\bt$ and $\wbC_\bt$. Assume further that $[\delta,f]=f\varphi=0$ and $[\dl,g]=0$. Then we have 
\begin{gather}
 (\dl + \delta)\tilde{g}=\tilde{g} (\dl + \tilde{\delta}),
\label{eq:Perturbation.fg-chain-maps-D0}\\
 f\tilde{g}=fg, \qquad 
\tilde{g}f=1 + (\dl + \delta)\tilde{\varphi} + (\dl + \delta)\tilde{\varphi},
\label{eq:Perturbation.ftg-inverses-D0}\\
 \tilde{\delta}= \delta fg, \qquad  \tilde{\delta}^2+ \dl\tilde{\delta} + \tilde{\delta}\dl =0,
 \label{eq:Perturbation.fdelta-D0} \\  
 f\tilde{\varphi}=0.
 \label{eq:Perturbation.ftvphi-D0}
\end{gather}
\end{lemma}
\begin{proof} 
As $[\delta, f]=f\varphi=0$, we know by Remark~\ref{rmk:Perturbation.f-tdelta-compatible} that $\tilde{f}=f$ and $\tilde{\delta}=\delta fg$. Moreover, in the same way as in the proof of Lemma~\ref{lem:perturbation-lemma-special}, we have  $f\tilde{\varphi}=f\varphi (1+ \delta \tilde{\varphi})=0$, and so $f\delta \tilde{\varphi}=\delta f \tilde{\varphi}=0$.  Combining this with~(\ref{eq:Perturbation.tf}) we get $f\tilde{g}=\tilde{f}\tilde{g}= f(1+ \delta \tilde{\varphi})(1+\tilde{\varphi}\delta)g= fg$. 

By assumption $[\dl, g]=0$, and as $\Delta=0$ the equalities~(\ref{eq:Perturbation.conditions}) are trivially satisfied.  Therefore, by Remark~\ref{rmk:Perturbation.Delta-0} we have (\ref{eq:Perturbation.tg-chain-map}), (\ref{eq:Perturbation.homotopy-tgtf}) and~(\ref{eq:Perturbation.tdelta-fdeltag}). 
As $\tilde{f}=f$ and $\Delta=0$ this gives~(\ref{eq:Perturbation.fg-chain-maps-D0}) and the equalities $\tilde{g}f=\tilde{g}\tilde{f}=1 + (\dl + \delta)\tilde{\varphi} + (\dl + \delta)\tilde{\varphi}$ and $ \tilde{\delta}^2+ \dl\tilde{\delta} + \tilde{\delta}\dl =f\Delta g  =0$. The proof is complete. 
\end{proof}

\begin{remark}\label{rmk:perturbation.Delta-zero-fg=1}
 When $fg=1$ we have $f\tilde{g}=fg=1$ and $\tilde{\delta}=\delta fg=\delta$, and so $(\dl + \delta)\tilde{g}=\tilde{g} (\dl + \delta)$. 
 Thus, if we further have the compatibility of $f$ with the $\dl$-operators, then we obtain a deformation retract of $(C_\bt, \dl+\delta)$ to $(\wb{C}_\bt,  \dl+\delta)$.  
\end{remark}

\subsection{Perturbation lemmas for parachain complexes} 
We shall now specialize Lemma~\ref{lem:perturbation-lemma-special} and Lemma~\ref{lem:perturbation.Delta-zero} to parachain complexes.  This will enable us to construct co-extensions of deformation retracts of Hochschild complexes.  In particular, this will generalize to parachain complexes the version of the basic perturbation lemma for mixed complexes of Kassel~\cite{Ka:Crelle90} (see also~\cite{Ba:Preprint98}). 
 
Suppose that $C=(C_\bt, b, B)$ and $\wbC=(\wbC_\bt, b,B)$ are parachain complexes. Any $S$-map $f:C^\natural_\bt\rightarrow \wbC^\natural_\bt$ is of the form $f=f^{(0)}u^0+f^{(1)}u+\cdots $, where $f^{(j)}:C_{\bt}\rightarrow \wbC_{\bt+2j}$ are $T$-compatible $k$-linear maps satisfying~(\ref{eq:Parachain.S-map}). In particular, the zero-th degree component  $f^{(0)}:C_{\bt}\rightarrow \wbC_{\bt}$ is a Hochschild chain map, i.e., $[b,f^{(0)}]=0$. Conversely,  given any $T$-compatible Hochschild chain map $f: C_\bt \rightarrow \wbC_\bt$,  we shall  call \emph{co-extension} of $f$  any $S$-map $f^\natural :C^\natural_\bt \rightarrow \wbC^\natural_\bt$ such that $f^{\natural(0)}=f$. This terminology was coined by Hood-Jones~\cite{HJ:KT87}.  

Throughout the rest of this section we assume we are given Hochschild chain maps $f:C_\bt \rightarrow \wb{C}_\bt$ and $g:\wb{C}_\bt \rightarrow C_\bt$ 
that provide us with a deformation retract, 
\begin{equation}
 fg=1, \qquad gf =1 +b\varphi +\varphi b,
 \label{eq:Perturbation.DR}
\end{equation}
where $\varphi: C_\bt \rightarrow C_{\bt+1}$ is some $k$-linear map. We seek for applying Lemma~\ref{lem:perturbation-lemma-special}  to the the para-twin complexes $(C_\bt^\natural, b, Bu^{-1})$ and $(\wbC_\bt^\natural, b, Bu^{-1})$. Note that the nilpotency of $\delta =Bu^{-1}$ ensures us that the condition~(\ref{eq:perturbation.delta-vphij}) is automatically satisfied.  

Let $\varphi^\natural :C_\bt^\natural  \rightarrow C_{\bt+1}^\natural$ be the $k$-linear map defined by 
\begin{equation*}
\varphi^\natural  = \sum_{j\geq 0} \varphi (B\varphi)^j u^{-j} =  \varphi + \varphi B \varphi u^{-1} + \cdots.  
\end{equation*}
We also define $k$-linear maps $f^\natural :C_\bt^\natural \rightarrow \wbC_\bt^\natural$ and $g^\natural:\wbC_\bt^\natural \rightarrow C_\bt^\natural$ by 
\begin{gather*}
 f^\natural = f + f B \varphi^\natural u^{-1} = f + \sum_{j \geq 1} f(B \varphi)^j u^{-j},\\
  g^\natural = g +   \varphi^\natural Bg u^{-1} = g + \sum_{j \geq 1} ( \varphi B)^jg u^{-j}. 
\end{gather*}
Note that the maps $f^\natural$ and $g^\natural$ are compatible with the $S$-operators and their zero-th degree componenents are $f$ and $g$, respectively. In addition, we let $\tilde{B}: \wbC_\bt \rightarrow \wbC_{\bt+1}$ be the $k$-linear map defined by
\begin{equation*}
 \tilde{B}= \sum_{j\geq 0} \varphi f(B\varphi)^jBg u^{-j} = fBg + fB\varphi Bg + \cdots. 
\end{equation*}

By assumption the chain homotopy $\varphi$ is special. Note also that if $\dl =b$ and $\delta =Bu^{-1}$, then in the notation of~(\ref{eq:Perturbation.tvarphi})--(\ref{eq:Perturbation.tdelta0}) the maps $(\varphi^\natural, f^\natural, g^\natural, \tilde{B}u^{-1})$ are precisely the maps $(\tilde{\varphi},\tilde{f},\tilde{g},\tilde{\delta})$. 
Moreover, we have
\begin{equation}
 \Delta=\delta^2+\dl \delta+\delta\dl=B^2u^{-2}+(bB+Bb)u^{-1}=(1-T)u^{-1}. 
 \label{eq:Perturbation.Delta-parachain}
\end{equation}
Thus, the compatibility with the $\Delta$-operators follows from the compatibility with the $T$-operators. Therefore, we have the following version of  Lemma~\ref{lem:perturbation-lemma-special}. 
 
\begin{lemma}\label{lem:Perturbation.parachain-complexes}
 Suppose that the maps $(f,g,\varphi)$ are $T$-compatible and the chain homotopy $\varphi$ is special. Then
 \begin{enumerate}
 \item[(i)] $\tilde{C}^\natural:=(\wbC_\bt^\natural, b+\tilde{B}u^{-1}, u^{-1}, T)$ is a para-$S$-module.   

 \item[(ii)] $f^\natural: C_\bt^\natural \rightarrow \tilde{C}_\bt^\natural$ and $g^\natural: \tilde{C}_\bt^\natural \rightarrow \tilde{C}_\bt^\natural$ are $S$-maps such that
 \begin{gather*}
 f^\natural g^\natural =1, \qquad g^\natural f^\natural =1 + (b+Bu^{-1}) \varphi^\natural +   \varphi^\natural (b+Bu^{-1}). 
\end{gather*}
In particular, we obtain an $S$-deformation retract of $C^\natural$ to $\tilde{C}^\natural$, and a $T$-deformation retract of $C^\sharp$ to  $\tilde{C}^\sharp$. 

\item[(iii)] The chain homotopy $\varphi^\natural$ is special.
\end{enumerate}
\end{lemma}
\begin{proof}
The $T$-compatibility of the maps $(f,g,\varphi)$ ensure us the $T$-compatibility of the maps $(f^\natural, g^\natural, \varphi^\natural)$. Granted this, the parts (ii) and (iii) then follow from the last two parts of Lemma~\ref{lem:perturbation-lemma-special} and from Remark~\ref{rmk:Perturbation.fg}. Moreover, by using~(\ref{eq:Perturbation.tdelta-fdeltag-fg=1}) and~(\ref{eq:Perturbation.Delta-parachain}) and the fact that $b^2=0$ we get
\begin{equation}
 (1-T)u^{-1}= \tilde{B}^2 u^{-2} + (b\tilde{B}+\tilde{B}b)u^{-1}= \big(b+\tilde{B}u^{-1}\big)^2. 
 \label{eq:Perturbation.d-tB2}
\end{equation}
As $\tilde{B}$ is compatible with $u^{-1}$ and $T$ it then follows that $(\wb{C}_\bt^\natural, b+\tilde{B}u^{-1}, u^{-1}, T)$ is a para-$S$-module. The proof is complete.
\end{proof}

\begin{remark}
 When $C$ and $\wbC$ are mixed complexes (i.e., $bB+Bb=0$) the para-$S$-module $\tilde{C}^\natural$ is actually an $S$-module, since in this case $T=1-(bB+Bb)=1$. In particular, we recover the version of the basic perturbation lemma for mixed complexes of Kassel~\cite{Ka:Crelle90} (see also~\cite[Lemma~I.2]{Ba:Preprint98}).
\end{remark}

\begin{remark}\label{rmk:Perturbation.d-tB-parachain}
 Write $B=\sum_{j\geq 0} B^{(j)}u^{-j}$, with $B^{(j)}:= f (B\varphi)^j Bg$. Then~(\ref{eq:Perturbation.d-tB2}) implies that
 \begin{equation*}
 bB^{(0)}+B^{(0)}b=1-T, \qquad \sum_{p+q=j-1} B^{(p)}B^{(q)} +  bB^{(j)}+B^{(j)}b=0, \quad j\geq 1. 
\end{equation*}
Thus, if $B^{(j)}=0$ for $j \geq 1$, then $(\wb{C}, b, B^{(0)})$ is a parachain complex whose para-$S$-module is precisely $\tilde{C}^\natural$. 
\end{remark}

\begin{remark}
 In general, $\tilde{C}^\natural$ need not be the para-$S$-module of a parachain complex. This shows the relevance of considering para-$S$-modules when attempting to apply perturbation lemmas in the setting of parachain complexes.
\end{remark}

Further elaborating on Remark~\ref{rmk:Perturbation.d-tB-parachain} we have the following result. 

 \begin{lemma}\label{lem:Lifting.co-extension-deformation-retract-fBg=B}
 Suppose that the maps $(f,g,\varphi)$ are $T$-compatible and the chain homotopy $\varphi$ is special. Assume further that
 \begin{equation}
 fBg=B \qquad \text{and} \qquad f (B\varphi)^j Bg=0 \quad \text{for $j\geq 1$}.
 \label{eq:Perturbation.conditions-Bj=0}  
\end{equation}
 Then the following holds. 
  \begin{enumerate}
 \item[(i)]  $ f^\natural:C_\bt^\natural \rightarrow \wb{C}^\natural_\bt$ and $g^\natural: \wb{C}_\bt^\natural \rightarrow C^\natural_\bt$ are $S$-maps and coextensions of $f$ and $g$, respectively. 

 \item[(ii)] We have 
 \begin{equation*}
 f^\natural g^\natural =1, \qquad g^\natural f^\natural =1 + (b+Bu^{-1}) \varphi^\natural +   \varphi^\natural (b+Bu^{-1}). 
\end{equation*}
In particular, we obtain an $S$-deformation retract of $C^\natural$ to $\wb{C}^\natural$ and a $T$-deformation retract of $C^\sharp$ to $\wb{C}^\sharp$. 

\item[(iii)] The chain homotopy $\varphi^\natural$ is special.
\end{enumerate}
\end{lemma}
\begin{proof}
The assumptions~(\ref{eq:Perturbation.conditions-Bj=0}) and Remark~\ref{rmk:Perturbation.d-tB-parachain} ensure us that the para-$S$-module $\tilde{C}^\natural$ of the first part of Lemma~\ref{lem:Perturbation.parachain-complexes} agrees with $\wb{C}^\natural$. The result then follows from the last two parts of Lemma~\ref{lem:Perturbation.parachain-complexes}. 
\end{proof}

When $f$ is already a map of parachain complexes we have the following statement. 

\begin{lemma}\label{lem:Lifting.co-extension-deformation-retract-f-parachain}
 Suppose that  $f$ is a map of parachain complexes,  the maps $(g,\varphi)$ are $T$-compatible, and the chain homotopy $\varphi$ is special. Then
 \begin{enumerate}
 \item[(i)]  $g^\natural: \wbC_\bt^\natural \rightarrow C_\bt^\natural $ is an $S$-map and a co-extension of $g$. 
 
 \item[(ii)]  We have
 \begin{equation*}
  f g^\natural =1, \qquad g^\natural f =1 + (b+Bu^{-1}) \varphi^\natural +   \varphi^\natural (b+Bu^{-1}). 
\end{equation*}
In particular, this provides us with an $S$-deformation retract of $C^\natural$ to $\wbC^\natural$ and a $T$-deformation retract of $C^\sharp$ to $\wb{C}^\sharp$.
  
\item[(iii)] The chain homotopy $\varphi^\natural$ is special.
\end{enumerate}
\end{lemma}
\begin{proof}
 The fact that $f$ is a map of parachain complexes ensures us it is compatible with the operators $T=1-(bB+Bb)$ and $\delta=Bu^{-1}$. As $fg=1$, we can argue in the same way as in Remark~\ref{rmk:Perturbation.f-tdelta-compatible} to show that $f^\natural =f $ and the conditions~(\ref{eq:Perturbation.conditions-Bj=0}) are satisfied. The result then follows from Lemma~\ref{lem:Lifting.co-extension-deformation-retract-fBg=B}. 
\end{proof}

Finally, when $C$ and $\wb{C}$ are mixed complexes (i.e., when $T=1$) we can  relax the assumptions on the chain homotopy $\varphi$ (compare~\cite{Ba:Preprint98, Ka:Crelle90}).  Namely, by using Lemma~\ref{lem:perturbation.Delta-zero} and Remark~\ref{rmk:perturbation.Delta-zero-fg=1} we obtain the following result. 

\begin{lemma}\label{lem:Lifting.co-extension-deformation-retract-f-mixed}
Suppose that $C$ and $\wbC$ are mixed complexes. Assume further that $f$ is a map of mixed complexes and $f\varphi=0$. Then
 \begin{enumerate}
 \item[(i)]  $g^\natural: \wbC_\bt^\natural \rightarrow C_\bt^\natural $ is an $S$-map and a co-extension of $g$. 
 
 \item[(ii)]  We have
 \begin{equation*}
  f g^\natural =1, \qquad g^\natural f =1 + (b+Bu^{-1}) \varphi^\natural +   \varphi^\natural (b+Bu^{-1}), \qquad f\varphi^\natural=0. 
\end{equation*}
This provides us with an $S$-deformation retract of $C^\natural$ to $\wbC^\natural$ and a $T$-deformation retract of $C^\sharp$ to $\wb{C}^\sharp$. 
\end{enumerate}
\end{lemma}

\subsection{Some Applications} 
It is well known that a map of mixed complexes is a quasi-isomorphism at the cyclic chain level if and only if it is an isomorphism at the ordinary chain level (see, e.g., \cite{Lo:CH}). As a first application of the results of this section we shall prove the following version of that result for parachain complexes. 

\begin{proposition}
Suppose that $C=(C_\bt, b,B)$ and $\wb{C}=(\wb{C}_\bt, b,B)$ are parachain complexes, and let $f:C_\bt \rightarrow \wb{C}_\bt$ be a parachain complex map. Then the following are equivalent: 
 \begin{enumerate}
   \item[(i)] $f$ gives rise to a $T$-deformation retract of $(C_\bt, b)$ to $(\wb{C}_\bt, b)$.

    \item[(ii)] $f$ gives rise to an $S$-deformation retract of $C^\natural$ to $\wb{C}^\natural$. 
\end{enumerate}
Furthermore, if (i) and (ii) hold, then we get a $T$-deformation retract of $C^\sharp$ to $\wb{C}^\sharp$. 
\end{proposition}
\begin{proof}
 Suppose that $f$ gives rise to a $T$-deformation retract of $(C_\bt, b)$ to $(\wb{C}_\bt, b)$. That is, there is a $T$-compatible chain map 
 $g:(\wb{C}_\bt, b)\rightarrow ( C_\bt, b)$ and a $T$-compatible $k$-linear map $\varphi:C_\bt \rightarrow C_{\bt+1}$ satisfying~(\ref{eq:Perturbation.DR}). 
 As pointed out in Remark~\ref{rmk:Perturbation.special}, we may assume that the chain homotopy $\varphi$ is special, since $fg=1$ and $b^2=0$. 
 Lemma~\ref{lem:Lifting.co-extension-deformation-retract-f-parachain} then produces a coextension  $g^\natural:\wb{C}^\natural_\bt \rightarrow C_\bt^\natural$ which is a right-inverse of $f$ on $\wb{C}^\natural$ and a $S$-homotopy left-inverse on $C^\natural_\bt$. 
 This gives an $S$-deformation retract of $C^\natural$ to $\wb{C}^\natural$, and so by Proposition~\ref{prop:Para-S-Mod.periodic-homotopy} we obtain a $T$-deformation retract of $C^\sharp$ to $\wb{C}^\sharp$. 
 
 Conversely, suppose there are an $S$-map $g^\natural:\wb{C}^\natural_\bt \rightarrow C_\bt^\natural$ and an $(S,T)$-compatible $k$-linear map $\varphi^\natural: C_\bt^\natural \rightarrow C_{\bt+1}^\natural$ such that
\begin{equation}
 fg^\natural=1, \qquad g^\natural f = 1+ (b+Bu^{-1}) \varphi^\natural +   \varphi^\natural (b+Bu^{-1}). 
 \label{eq:Perturbation.S-deformation-g}
\end{equation}
 Set $g^\natural = \sum_{j\geq 0} g^{(j)} u^{-j}$ and $ \varphi^\natural   = \sum_{j\geq 0} \varphi^{(j)} u^{-j}$, where $g^{(j)}:C_{\bt}\rightarrow C_{\bt+2j}$ and $\varphi^{(j)}:C_\bt \rightarrow C_{\bt+2j+1}$ are $T$-compatible $k$-linear maps. 
 We observe that the zeroth degree component parts of $fg^\natural$, $g^\natural f$, and $ 1+ (b+Bu^{-1}) \varphi^\natural +   \varphi^\natural (b+Bu^{-1})$ are $fg^{(0)}$, $g^{(0)} f$, and $ 1+ b\varphi^{(0)} +   \varphi^{(0)} b$, respectively. 
 Therefore, taking zeroth degree component parts in both equalities in~(\ref{eq:Perturbation.S-deformation-g}) gives $fg^{(0)}=1$ and $g^{(0)} f= 1+ b\varphi^{(0)} +   \varphi^{(0)} b$. In particular, we get a $T$-deformation retract of $(C_\bt, b)$ to $(\wb{C}_\bt, b)$. The proof is complete. 
\end{proof}

Applying the above result to the canonical projection $\pi_T:C_\bt \rightarrow C_{T,\bt}$ of a parachain complex provides us with the following characterization of property (DR) for para-$S$-modules associated with parachain complexes. 

\begin{corollary}
 Let $C=(C_\bt, b,B)$ be a parachain complex. Then the following are equivalent:
  \begin{enumerate}
    \item[(i)] The para-$S$-module $C^\natural$ has property (DR). 
      
 \item[(ii)] The canonical projection $\pi_T:C_\bt \rightarrow C_{T,\bt}$ gives rise to a $T$-deformation retract of $(C_\bt, b)$ to $(C_{T,\bt}, b)$. 
\end{enumerate}
\end{corollary}

\begin{remark}
 A parachain complex $(C_\bt,b,B)$ such that the canonical projection $\pi_T:(C_\bt, b)\rightarrow (C_{T,\bt}, b)$ is a quasi-isomorphism is called a \emph{homological skycraper} in~\cite{KM:HHA18}. 
\end{remark}

\section{Comparison of the Para-$S$-Modules $C^\nnatural$ and $C^\natural$}\label{sec:Cn-Cnn}
When $C$ is a cyclic module, Connes~\cite{Co:CRAS83} showed that $C^\nnatural$ and $C^\natural$ are quasi-isomorphic chain complexes (see also~\cite{Lo:CH, LQ:CMH84}). Subsequently, Kassel~\cite{Ka:Crelle90} exhibited a deformation retract of $C^\nnatural$ to $C^\natural$ when $C$ is an $H$-unital precyclic module. In addition, he pointed out that his deformation retract could be obtained by using the basic perturbation lemma, and the $B$-operator would naturally re-appear from this process (\emph{cf}.\ \cite[Remarque~5.3]{Ka:Crelle90}). 

In this section, by elaborating on Kassel's observation we shall extend to $H$-unital para-precyclic modules the aforementioned equivalence results. This will use the generalization of the basic perturbation lemma given by Lemma~\ref{lem:perturbation-lemma-special}. In particular, as we shall see, and further confirming Kassel's observation, the $B$-differential~(\ref{eq:parachain-paracyclic.B}) naturally arises from the perturbation of the deformation retract of $(C_\bt^\nnatural, \delta)$ to $(C^\natural_\bt, b)$. 

The deformation retract of $(C_\bt^\nnatural, \delta)$ to $(C^\natural_\bt, b)$ is obtained as follows.  We have a natural $k$-module embedding $I_0:C_\bt^\natural \rightarrow C^{\nnatural}_\bt$ given by 
\begin{equation*}
 I_0(xu^p)= xu^{2p}, \qquad x \in C_\bt. 
\end{equation*}
As $ I_0b(xu^p)=I_0(bxu^p)=bxu^{2p}=\delta(xu^{2p})=\delta I_0(xu^p)$, we see that $I_0$ is a chain map from $(C^\natural_\bt,b)$ to $(C^\nnatural_\bt,\delta)$. 

We also have a natural projection $J_0:C_\bt^\nnatural \rightarrow C^\natural_\bt$ given by
\begin{equation*}
 J_0(xu^2p)=xu^p, \qquad J_0(xu^{2p+1})=0. 
\end{equation*}
As $J_0\delta(xu^{2p})=J_0(bxu^{2p})=bxu^p=bJ_0(xu^{2p})$ and $J_0\delta(xu^{2p+1})=J_0(bxu^{2p+1})=0=bJ_0(xu^{2p+1})$, we also see that $J_0$ is a chain map from $(C^\nnatural_\bt,\delta)$ to $(C^\natural_\bt,b)$. 

Let $h:C^\nnatural_\bt \rightarrow C^\nnatural_{\bt+1}$ be the $k$-linear map defined by
\begin{equation*}
 h(xu^{2p})=0, \qquad h(xu^{2p+1})=s'xu^{2p+1}. 
\end{equation*}

\begin{lemma}\label{lem:Cnn-Cn.homotopy-I0J0}
 We have 
\begin{equation}
 J_0I_0=1, \qquad I_0J_0=1+ \delta h + h \delta.
 \label{eq:Cnn-Cn.homotopy-I0J0}
\end{equation}
In particular, we get a deformation retract of $(C^\nnatural_\bt,\delta)$ to $(C^\natural_\bt,b)$. In addition, the chain homotopy $h$ is special. 
\end{lemma}
\begin{proof}
 Let $x\in C_\bt$. We have $J_0I_0(xu^p)=J_0(xu^p)=1$, and so $J_0I_0=1$.  Likewise, we have $I_0J_0(xu^{2p})=I_0(xu^{p})=xu^{2p}$. As $h=0$ on $C_\bt u^{2p}$, we also get $(\delta h +h\delta)(xu^{2p})=h[bxu^{2p}]=0$. Thus, 
 \begin{equation}
 (1+\delta h +h\delta)(xu^{2p})=xu^{2p}=I_0J_0(xu^{2p}).
 \label{eq:Cnn-Cn.homotopy0}
\end{equation}
Note also that $I_0J_0(xu^{2p+1})=0$. Moreover, we have 
\begin{equation*}
 (\delta h + h\delta )(xu^{2p+1}) = \delta [s'x u^{2p+1}] - h [b'xu^{2p+1}]  = -(b's'+s'b')xu^{2p+1} = -xu^{2p+1}. 
\end{equation*}
Therefore, we see that $(1+\delta h +h\delta)(xu^{2p+1})=0=I_0J_0(xu^{2p+1})$. 
Combining this with~(\ref{eq:Cnn-Cn.homotopy0}) gives the homotopy formula $I_0J_0=1+ \delta h + h \delta$. 

In addition, we have $h^2(xu^{2p})=0$ and $h^2(xu^{2p+1})=(s')^2xu^{2p+1}=0$, and so $h^2=0$. As $hI_0(xu^{p})=h(xu^{2p})=0$, we also see that $hI_0=0$. In addition, we have $J_0h(xu^{2p})=0$ and $J_0h(xu^{2p+1})=J_0(s'xu^{2p+1})=0$, and so $J_0h=0$. All this shows that the chain homotopy $h$ is special. The proof is complete. 
\end{proof}

We seek for applying Lemma~\ref{lem:perturbation-lemma-special} to the deformation retract~(\ref{eq:Cnn-Cn.homotopy-I0J0}). That is, for the para-twin complexes $(C^\nnatural_\bt, \delta, \dl)$ and $(C^\natural_\bt, b, Bu^{-1})$, where the respective roles of the maps $(f,g,\varphi)$ are played by the maps  $(J_0, I_0, h)$. As $(C^\natural_\bt, b, Bu^{-1})$ is a parachain complex, its $\Delta$-operator~(\ref{eq:perturbation.Delta}) is $(1-T)u^{-1}$. Using Lemma~\ref{lem:CCparacyclic.square} we also see that the $\Delta$-operator of $(C^\nnatural_\bt, \delta, \dl)$ is given by
\begin{equation*}
 \Delta = \dl^2+ \delta \dl+\dl \delta= (1-T)u^{-2}. 
\end{equation*}
It is immediate from their definitions that the maps $(J_0, I_0, h)$ are $T$-compatible. It is also straightforward to check that $J_0u^{-2}=u^{-1}J_0$, $I_0u^{-1}=u^{-2}I_0$ and $hu^{-2}=u^{-2}h$. Therefore, we see that the maps $(J_0, I_0, h)$ are $\Delta$-compatible. Together with Lemma~\ref{lem:Cnn-Cn.homotopy-I0J0} this allows us to apply Lemma~\ref{lem:perturbation-lemma-special}. 

Let $(\tilde{h}, \tilde{J}, \tilde{I}, \tilde{\dl})$ be the maps~(\ref{eq:Perturbation.tvarphi})--(\ref{eq:Perturbation.tf}) and (\ref{eq:Perturbation.tdelta0}) associated with $(h, J_0, I_0, \dl)$. Namely, 
\begin{equation}
 \tilde{h}= \sum_{j\geq 0} h(\dl h)^j, \quad \tilde{J} = J_0(1+\dl \tilde{h}), \quad \tilde{I}=(1+ \tilde{h}\dl)I_0, \quad \tilde{\dl}=J_0(\dl +\dl \tilde{h}\dl)I_0. 
 \label{eq:CnnCn.perturbation-hJLdl}
\end{equation}
By Lemma~\ref{lem:perturbation-lemma-special} these maps provide us with a deformation retract from $(C^\nnatural_\bt, \delta +\dl)$ to $(C^\natural_\bt, b +\tilde{\dl})$. It just remains to identify them.

Let $I:C_\bt^\natural \rightarrow C^\nnatural_\bt$ be the $k$-linear map defined by
\begin{equation}
 I(xu^0)=xu^0, \qquad I(xu^p)=xu^{2p} + s'N xu^{2p-1}, \qquad p\geq 1.
 \label{eq:Cnn-Cn.I-map}
\end{equation}
We also define the $k$-linear map $J:C_\bt^{\nnatural} \rightarrow C_\bt^\natural$ by
\begin{equation}
 J(xu^{2p})= xu^p, \qquad J(xu^{2p+1})= (1-\tau)s'x u^p,  \qquad  p \geq 0. 
  \label{eq:Cnn-Cn.J-map}
\end{equation}

\begin{lemma}\label{lem:CnnCn.computation-thJLdl}
 We have
 \begin{equation*}
 \tilde{h}=h, \qquad \tilde{J}=J, \qquad \tilde{I}=I, \qquad \tilde{\dl}=Bu^{-1}.  
\end{equation*}
\end{lemma}
\begin{proof}
Let $x\in C_\bt$. We have 
\begin{equation}
 \dl h(xu^{2p})=0 \qquad \text{and} \qquad \dl h(xu^{2p+1})=\dl (s'xu^{2p+1})=(1-\tau)s'xu^{2p}. 
 \label{eq:CnnCn.dlh}
\end{equation}
This implies that $h\dl h(xu^{2p})=0$ and $h\dl h(xu^{2p+1})=h[(1-\tau)s'xu^{2p}]=0$, and hence $h \dl h=0$. It then follows that $\tilde{h}=\sum_{j\geq 0} h(\dl h)^j=h$. As a result we can substitute $h$ for $\tilde{h}$ in the definitions of $(\tilde{I}, \tilde{J}, \tilde{\dl})$ in~(\ref{eq:CnnCn.perturbation-hJLdl}). 

By using~(\ref{eq:CnnCn.dlh}) we get $\tilde{J}(xu^{2p})= J_0(1+\dl h)(xu^{2p})=J_0(xu^{2p})=xu^{p}$. We also get 
\begin{equation*}
 \tilde{J}(xu^{2p})=J_0(1+\dl h)(xu^{2p+1})=J_0\dl\left[s'xu^{2p+1}\right]=J_0\left[(1-\tau)s'xu^{2p}\right]=(1-\tau)s'xu^{p}.  
\end{equation*}
It then follows that $\tilde{J}=J$.  
We also have $h\dl I_0(xu^0)=h\dl (xu^0)=0$. If $p\geq 1$, then $h\dl I_0(xu^p)=h\dl (xu^p)=0$. If $p\geq 1$, then
\begin{equation}
 h\dl I_0\left(xu^{p}\right)= h\dl\left(xu^{2p}\right)= h\left( Nxu^{2p-1}\right) =s'Nxu^{2p-1}. 
 \label{eq:Cnn-Cn.hdlI0}
\end{equation}
This shows that $h\dl I_0= s'Nu^{-1}I_0u^{-1}$, and so $\tilde{I}=(1+h \dl)I_0=(1+s'Nu^{-1})I_0=I$. 

We also have $J_0 \dl I_0(xu^p)=J_0\dl(xu^{2p})=J_0(Nxu^{2p-1})=0$, and so $J_0\dl I_0=0$. In addition, we have $J_0 \dl h \dl I_0(xu^0)= J_0 \dl h \dl (xu^0)=0=Bu^{-1}(xu^0)$. If $p\geq 1$, then by~(\ref{eq:Cnn-Cn.hdlI0}) we have $h\dl I_0(xu^{p})= s'Nxu^{2p-1}$, and so we get
\begin{equation*}
 J_0 \dl h \dl I_0(xu^p)=J_0 \dl \left (s'Nxu^{2p-1}\right)=J_0\left[ (1-\tau)s'Nxu^{2p-1}\right]=Bxu^{p-1}.   
\end{equation*}
Therefore, we see that $J_0 \dl h \dl I_0=Bu^{-1}$. Thus, 
\begin{equation*}
 \tilde{\dl}= J_0 (\dl+ \dl h \dl)I_0=  J_0 \dl I_0 + J_0 \dl h \dl I_0=Bu^{-1}. 
\end{equation*}
The proof is complete. 
\end{proof}

We are now in a position to prove the main result of this section. 

\begin{proposition}\label{prop:CnnCn.deformation-retract}
Let $C=(C_\bt, d,s,t)$ be an $H$-unital para-precyclic $k$-module. 
 \begin{enumerate}
\item The maps $I:C^\natural_\bt \rightarrow C^\nnatural_\bt$ and $J:C^\nnatural_\bt \rightarrow C^\natural_\bt$ given by~(\ref{eq:Cnn-Cn.I-map})--(\ref{eq:Cnn-Cn.J-map}) are $S$-maps. 

\item We have 
\begin{equation}
 JI=1, \qquad IJ=1+(\dl+\delta)h+h(\dl+\delta). 
 \label{lem:CnnCn.deformation-retract-IJ}
\end{equation}
This provides us with an $S$-deformation retract of $C^\nnatural$ to $C^\natural$.

\item The chain homotopy $h$ is special (i.e., $J h=0$, $h I=0$, and $h^2=0$).  
\end{enumerate}
\end{proposition}
\begin{proof}
 It follows from Lemma~\ref{lem:perturbation-lemma-special} and Lemma~\ref{lem:CnnCn.computation-thJLdl} that $I$ and $J$ are chains maps giving rise to the deformation retract~(\ref{lem:CnnCn.deformation-retract-IJ}). Moreover, the chain homotopy $h$ is special, in the sense that $J h=0$, $h I=0$, and $h^2=0$.  As mentioned above, the chain homotopy $h$ is compatible with the operators $u^{-2}$ and $T$ on $C^\nnatural_\bt$. It follows from their definitions that $I$ and $J$ are $T$-compatible maps. Moreover, it can be checked that $Iu^{-1}=u^{-2}I$ and $Ju^{-2}=u^{-1}J$. Therefore, the maps $I$ and $J$ are $S$-maps and the deformation retract~(\ref{lem:CnnCn.deformation-retract-IJ}) is an $S$-deformation retract of $C^\nnatural$ to $C^\natural$. 
 The proof is complete. 
\end{proof}

\begin{remark}
When $C$ is an $H$-unital  precyclic module Proposition~\ref{prop:CnnCn.deformation-retract} was proved by Kassel~\cite{Ka:Crelle90}. 
\end{remark}

\begin{remark}
 It was shown by Loday-Quillen~\cite{LQ:CMH84} that the chain map $I:C^\natural_\bt \rightarrow C^\nnatural_\bt$ given by~(\ref{eq:Cnn-Cn.I-map}) is a quasi-isomorphism when $C$ is the cyclic module of an unital $k$-algebra and $k$ is commutative. 
\end{remark}

\section{Comparison of  $C^\nnatural$ and $C^\lambda$}\label{sec:Cnn-Cl} 
When $C$ is a precyclic $k$-module with $k\supset \Q$, Kassel~\cite{Ka:Crelle90} used a version of the basic perturbation lemma to construct a deformation retract of the chain complex $C^\nnatural$ to Connes' cyclic complex $C^\lambda$. In this section, we seek for a similar result for para-precyclic modules by using the perturbation theory of Section~\ref{sec:Perturbation}. In particular, this approach avoids using the cyclic relation $\tau^{m+1}=1$, which is used in~\cite{Ka:Crelle90}, but is not available in general with para-precyclic modules.  

Recall that if  $C=(C_\bt, d,t)$ is a precyclic $k$-module, then its cyclic chain complex in the sense of Connes~\cite{Co:MFO81, Co:CRAS83, Co:IHES85} is $C^\lambda=(C_\bt^\lambda,b)$, where
\begin{equation}
 C_m^\lambda = C_m \slash \ran(1-\tau), \qquad m\geq 0.
 \label{eq:Cnnat-Clambda.Clambda}
\end{equation}
Here $\tau$ is given by~(\ref{eq:parachain-paracyclic.tau-N}) and the differential $b$ is induced from the Hochschild differential $b:C_\bt \rightarrow C_{\bt-1}$. Indeed, as $b(1-\tau)=(1-\tau)b'$ this operator descends to a unique $k$-linear differential $b:C_\bt^\lambda \rightarrow C_{\bt-1}^\lambda$.
More generally, if $C=(C_\bt, d,t)$ is any para-precyclic $k$-module, then we can define the $k$-modules $C_m^\lambda$, $m\geq 0$, as in~(\ref{eq:Cnnat-Clambda.Clambda}). As we still have the relation $b(1-\tau)=(1-\tau)b'$, the Hoschschild differential descends to a $k$-linear differential  $b:C_\bt^\lambda \rightarrow C_{\bt-1}^\lambda$, and so we get a chain complex $C^\lambda:=(C_\bt^\lambda,b)$. 

In what follows, we assume we are given a para-precyclic $k$-module $C=(C_\bt, d,t)$. We let $\pi^\lambda:C_\bt \rightarrow C_\bt^\lambda$ be the canonical projection of $C_\bt$ onto $C_\bt^\lambda$. We also let $\pi_0^\nnatural:C_\bt^\nnatural \rightarrow C_\bt$ be the projection onto the zeroth degree component $C_\bt u^0=C_\bt$. That is, $\pi_0^\nnatural (xu^0)=x$ and $\pi_0^\nnatural (xu^p)=0$ for $p\geq 1$. Set $\pi^\nnatural:=\pi^\lambda \pi_0^\nnatural$; this is the  $k$-linear map from  $C_\bt^\nnatural$ to $C_{\bt}^\lambda$ such that
\begin{equation*}
 \pi^\nnatural(xu^0)=x^\lambda, \qquad \pi^\nnatural(xu^p)=0, \quad p\geq 1.
\end{equation*}

\begin{lemma}\label{lem:Cn-Cl.pinn-chain-map}
 We have 
 \begin{equation*}
 \pi_0^\nnatural \dl=(1-\tau) \pi_0^\nnatural u^{-1},  \qquad  \pi^\nnatural_0 \delta =b\pi^\nnatural, \qquad 
  \pi^\nnatural \dl=0,  \qquad  \pi^\nnatural \delta =b\pi^\nnatural.
\end{equation*}
In particular, the projection $\pi^\nnatural: C^\nnatural \rightarrow C^\lambda$ is a chain map. 
\end{lemma}
\begin{proof}
Let $x\in C_\bt$. If $p\geq 2$, then $\dl (xu^p)$ and $u^{-1}(xu^p)=xu^{p-1}$ are both contained in $\oplus_{q\geq 1}C_\bt u^q =\ker \pi^\nnatural_0$, and so 
$ \pi_0^\nnatural \dl(xu^p)=0=(1-\tau) \pi_0^\nnatural u^{-1}(xu^p)$. As $\dl(xu^0)=0=u^{-1} (xu^0)$ we also have $ \pi_0^\nnatural \dl(xu^0)=0=(1-\tau) \pi_0^\nnatural u^{-1}(xu^0)$. In addition, by definition $\dl(xu)=(1-\tau)xu^0$, and so we have $\pi_0^\nnatural \dl(xu)= (1-\tau)x =(1-\tau) \pi_0^\nnatural (xu^0) =(1-\tau) \pi_0^\nnatural u^{-1}(xu)$. 
Therefore, we see that $ \pi_0^\nnatural \dl=(1-\tau) \pi_0^\nnatural u^{-1}$, and hence $ \pi^\nnatural \dl=\pi^\lambda (1-\tau) \pi_0^\nnatural u^{-1}=0$. 
 
Given any  $x\in C_\bt$ we have $ \pi^\nnatural_0 \delta (xu^0) = \pi^\nnatural_0(bx u^0) = bx =  b\pi^\nnatural_0 (xu^0)$. Moreover, if $p\geq 1$, then $xu^p$ and $\delta (xu^p)$ are both contained in in $\oplus_{q\geq 1}C_\bt u^q =\ker \pi^\nnatural_0$, and hence $\pi^\nnatural_0 \delta (xu^p)=0= b \pi^\nnatural_0 (xu^p)$. This shows that $\pi^\nnatural_0 \delta =b\pi^\nnatural$. As $b\pi^\lambda = \pi^\lambda b$, we further see that $\pi^\nnatural \delta  = \pi^\lambda b \pi_0^\nnatural=b\pi^\nnatural$.  The lemma is proved. 
\end{proof}

From now on we assume that $k\supset \Q$. In what follows, given $x\in C_m$, $m\geq 0$, we set 
\begin{equation*}
\hat{x}=(m+1)^{-1}x \qquad \text{and} \qquad \hat{N} x=N\hat{x}.
\end{equation*}
We then let $\nu:C_\bt \rightarrow C_\bt^\nnatural$ be the $k$-linear map defined by 
\begin{equation*}
 \nu(x)=\hat{N}xu^0, \qquad x\in C_\bt. 
\end{equation*}

For $j\geq 0$, set $N_j(X)=\sum_{\ell \leq j}X^\ell \in k[X]$. Note that $X^{j+1}-1=(X-1)N_j(X)$. In addition, let $D_m(X)\in k[X]$ be the polynomial given by 
\begin{equation}
 D_m(X)= \sum_{0\leq j \leq m}(m-j)X^j=\sum_{0\leq j \leq m-1}N_j(X).
 \label{eq:CnCl.Dm}  
\end{equation}
We observe that 
\begin{equation}
 N_m(X)-(m+1)= \sum_{0\leq j \leq m}(X^j-1)= \sum_{1\leq j \leq m}(X-1)N_{j-1}(X)=(X-1)D_m(X).
 \label{eq:Cn-Cl.Nm-Dm} 
\end{equation}
Let $\hat{D}:C_\bt \rightarrow C_\bt$ be the $k$-linear map defined by
\begin{equation}
\hat{D}x:=D_m(\tau) \hat{x} = \sum_{0\leq j \leq m}(m-j)\tau^j \hat{x},  \qquad x\in C_\bt. 
\label{eq:CnCl.D}
\end{equation}
Note that by~(\ref{eq:Cn-Cl.Nm-Dm}) we have
\begin{equation}
 \hat{N}+(1-\tau)\hat{D}=1. 
 \label{eq:Cn-Cl.N-D}
\end{equation}
We then let $\varphi: C_\bt^{\nnatural} \rightarrow C_\bt^{\nnatural}$ be the $k$-linear map defined by 
\begin{equation}
 \varphi(xu^{2p})= -\hat{D}xu^{2p+1}, \qquad \varphi(xu^{2p+1}) = -\hat{x}u^{2p+2}, \qquad x \in C_\bt. 
 \label{eq:Cn-Cl.vphi}
\end{equation}

\begin{lemma}[compare \cite{Ka:Crelle90}]\label{lem:Cn-Cl.nu-tilde-nu}
 We have
 \begin{equation*}
 \partial \nu=0, \qquad \pi^\nnatural_0 \nu =\hat{N}, \qquad \nu \pi_0^\nnatural =1+\dl \varphi + \varphi \dl. 
\end{equation*}
\end{lemma}
\begin{proof}
 As the range of $\nu$ is contained in $C_\bt u^0\subset \ker\dl$, we have $\dl \nu=0$. Let $x\in C_\bt$. We have 
 $\pi^\nnatural_0 \nu(x)=\pi^\nnatural_0(\hat{N}xu^0)=\hat{N}x$, and so $\pi^\nnatural_0 \nu =\hat{N}$.  
 By using~(\ref{eq:Cn-Cl.N-D}) we also see that $(1+\varphi \dl + \dl \varphi)(x u^0)$ is equal to
\begin{equation*}
  xu^0 - \dl (\hat{D}xu)= \left[1-(1-\tau)\hat{D}\right]xu^0 =\hat{N}xu^0=\nu(x)=\nu\pi_0^\nnatural(xu^0). 
\end{equation*}
If $p\geq 1$, then  $(\varphi \dl + \dl \varphi)(x u^{2p})$ is equal to 
\begin{equation*}
 \varphi(Nxu^{2p-1}) - \dl( \hat{D}xu^{2p+1})= -\left( \hat{N}+(1-\tau)\hat{D}\right) xu^{2p}=-xu^{2p}.
\end{equation*}
This gives $(1+\varphi \dl + \dl \varphi)(x u^{2p})=0= \nu\pi_0^\nnatural(xu^{2p})$. In addition, if $p\geq 0$, then 
  $(\varphi \dl + \dl \varphi)(x u^{2p+1})$ is equal to 
\begin{equation*}
 \varphi\left[(1-\tau)xu^{2p}\right] - \dl(\hat{x}u^{2p+1})= -\left((1-\tau)\hat{D}+\hat{N}\right) xu^{2p+1}=-xu^{2p+1}.
\end{equation*}
Thus, as above,  we have $(1+\varphi \dl + \dl \varphi)(x u^{2p+1})=0=\nu\pi_0^\nnatural(xu^{2p+1})$. All this shows that $1+\varphi \dl + \dl \varphi$ agrees with $\nu\pi_0^\nnatural$ on all $C^\nnatural_\bt$. The proof is complete. 
\end{proof}

As it turns out, Lemma~\ref{lem:Cn-Cl.pinn-chain-map} and Lemma~\ref{lem:Cn-Cl.nu-tilde-nu} allows us to apply Lemma~\ref{lem:perturbation.Delta-zero} to the para-twin complexes 
$(C^\nnatural_{\bt}, \dl, \delta)$ and $(C_\bt, 0,b)$ and the maps $\pi_0^\nnatural: C_\bt^\nnatural \rightarrow C_\bt$ and $\nu: C_\bt \rightarrow C_\bt^\nnatural$. 
More precisely, in the notation of Section~\ref{sec:Perturbation} the $\Delta$-operator of $(C_\bt, 0,b)$ is zero and it follows from  Lemma~\ref{lem:CCparacyclic.square}  that the $\Delta$-operator of $(C^\nnatural_{\bt}, \dl, \delta)$ is zero as well. In addition, by Lemma~\ref{lem:Cn-Cl.pinn-chain-map} and Lemma~\ref{lem:Cn-Cl.nu-tilde-nu} the map $\pi_0^\nnatural$ is compatible with the $\delta$-operators, the map $\nu$ is  compatible with the $\dl$-operators and $\pi_0\varphi=0$. Therefore, the assumptions of Lemma~\ref{lem:perturbation.Delta-zero} are fulfilled.  

Let $\varphi^\nnatural: C_{\bt}^\nnatural \rightarrow C_{\bt+1}^\nnatural$  and $\nu^\nnatural: C_\bt^\lambda \rightarrow C_{\bt}^\nnatural$ be the $k$-linear maps defined by
\begin{equation}
\varphi^\nnatural =  \sum_{j\geq 0} (\varphi \delta)^j \varphi \qquad \text{and} \qquad  \nu^\nnatural_{0}= (1+\varphi^\nnatural \delta) \nu^\lambda= \sum_{j\geq 0} (\varphi \delta)^j \nu. 
\label{eq:Cn-Cl.wbvphinn-wbnunn}
\end{equation} 
We also introduce the $k$-linear map $\mu^\nnatural:C_\bt \rightarrow C_\bt^\nnatural$ given by 
\begin{equation*}
 \mu^\nnatural (x)= \sum_{j\geq 0} (\varphi \delta)^j (\hat{x}u^0), \qquad x \in C_\bt. 
\end{equation*}
Note that $\nu^\nnatural_{0} (x) = \mu^\nnatural (Nx)$, since $\nu(x)=N\hat{x}u^0$. In addition, for  $x\in C_m$, we set $\hat{b}'x=m^{-1}bx$, with the convention that 
$\hat{b}'x=0$ when $m=0$. 

\begin{lemma}\label{lem:Cn-Cl.Cnn-C}
Let $C$ be a para-precyclic $k$-module, and assume that $k\supset \Q$.  
\begin{enumerate}
 \item We have 
 \begin{gather}
 (\dl +\delta) \nu^\nnatural_{0}  = \nu^\nnatural_{0} b - (1-T) \mu^\nnatural b' \hat{D},
 \label{eq:Cn-Cl.pi0-nu0nn-chain}\\
\pi_0^\nnatural \nu^\nnatural_{0} =1 -(1-\tau) \hat{D} ,\quad 
\nu^\nnatural_{0} \pi_0^\nnatural =1+(\dl +\delta) \varphi^\nnatural + \varphi^\nnatural (\dl +\delta),
\label{eq:Cn-Cl.pi0nn-nunn-inverses}\\
\nu^\nnatural_{0} (1-\tau)=(1-T)\mu^\nnatural,\qquad
\pi^\nnatural_0 \varphi^\nnatural =0.    
\label{eq:Cn-Cl.nunn-tau}  
\end{gather}

 \item For all $x\in C_\bt$, we have
 \begin{equation}
 \nu^\nnatural_{0}(x) = \sum_{j\geq 0}(-1)^j ( 1-\hat{D}bu)(\hat{b}'\hat{D}b)^j \hat{N} xu^{2j}.
 \label{eq:CnCl.nunn}  
\end{equation}
\end{enumerate}
\end{lemma}
\begin{proof}
In the current setup~(\ref{eq:Perturbation.ftg-inverses-D0}) and~(\ref{eq:Perturbation.ftvphi-D0}) give
\begin{equation*}
 \pi_0^\nnatural \nu^\nnatural_{0} = \pi_0^\nnatural \nu, \qquad \nu^\nnatural_{0} \pi_0^\nnatural =1+(\dl +\delta) \varphi^\nnatural + \varphi^\nnatural (\dl +\delta), \qquad 
 \pi^\nnatural_0 \varphi^\nnatural =0.
\end{equation*}
As by~(\ref{eq:Cn-Cl.N-D}) and Lemma~\ref{lem:Cn-Cl.nu-tilde-nu} we have $\pi_0^\nnatural \nu=\hat{N}=1-(1-\tau)\hat{D}$, we see that $\pi_0^\nnatural \nu^\nnatural_{0} =1-(1-\tau)\hat{D}$. Thus, in the notation of~(\ref{eq:Perturbation.fdelta-D0}) we have 
\begin{equation*}
 \tilde{\delta}= b \pi_0^\nnatural \nu = b - b (1-\tau)\hat{D}= b-(1-\tau)b'\hat{D}. 
\end{equation*}
We also observe that 
\begin{equation*}
 \nu^\nnatural_{0} (1-\tau) =\mu^\nnatural N(1-\tau)= \mu^\nnatural (1-T) = (1-T) \mu^\nnatural. 
\end{equation*}
Therefore, the chain map property~(\ref{eq:Perturbation.fg-chain-maps-D0}) gives 
\begin{equation*}
 (\dl+\delta) \nu^\nnatural_{0} =  \nu^\nnatural_{0}(0+\tilde{\delta})=  \nu^\nnatural_{0}b- \nu^\nnatural_{0}(1-\tau)b'\hat{D} = \nu^\nnatural_{0} b -(1-T) \mu^\nnatural b' \hat{D}. 
\end{equation*}

It remains to prove~(\ref{eq:CnCl.nunn}). Let $x\in C_\bt$. We have 
\begin{gather}
 (\varphi \delta)(xu^{2p})= \varphi (bxu^{2p+1})= -\hat{D}bxu^{2p+1}, 
   \label{eq:CnCl.vphi-delta-even}\\
  (\varphi \delta)(xu^{2p+1})=   -\varphi (b'xu^{2p+1})= \hat{b'}x u^{2p+2}. 
   \label{eq:CnCl.vphi-delta-odd}
\end{gather}
An induction then shows, for all $j\geq 0$, we have 
\begin{equation}
  (\varphi \delta)^{2j}(xu^{0})= (-1)^j (\hat{b}'\hat{D}b)^j x u^{2j}, \qquad (\varphi \delta)^{2j+1}(xu^{0})= 
  (-1)^{j+1} \hat{D} b(\hat{b}'\hat{D}b)^j x u^{2j+1}.
  \label{eq:CnCl.vphi-delta-j} 
\end{equation}
Thus, 
\begin{equation}
 \mu^\nnatural(x)= \sum_{j\geq 0} (\varphi \delta)^j (\hat{x}u^0)=  \sum_{j\geq 0}(-1)^j ( 1-\hat{D}bu)(\hat{b}'\hat{D}b)^j \hat{x}u^{2j}.
 \label{eq:CnCl.munn} 
\end{equation}
As  $\nu^\nnatural_{0} = \mu^\nnatural N$ we obtain~(\ref{eq:CnCl.nunn}). The proof is complete. 
\end{proof}

\begin{remark}
 For future purpose we record the following formulas for the chain homotopy $\varphi^\nnatural$ that we get from~(\ref{eq:CnCl.vphi-delta-even})--(\ref{eq:CnCl.vphi-delta-odd}). Namely, for all $x\in C_\bt$, we have 
\begin{gather}
 \varphi^\nnatural(xu^{2p})=  \sum_{j\geq 0} (-1)^{j+1}(1+\hat{b}'u)(\hat{D}b\hat{b}')^j\hat{D}x u^{2p+2j+1},
 \label{eq:Cn-Cl.wbvphinn1}\\
  \varphi^\nnatural(xu^{2p+1})= \sum_{j\geq 0} (-1)^{j+1}(1-\hat{D}b u)(\hat{b}'\hat{D}b)^j \hat{x}u^{2p+2j+2}.
   \label{eq:Cn-Cl.wbvphinn2}
\end{gather}
\end{remark}

As $R^T_\bt\subset \ran(1-\tau)$, the chain map $\pi^\nnatural:C_\bt^\nnatural \rightarrow C_\bt^\lambda$  descends to a unique $k$-linear chain map 
\begin{equation*}
 \wb{\pi}^\nnatural: C^\nnatural_{T,\bt} \longrightarrow C_\bt^\lambda, \qquad \wb{\pi}^\nnatural \pi_T = \pi^\natural =\pi^\lambda \pi_0^\nnatural.
\end{equation*}
Moreover, it follows from~(\ref{eq:Cn-Cl.nunn-tau}) that $\nu^\nnatural_{0} (\ran (1-\tau))\subset R^T_\bt$, and so the $k$-linear map $\nu^\nnatural_{0}:C_\bt \rightarrow C_\bt^\nnatural$ 
also descends to a $k$-linear map, 
\begin{equation}
 \wb{\nu}^\nnatural: C_\bt^\lambda \longrightarrow C^\nnatural_{T,\bt}, \qquad \wb{\nu}^\nnatural \pi^\lambda = \pi_T \nu^\nnatural_{0}.
 \label{eq:Cnn-Cl.wbnunn}
\end{equation}
It also follows from~(\ref{eq:Cn-Cl.wbvphinn1})--(\ref{eq:Cn-Cl.wbvphinn2}) that the homotopy $\varphi^\nnatural$ is compatible with the $T$-operators. Therefore, it  descends to a unique $k$-linear map $\wb{\varphi}^\nnatural: C^\nnatural_{T,\bt} \rightarrow C^\nnatural_{T,\bt+1}$ such that $\wb{\varphi}^\nnatural \pi_T= \pi_T \wb{\varphi}^\nnatural$. 

\begin{proposition}\label{prop:Cn-Cl.CTnn}
Suppose that $C$ is a para-precyclic $k$-module and $k\supset \Q$. Then the $k$-linear map $\wb{\nu}^\nnatural: C_\bt^\lambda \rightarrow C^\nnatural_{T,\bt}$ is a chain map such that
\begin{equation*}
 \wb{\pi}^\nnatural \wb{\nu}^\nnatural=1, \qquad   \wb{\nu}^\nnatural \wb{\pi}^\nnatural=1 + 1+(\dl +\delta) \wb{\varphi}^\nnatural + \wb{\varphi}^\nnatural (\dl +\delta), 
 \qquad  \wb{\pi}^\nnatural \wb{\varphi}^\nnatural =0. 
\end{equation*}
In particular, we obtain a deformation retract of $C^\nnatural_T$ to $C^\lambda$. 
\end{proposition}
\begin{proof}
 We observe that $[(\dl+\delta) \wb{\nu}^\nnatural -  \wb{\nu}^\nnatural b] \pi^\lambda$ is equal to 
\begin{equation*}
  (\delta +\delta) \wb{\nu}^\nnatural \pi^\lambda - \wb{\nu}^\nnatural \pi^\lambda b 
  =  (\delta +\delta) \pi_T\nu^\nnatural_{0} \pi^\lambda - \pi_T\nu^\nnatural_{0} b 
 = \pi_T \left[  (\delta +\delta) \nu^\nnatural_{0} \pi^\lambda - \nu^\nnatural_{0} b\right].
\end{equation*}
As by~(\ref{eq:Cn-Cl.pi0-nu0nn-chain}) the range of $ (\delta +\delta) \nu^\nnatural_{0} \pi^\lambda - \nu^\nnatural_{0} b$ is contained in $R^{T,\nnatural}_\bt$, we see that $[ (\dl+\delta) \wb{\nu}^\nnatural -  \wb{\nu}^\nnatural b] \pi^\lambda=0$. This implies that $(\dl+\delta) \wb{\nu}^\nnatural =  \wb{\nu}^\nnatural b$ on $C^\lambda_\bt$, i.e., $ \wb{\nu}^\nnatural$ is a chain map. 

We have $\wb{\pi}^\nnatural \wb{\nu}^\nnatural \pi^\lambda = \wb{\pi}^\nnatural \pi_T \nu^\nnatural_{0} = \pi^\lambda \pi_0^\nnatural \nu^\nnatural_{0}$. Therefore, by using~(\ref{eq:Cn-Cl.pi0nn-nunn-inverses}) we get 
\begin{equation}
 \wb{\pi}^\nnatural \wb{\nu}^\nnatural \pi^\lambda= \pi^\lambda [1-(1-\tau)\hat{D}]= \pi^\lambda.
 \label{eq:Cn-Cl.wbpi-wnunn} 
\end{equation}
It then follows that $\wb{\pi}^\nnatural \wb{\nu}^\nnatural \pi^\lambda=1$ on on $C^\lambda_\bt$. 

We have $\wb{\nu}^\nnatural \wb{\pi}^\nnatural \pi_T= \wb{\nu}^\nnatural \pi^\lambda \pi_0^\nnatural = \pi_T \nu_0^\nnatural \pi_0^\nnatural$. Combining this with~(\ref{eq:Cn-Cl.pi0nn-nunn-inverses}) then  gives 
\begin{equation*}
 \wb{\nu}^\nnatural \wb{\pi}^\nnatural \pi_T= \pi_T \left[ 1+ (\dl+\delta) \varphi^\nnatural + \varphi^\nnatural(\dl+\delta) \right] = \pi_T + 
  (\dl+\delta) \pi_T\varphi^\natural + \pi_T\varphi^\nnatural (\dl+\delta). 
\end{equation*}
As $ (\dl+\delta) \pi_T\varphi^\nnatural= (\dl+\delta) \wb{\varphi}^\nnatural \pi_T$ and $ \pi_T\varphi^\nnatural (\dl+\delta) =  \wb{\varphi}^\nnatural \pi_T (\dl+\delta) =  \wb{\varphi}^\nnatural (\dl+\delta) \pi_T$, we get 
\begin{equation*}
        \wb{\nu}^\nnatural \wb{\pi}^\nnatural \pi_T= [ 1+ (\dl+\delta) \wb{\varphi}^\nnatural + \wb{\varphi}^\nnatural(\dl+\delta)]\pi_T. 
\end{equation*}
This implies that $ \wb{\nu}^\nnatural \wb{\pi}^\nnatural = 1+ (\dl+\delta) \wb{\varphi}^\nnatural + \wb{\varphi}^\nnatural(\dl+\delta)$ on $C^\nnatural_{T,\bt}$. Likewise, as by~(\ref{eq:Cn-Cl.nunn-tau}) we have $\wb{\pi}^\nnatural \wb{\varphi}^\nnatural \pi_T= \wb{\pi}^\nnatural \pi_T\varphi^\nnatural = \pi_T \pi_0^\nnatural \varphi^\nnatural =0$, we see that 
$\wb{\pi}^\nnatural \wb{\varphi}^\nnatural \pi_T= 0$ on  on $C^\nnatural_{T,\bt}$. The proof is complete. 
\end{proof}

It is worth specializing Proposition~\ref{prop:Cn-Cl.CTnn} to precyclic $k$-modules. In this case $\wb{\pi}^\nnatural$ is just the original projection $\pi^\nnatural$ and 
$\wb{\varphi}^\nnatural$ agrees with the original chain homotopy  is $\varphi^\nnatural$. Furthermore, as $T=1$ the 
equality~(\ref{eq:Cn-Cl.pi0-nu0nn-chain}) gives $(\dl +\delta) \nu^\nnatural_{0}  = \nu^\nnatural_{0} b$. Thus, in the precyclic case it is immediate that $\nu^\nnatural_{0}$  descends to a unique $k$-linear chain map, 
\begin{equation}
 \nu^\nnatural: C^\lambda_\bt \longrightarrow C^\nnatural_\bt, \qquad  \nu^\nnatural\pi^\lambda =\nu_0^\nnatural. 
 \label{eq:Cnn-Cl.nunn}
\end{equation}
Therefore, we arrive at the following statement. 

\begin{corollary}[see also~\cite{Ka:Crelle90}]
 Suppose that $C$ is a precyclic $k$-module and $k\supset \Q$. Then we have 
 \begin{equation}
 \pi^\nnatural \nu^\nnatural=1, \qquad   \nu^\nnatural \pi^\nnatural=1 + 1+(\dl +\delta) \varphi^\nnatural + \varphi^\nnatural (\dl +\delta), \qquad \pi^\nnatural \varphi^\nnatural=0. 
 \label{eq:Cn-Cl.DR-nunn} 
\end{equation}
In particular, this provides us with a deformation retract of $C^\nnatural$ to $C^\lambda$. 
\end{corollary}

\begin{remark}
 When $C$ is a precyclic module it is well known that the projection $\pi^\nnatural:C^\nnatural_\bt\rightarrow C^\lambda_\bt$ is a quasi-isomorphism~\cite{Co:CRAS83} (see also~\cite{Lo:CH, LQ:CMH84}). 
\end{remark}  
 
 \begin{remark}
 When $C$ is a precyclic $k$-module we  recover the deformation retract of $C^\nnatural$ to $C^\lambda$ produced by Kassel~\cite{Ka:Crelle90} as follows. For $m\geq 0$ set $\tilde{D}_m(X)= \sum_{j=0}^m jX^j\in k[X]$, and let $\check{D}:C_\bt \rightarrow C_\bt$ be the $k$-linear map defined by
 \begin{equation*}
 \check{D}x=\tilde{D}_m(\tau)\hat{x}, \qquad x\in C_m, \ m\geq 0. 
\end{equation*}
We have $\tilde{D}_m(\tau)=\sum_{j=0}^m (m-j)\tau^{m-j}=\tau^m D_m(\tau^{-1})$. By combining this with~(\ref{eq:Cn-Cl.N-D}) and the equality $\tau^{m+1}=1$, it can be shown that, when $C$ is precyclic, we have 
\begin{equation*}
 \hat{N} - (1-\tau)\check{D}=1. 
\end{equation*}
Therefore, in the precyclic case, we obtain an alternative chain homotopy satisfying~(\ref{eq:Cn-Cl.DR-nunn}) by substituting $-\check{D}$ for $
\hat{D}$ in the formula~(\ref{eq:Cn-Cl.vphi}) for $\varphi$. This gives back the chain homotopy used by Kassel~\cite{Ka:Crelle90}. We then recover his formula for the homotopy inverse of $\pi^\nnatural$ by substituting  $-\check{D}$ for $\hat{D}$ in the formula~(\ref{eq:CnCl.nunn}). 
\end{remark}

Suppose that $C$ is a quasi-precyclic $k$-module, so that $C_\bt=C^T_\bt \oplus R^T_\bt$. Let $\pi^T:C_\bt \rightarrow C_\bt$ be the projection onto $C^T_\bt$ defined by this splitting. This is a para-precyclic $k$-module map, and so it gives rise to a chain map $\pi: C_\bt^\nnatural \rightarrow C_\bt^\nnatural$. By Proposition~\ref{prop:para-precyclic.CTCT-Cnn} this chain map is chain homotopic to the identity map. Namely, we have $\pi^T=1+(\dl+\delta)h^\nnatural + h^\nnatural (\dl+\delta)$, where $h^\nnatural$ is given by~(\ref{eq:para-precyclic.homotopy-hnn}). 
We observe that~(\ref{eq:Cn-Cl.nunn-tau}) implies that $\pi^T \nu^\nnatural_{0}(1-\tau)=\pi^T(1-T)\mu^\nnatural=0$, and so $\pi^T \nu^\nnatural_{0}$  descends to a unique $k$-linear map, 
\begin{equation}
 \nu^{T,\nnatural}: C^\lambda_\bt \longrightarrow C^\nnatural_\bt, \qquad \nu^{T,\nnatural}\pi^\lambda = \pi^T \nu^\nnatural_{0}.
 \label{eq:Cnn-Cl.nuTnn}
\end{equation}
We also let $\varphi^{T,\nnatural}: C_\bt^\nnatural \rightarrow C_{\bt+1}^\nnatural$ be the $k$-linear map defined by
 \begin{equation*}
  \varphi^\nnatural = h^\nnatural + \varphi^\nnatural \pi^T= h^\nnatural + \pi^T \varphi^\nnatural.  
\end{equation*}

\begin{proposition}\label{prop:CnCl.Cnn-quasi}
 Assume that $C$ is a quasi-precyclic $k$-module and $k\supset \Q$. Then the $k$-linear map $\nu^{T,\nnatural}:C^\lambda \rightarrow C_\bt^\nnatural$ is a chain map such that 
 \begin{equation*}
 \pi^\nnatural \nu^{T,\nnatural} =1, \qquad  \nu^{T,\nnatural}\pi^\nnatural =1 +(\dl +\delta)\varphi^{T,\nnatural} +\varphi^{T,\nnatural}  (\dl +\delta), \qquad  
 \pi^\nnatural\varphi^{T,\nnatural}=0.
\end{equation*}
In particular, we obtain a deformation retract of $C^\nnatural$ to $C^\lambda$. 
 \end{proposition}
\begin{proof}
Like the canonical projection $\pi_T:C_\bt \rightarrow C_{T,\bt}$, the projection $\pi^T$ is a para-precyclic $k$-linear map that is annihilated by $R^T_\bt$. Thus, by using the equality $\nu^{T,\nnatural}\pi^\lambda = \pi^T \nu^\nnatural_{0}$ and arguing along similar lines as that of the proof of 
Proposition~\ref{prop:Cn-Cl.CTnn} it can be shown that 
$\nu^{T,\nnatural}:C^\lambda \rightarrow C_\bt^\nnatural$ is a chain map.  In addition, as $\pi^\nnatural \pi^T=\pi^\nnatural$, in the same way as in~(\ref{eq:Cn-Cl.wbpi-wnunn}) we have 
\begin{equation*}
 \pi^\nnatural \nu^{T,\nnatural} \pi^\lambda = \pi^\nnatural \pi^T \nu^\nnatural_{0} = \pi^\lambda \pi_0^\nnatural \nu^\nnatural_{0} = \pi^\lambda [1-(1-\tau)\hat{D}]=\pi^\lambda. 
\end{equation*}
 Therefore, we see that $ \pi^\nnatural \nu^{T,\nnatural}=1$ on $C_\bt^\lambda$. 
 
 We also have $\nu^{T,\nnatural}\pi^\nnatural = \nu^{T,\nnatural}\pi^\nnatural \pi^\nnatural_0= \pi^T \nu^\nnatural_0 \pi_0^\nnatural$. Thus, by using~(\ref{eq:Cn-Cl.pi0nn-nunn-inverses}) we get
 \begin{equation*}
  \nu^{T,\nnatural}\pi^\nnatural = \pi^T \left[ 1 +(\dl +\delta)\varphi^{\nnatural} +\varphi^{\nnatural}  (\dl +\delta)\right] = \pi^T + (\dl +\delta)(\pi^T\varphi^{\nnatural}) + 
 (\pi^T\varphi^{\nnatural}) (\dl +\delta). 
\end{equation*}
Combining this with the homotopy formula $\pi^T=1+(\dl+\delta)h^\nnatural + h^\nnatural (\dl+\delta)$ then gives 
 \begin{align*}
  \nu^{T,\nnatural}\pi^\nnatural & = 1+ (\dl +\delta)(h^\nnatural+ \varphi^{\nnatural}\pi^T) + 
 (h^\nnatural+ \varphi^{\nnatural}\pi^T) (\dl +\delta) \\
 & = 1 +(\dl +\delta)\varphi^{T,\nnatural} +\varphi^{T,\nnatural}  (\dl +\delta). 
\end{align*}

Thanks to~(\ref{eq:Cn-Cl.nunn-tau}) we have $\pi^\nnatural \varphi^\nnatural= \pi^\lambda \pi_0^\nnatural \varphi^\nnatural=0$. 
As mentioned in Remark~\ref{rmk:para-precyclic.hnn-special} the range of $h^\nnatural$ is contained in $R^{T,\nnatural} \subset \ran(1-\tau)$, and so 
$\pi^\nnatural h^\nnatural=0$.  Thus, $\pi^\nnatural  \varphi^{T,\nnatural} =  \pi^\nnatural  h^\nnatural +  \pi^\nnatural   \varphi^\nnatural \pi^T =  0$. 
The proof is complete.
\end{proof}

\section{Comparison of $C^\natural$ and $C^\lambda$}  \label{sec:Cn-Cl}
When $C$ is a precyclic $k$-module it is well known that if $k\supset \Q$, then the chain complexes $C^\natural$ and $C^\lambda$ are quasi-isomorphic~\cite{Co:MFO81, Co:CRAS83, Co:IHES85} (see also~\cite{LQ:CMH84}). In this section, we shall combine the results of the previous two sections to compare $C^\natural$ and $C^\lambda$  in the case of $H$-unital para-precyclic modules. 

Throughout this section we assume that $k\supset \Q$, and let $C=(C_\bt, d,s,t)$ be an $H$-unital preparacyclic $k$-module. 
Let $\pi_0:C_\bt^\natural\rightarrow C_\bt$ be the natural projection onto the zeroth degree summand $C_\bt u^0=C_\bt$. By composing it with the projection $\pi^\lambda:C_\bt\rightarrow C_\bt^\lambda$ we get natural projection $\pi^\natural :=\pi^\lambda \pi_0^\natural: C_\bt^\natural \rightarrow C_\bt^\lambda$. Thus, 
\begin{equation}
 \pi^\natural(xu^0)=x^\lambda, \qquad  \pi^\natural(xu^p)=0, \quad p\geq 1. 
 \label{eq:CnCl.pin}
\end{equation}
Note that $\pi^\natural_0=\pi^\nnatural_0 I$ and $\pi^\natural=\pi^\nnatural I$. We also observe that
 \begin{equation}
 \pi^\natural_0 b= b\pi^\natural_0,  \qquad  \pi^\natural_0 B= B\pi_0^\natural,  \qquad   \pi^\natural b= b\pi^\natural,  \qquad  \pi^\natural B=0.
 \label{eq:CnCl.pin-chain-map} 
\end{equation}
In particular, this implies that $\pi^\natural: C_\bt^\natural \rightarrow C_\bt^\lambda$ is a chain map.

Let $\nu^\natural_{0}: C_\bt \rightarrow C_\bt^\natural$ and $\varphi^\natural: C_\bt^\natural \rightarrow C_\bt$ be the $k$-linear maps defined by
\begin{equation*}
 \nu^\natural_{0} = J \nu^\nnatural_{0}, \qquad \varphi^\natural= J \varphi^\nnatural I.
\end{equation*}
We also set $\mu^\natural = J \mu^\nnatural$, where $\mu^\nnatural: C_\bt \rightarrow C^\nnatural$ is given by~(\ref{eq:CnCl.munn}).

\begin{lemma}\label{lem:Cn-Cl.nu0}
 Suppose that $C$ is an $H$-unital para-precyclic $k$-module and $k\supset \Q$.
\begin{enumerate}
 \item We have
\begin{gather}
 (b+Bu^{-1})\nu^\natural_{0}=\nu^\natural_{0} b -(1-T)\mu^\natural b'\hat{D},
 \label{eq:Cn-Cl.nun0-chain-map}\\
 \pi_0^\natural \nu^\natural_{0} =1 -(1-\tau) \left[ 1+ s'\hat{D} b\hat{N}\right],
 \label{eq:Cn-Cl.pio-nun0}\\
 \nu^\natural_{0} \pi^\natural_0 =1 + (b+Bu^{-1}) \varphi^\natural + \varphi^\natural(b+Bu^{-1}),\\
 \nu^\natural_{0} (1-\tau)= (1-T) \mu^\natural, \qquad \pi_0^\natural \varphi^\natural =-(1-\tau) s'\hat{b} \hat{D} \pi_0^\natural. 
 \label{eq:Cn-Cl.nun-tau}
\end{gather}

\item For all $x\in C_\bt$, we have 
 \begin{equation}
                 \nu^\natural_{0}(x) =\sum_{j\geq 0}(-1)^j \left[ 1-(1-\tau)s'\hat{D}b\right](\hat{b}'\hat{D}b)^j \hat{N} x u^{j}. 
                 \label{eq:CnCl.nun}
  \end{equation}
\end{enumerate}
\end{lemma}
\begin{proof}
 By~(\ref{eq:Cn-Cl.nunn-tau}) we have $\nu^\natural_{0} (1-\tau)=J\nu^\nnatural_{0} (1-\tau)= J(1-T)  \mu^\nnatural = (1-T)\mu^\natural$. As $J$ is a chain map we also have $(b+Bu^{-1})\nu^\natural_{0}= (b+Bu^{-1})J\nu^\nnatural_{0} =
 J(\dl+\delta)\nu^\nnatural_{0}$. Combining this with~(\ref{eq:Cn-Cl.pi0-nu0nn-chain}) then gives 
\begin{equation*}
 (b+Bu^{-1})\nu^\natural_{0}= J\nu^\nnatural_{0} b-(1-T) J\mu^\nnatural b'D= \nu^\natural_{0} b -(1-T)\mu^\natural b'\hat{D}. 
\end{equation*}
Moreover, by using~(\ref{eq:Cn-Cl.pi0nn-nunn-inverses}) we see that $\nu^\natural_{0} \pi_0^\natural$ is equal to 
\begin{align*}
 J( \nu^\nnatural_{0} \pi_0^\nnatural) I & = J \left[ 1 +(\dl +\delta)\varphi^\nnatural +\varphi^\nnatural  (\dl +\delta)\right] J \\
 &= JI + (b+Bu^{-1}) J\varphi^\nnatural I + J \varphi^\nnatural I (b+Bu^{-1}) \\
 & = 1 + (b+Bu^{-1}) \varphi^\natural + \varphi^\natural (b+Bu^{-1}). 
\end{align*}

It follows from the formulas~(\ref{eq:Cnn-Cn.J-map}) and~(\ref{eq:CnCl.nunn}) for $J$ and $\nu^\nnatural_{0}$ that, given any $x\in C_\bt$, we have
 \begin{align*}
 \nu^\natural_{0}(x) & = \sum_{j \geq 0} (-1)^j J\left[ (\hat{b}'\hat{D}b)^j \hat{N}xu^{2j} -\hat{D} b(\hat{b}'\hat{D}b)^j \hat{N}xu^{2j+1}\right]\\ 
 &= \sum_{j\geq 0}(-1)^j \left[ 1-(1-\tau)s'\hat{D}b\right](\hat{b}'\hat{D}b)^j \hat{N}u^{j}. 
\end{align*}
This proves~(\ref{eq:CnCl.nun}). Combining this with the equality $\hat{N}+(1-\tau)\hat{D}=1$ further gives
\begin{equation}
 \pi_0\nu^\natural_{0} =\hat{N}-(1-\tau)s'\hat{D}b\hat{N}= 1-(1-\tau)\left[ \hat{D}+ s'\hat{D}b\hat{N}\right].
 \label{eq:Cn-Cl.pio-nun0-hN}
\end{equation}

We have $\pi_0^\natural \varphi^\natural = \pi^\natural_0 J \varphi^\nnatural I$. It follows from~(\ref{eq:Cnn-Cn.J-map}) that we have 
\begin{equation}
 \pi_0^\natural J= \pi_0^\nnatural +(1-\tau)s' \pi_0^\nnatural u^{-1}. 
 \label{eq:Cn-Cl.pinn0J-pinn0}
\end{equation}
Let $x\in C_\bt$. The formulas~(\ref{eq:Cn-Cl.wbvphinn1})--(\ref{eq:Cn-Cl.wbvphinn2}) implies that $\varphi^\nnatural(xu^p)$ is contained in $\oplus_{q\geq p+1} C_\bt u^q$, and so $\varphi^\nnatural(xu^p)=0$ when $p\geq 1$. As~(\ref{eq:Cn-Cl.vphi}) shows that $\varphi(xu^0)= -\hat{D} xu \bmod  \oplus_{q\geq 2} C_\bt u^q$, we also get 
\begin{equation*}
 \pi_0^\natural J\varphi^\nnatural (xu^0)= - \pi_0^\nnatural J\left(\hat{D} xu\right) =-(1-\tau) s' \hat{D} x= -(1-\tau) \hat{D} \pi_0^\nnatural (xu^0). 
\end{equation*}
Therefore, we see that $\pi_0^\natural J\varphi^\nnatural = -(1-\tau) \hat{D} \pi_0^\nnatural $. As $\pi_0^\nnatural I=\pi_0^\natural$, we then obtain 
\begin{equation*}
 \pi_0^\natural \varphi^\natural = \pi^\natural_0 J \varphi^\nnatural I =  -(1-\tau) \hat{D}s' \pi_0^\nnatural I =  -(1-\tau) \hat{D} \pi_0^\natural. 
\end{equation*}
The proof is complete. 
\end{proof}

Let $\pi_T:C_\bt \rightarrow C_{T,\bt}$ be the canonical projection of $C_\bt$ onto $C_{T,\bt}$. As this is a map of $H$-unital para-precyclic $k$-modules, it gives rise to a chain map  $\pi_T:C^\natural \rightarrow C^\natural_{T,\bt}$. 
Therefore, in the same way as with the projection $\pi^\nnatural$ above, the projection $\pi^\natural:C^\natural_\bt \rightarrow C_\bt^\lambda$  descends to a unique $k$-linear chain map, 
\begin{equation*}
 \wb{\pi}^\natural: C_{T,\bt}^\natural \longrightarrow C^\lambda_\bt, \qquad \wb{\nu}^\natural \pi^\lambda =\pi_T \pi^\natural. 
\end{equation*}
 
It follows from~(\ref{eq:Cn-Cl.nun-tau}) that $\pi_T \nu^\natural_0(1-\tau)=\pi_T(1-T)\mu^\natural =0$. Therefore, the $k$-linear map $\pi_T\nu_0^\natural:C_\bt \rightarrow C^\natural_{T,\bt}$  descends to a unique $k$-linear map,
\begin{equation}
 \wb{\nu}^\natural: C_\bt^\lambda \longrightarrow C^\natural_{T,\bt}, \qquad \wb{\nu}^\natural \pi^\lambda=\pi_T\nu_0^\natural.
  \label{eq:Cn-Cl.wbnun} 
\end{equation}
In fact, as $\nu_0^\natural = J \nu^\nnatural_0$, we have $\pi_T\nu_0^\natural= \pi_T J \nu^\nnatural_0 = J\pi_T \nu^\nnatural_0 = J \wb{\nu}^\nnatural$. Therefore, we see that $\wb{\nu}^\natural =J \wb{\nu}^\nnatural$. 
In particular, as $J$ is a chain map, and $\wb{\nu}^\nnatural$ is a chain map as well by Proposition~\ref{prop:Cn-Cl.CTnn}, we see that 
$ \wb{\nu}^\natural: C_\bt^\lambda \rightarrow C^\natural_{T,\bt}$ is a chain map. 

In addition, as the chain homotopy $\varphi^\natural: C^\natural_{\bt} \rightarrow C^\natural_{\bt+1}$ is compatible with the $T$-operator, it  descends to a unique $k$-linear map $\wb{\varphi}^{\natural}: C^\natural_{T,\bt} \rightarrow C^\natural_{T,\bt+1}$ such that $\wb{\varphi}^{\nnatural} \pi_T = \pi_T \varphi^{\natural}$. 
By using Lemma~\ref{lem:Cn-Cl.nu0} and arguing along similar lines as that of the proof of Proposition~\ref{prop:Cn-Cl.CTnn} we then obtain the following result.

\begin{proposition}\label{prop:CnCl.CTn}
 Suppose that $C$ is an $H$-unital para-precyclic $k$-module and $k\supset \Q$. Then
\begin{equation*}
\wb{\pi}^\natural \wb{\nu}^\natural= 1, \qquad \wb{\nu}^\natural \wb{\pi}^\natural= 1+ (b+Bu^{-1}) \wb{\varphi}^\natural + \wb{\varphi}^\natural (b+Bu^{-1}), \qquad 
\wb{\pi}^\natural \wb{\varphi}^\natural=0. 
\end{equation*} 
In particular, we obtain a deformation retract of $C^\natural_T$ to $C^\lambda$.          
\end{proposition}
\begin{proof}
 By using Lemma~\ref{lem:Cn-Cl.nu0} and arguing along similar lines as that of the proof of Proposition~\ref{prop:Cn-Cl.CTnn} it can be shown that $\wb{\pi}^\natural \wb{\nu}^\natural= 1$ and $\wb{\nu}^\natural \wb{\pi}^\natural= 1+ (b+Bu^{-1}) \wb{\varphi}^\natural + \wb{\varphi}^\natural (b+Bu^{-1})$. Moreover, by using~(\ref{eq:Cn-Cl.nun-tau}) we see that 
 $\wb{\pi}^\natural \wb{\varphi}^\natural \pi_T= \wb{\pi}^\natural  \pi_T \varphi^\natural =\pi^\lambda \pi_0^\natural \varphi^\natural = - \pi^\lambda (1-\tau) s'\hat{b} \hat{D}=0$. It then follows that $\wb{\pi}^\natural \wb{\varphi}^\natural = 0$ on $C^\natural_{T,\bt}$. The proof is complete. 
\end{proof}

It is interesting to specialize Proposition~\ref{prop:CnCl.CTn} to $H$-unital precyclic $k$-modules. In this case $\wb{\pi}^\natural$ is just the projection $\pi^\natural$ and $\wb{\varphi}^\natural$ agrees with the chain homotopy $\varphi^\natural$. 
Furthermore, as $T=1$ the equalities~(\ref{eq:Cn-Cl.nun0-chain-map}) and~(\ref{eq:Cn-Cl.nun-tau}) ensure us that $\nu^\natural_0(1-\tau)=0$ and $\nu_0^\natural b = (b+Bu^{-1}) \nu_0^\natural$. Thus, it is immediate that $\nu_0^\natural$ descends to a unique $k$-linear chain map, 
\begin{equation}
 \nu^\natural :C^\lambda_\bt \longrightarrow C^\natural_\bt, \qquad  \nu^\natural  \pi^\lambda =  \nu^\natural_0.
 \label{eq:Cn-Cl.nun} 
\end{equation}
Therefore, we obtain the following statement. 

\begin{corollary}
 Suppose that $C$ is an $H$-unital precyclic $k$-module and $k\supset \Q$. Then we have 
\begin{equation*}
 \pi^\natural \nu^\natural= 1, \qquad \nu^\natural \pi^\natural= 1+ (b+Bu^{-1}) \varphi^\natural + \varphi^\natural (b+Bu^{-1}), \qquad \pi^\natural \varphi^\natural=0. 
\end{equation*} 
In particular, we obtain a deformation retract of $C^\natural$ to $C^\lambda$.          
\end{corollary}

Suppose now that $C$ is an $H$-unital quasi-precyclic $k$-module, so that $C_\bt=C_\bt^T \oplus R_\bt^T$. Let $\pi^T:C_\bt \rightarrow C_\bt$ be the projection of $C_\bt$ on $C^T_\bt$ associated with this splitting. This is an $H$-unital para-precyclic $k$-module map, and so it yields a chain map $C_\bt^\natural \rightarrow C^\natural_\bt$. By Proposition~\ref{prop:Parachain.quasi-mixed-DR} this chain map is $S$-homotopy equivalent to the identity map. 
Namely, $\pi^T=1+(b+Bu^{-1})h + h (b+Bu^{-1})$, where $h:C_\bt^\natural \rightarrow C^\natural_{\bt+1}$ is given by~(\ref{eq:paracyclic.h}). 
Note also that by~(\ref{eq:Cn-Cl.nun-tau}) we have $\pi^T \nu^\natural_0 (1-\tau)= \pi^T (1-T) \mu^\natural=0$, and so  $\pi^T \nu^\natural_0$  descends to a unique $k$-linear map, 
\begin{equation}
 \nu^{T,\natural}: C_\bt^\lambda \longrightarrow C^\natural_\bt, \qquad  \nu^{T,\natural} \pi^\lambda= \pi^T \nu_0^\natural.
 \label{eq:Cn-Cl.nuTn} 
\end{equation}
We also let $\varphi^{T,\natural}: C_\bt^\natural \rightarrow C_{\bt+1}^\natural$ be the $k$-linear map defined by
 \begin{equation*}
  \varphi^{T,\natural} = h+ \varphi^\natural \pi^T= h+ \varphi^\natural \pi^T. 
\end{equation*}

\begin{proposition}\label{prop:CnCl.Cn-quasi}
 Assume that $C$ is an $H$-unital quasi-precyclic $k$-module and $k\supset \Q$. Then the $k$-linear map  $\nu^{T,\natural}:C^\lambda_\bt \rightarrow C^\natural_\bt$ is a chain map such that
\begin{equation*}
 \pi^\natural \nu^{T,\natural}= 1, \qquad \nu^{T,\natural} \pi^\natural= 1+ (b+Bu^{-1}) \varphi^{T,\natural} + \varphi^{T,\natural} (b+Bu^{-1}), \qquad \pi^\natural \varphi^{T,\natural}=0.  
\end{equation*} 
In particular, we obtain a deformation retract of $C^\natural$ to $C^\lambda$.          
\end{proposition}
\begin{proof}
 By using Lemma~\ref{lem:Cn-Cl.nu0} and arguing along similar lines as that of the proof of Proposition~\ref{prop:CnCl.Cnn-quasi} we see that 
 $\pi^\natural \nu^{T,\natural}= 1$ and $\nu^{T,\natural} \pi^\natural= 1+ (b+Bu^{-1}) \varphi^{T,\natural} + \varphi^{T,\natural} (b+Bu^{-1})$. 
 As mentioned in Remark~\ref{rmk:Para-S-Module.homotopy-special-C'}, the range of $h$ is contained in $R^{T,\natural} \subset \ran (1-\tau)$, and so
 $\pi^\natural h =0$. By using~(\ref{eq:Cn-Cl.nun-tau}) we also get 
 $\pi^\natural \varphi^\natural = \pi^\lambda \pi_0^\natural \varphi^\natural = - \pi^\lambda (1-\tau) s'\hat{b} \hat{D}=0$. Therefore, we see that 
$\pi^\natural  \varphi^{T,\natural }=  \pi^\natural  h +\pi^\natural  \varphi^\natural \pi^T =  0$.
The proof is complete.
\end{proof}

\section{The Periodicity Operator}\label{sec:S}
It was observed by Kassel~\cite{Ka:Crelle90} that, given any precyclic $k$-module $C$ with $k\supset \Q$,  the deformation retract of $C^\nnatural$ to Connes' complex $C^\lambda$ allows us to get an alternative description of the periodicity operator of Connes~\cite{Co:MFO81, Co:CRAS83, Co:IHES85} in cyclic homology. 

In this section, we further elaborate on Kassel's approach to the periodicity operator. We shall relate the periodicity operator to the comparison results of the previous sections between $C^\nnatural$ and $C^\lambda$ for para-precyclic modules,  and between $C^\natural$ and $C^\lambda$ in the $H$-unital case. We will also give a few applications in periodic cyclic homology. 

One feature of the approach of~\cite{Ka:Crelle90} is the use of a special chain homotopy in the construction of the deformation retract of $C^\nnatural$ and $C^\lambda$ in the precyclic case. A special chain homotopy need not be available in general in the setting of  para-precyclic modules (\emph{cf}.~Remark~\ref{rmk:Perturbation.special}). 
We shall bypass this issue by using the almost chain homotopy inverse $\nu_0^\nnatural$ of Section~\ref{sec:Cnn-Cl}. Not only will this allow us to proceed in the para-precyclic case, this will also lead us to a simpler formula for the periodicity operator and an equality with Connes' periodicity at the level of chains, rather than at the cyclic homology level (compare~\cite{Ka:Crelle90, Lo:CH}). 

\subsection{The periodicity operator on $C^\lambda_\bt$} Suppose that $k\supset \Q$, and let $C$ be a (pre)cyclic $k$-module. The cyclic homology of $C$ is a module over the cyclic cohomology of $k$. By using Connes' cyclic chain complex $C^\lambda$, this module structure is implemented at the chain level by combining the canonical identification $C_\bt\simeq k \otimes_k C_\bt$ with the cap product with the Connes cyclic cochain complex of $k$ (see~\cite{Co:MFO81, Co:CRAS83, Co:IHES85}). In fact, the cyclic cohomology of $k$ is a polynomial ring over $k$ generated by the 2-cocycle such that $\sigma(1,1,1)=1$. The action of $\sigma$ on $C_\bt^\lambda$ is implemented by the chain map $S:C^\lambda_\bt \rightarrow C^\lambda_{\bt-2}$ given by
\begin{equation}
 S (x^\lambda) = \frac{1}{(m-1)m} \sum_{0\leq i<j\leq m} (-1)^{i+j} (d_id_jx)^\lambda, \qquad x\in C_m, \ m\geq 2.
 \label{eq:S.Connes-Periodicity} 
\end{equation}
This is Connes' \emph{periodicity operator}. In particular, this turns the chain complex $C^\lambda$ into an $S$-module. In addition, this operator fits into Connes' exact sequence in cyclic homology, 
\begin{equation*}
 \cdots \rightarrow H_\bt(C) \stackrel{I}{\longrightarrow} H^\lambda_\bt(C)  \stackrel{S}{\longrightarrow} H^\lambda_{\bt-2} (C) \stackrel{B}{\longrightarrow} H_{\bt-1} (C)\rightarrow \cdots,
\end{equation*}
where $H_\bt(C)$ (resp., $H^\lambda_\bt(C)$) is the Hochschild (resp., cyclic) homology of $C$ (see~\cite{Co:MFO81, Co:CRAS83, Co:IHES85}).

From now on, we assume that $k\supset \Q$, and let $C=(C_\bt, d,t)$ be a para-precyclic $k$-module. It is not possible to find a chain map $S:C_\bt^\lambda\rightarrow C_{\bt-2}^\lambda$ such that  $S\pi^\nnatural=\pi^\nnatural u^{-2}$ or $S\wb{\pi}^\nnatural=\wb{\pi}^\nnatural u^{-2}$. Nevertheless, by Proposition~\ref{prop:Cn-Cl.CTnn} 
the chain map $\wb{\pi}^\nnatural$ has a right-inverse and chain homotopy left-inverse. Namely, the embedding  $\wb{\nu}^\nnatural: C^\lambda_\bt \rightarrow  C^\nnatural_{T,\bt}$ given by~(\ref{eq:CnCl.nunn}) and~(\ref{eq:Cnn-Cl.wbnunn}). Therefore, it is natural to seek for a chain map $S:C_\bt^\lambda\rightarrow C_{\bt-2}^\lambda$ such that 
\begin{equation}
    \wb{\nu}^\nnatural S = u^{-2} \wb{\nu}^\nnatural \qquad \text{for all $x\in C_\bt$}. 
    \label{eq:Spinn-nunnS}
\end{equation}
As we shall see this equation is satisfied by a unique $k$-linear chain map of degree~$-2$ and will provide us with a solution of  $S\pi^\nnatural=\pi^\nnatural u^{-2}$ up to 
homotopy (see Proposition~\ref{prop:S.para-precyclic} below). 

Let $S_0:C_\bt \rightarrow C_{\bt-2}$ be the $k$-linear map defined by 
\begin{equation*}
 S_0 x= \pi^\nnatural_0 \big(u^{-2} \nu^\nnatural_{0} (x)\big), \qquad x \in C_\bt. 
\end{equation*}
We also introduce the $k$-linear map $\psi^\nnatural_0:C_\bt^\nnatural\rightarrow C_{\bt-1}$ given by
\begin{equation*}
    \psi_0^\nnatural = \pi_0^\nnatural \left(u^{-2} \varphi^\nnatural \right).
\end{equation*}
It follows from~(\ref{eq:Cn-Cl.wbvphinn1})--(\ref{eq:Cn-Cl.wbvphinn2}) that we have
\begin{equation}
 \psi_0^\nnatural (xu^0)=-\hat{b}'\hat{D} x, \qquad  \psi_0^\nnatural (xu)=-\hat{x}, \qquad  \psi_0^\nnatural (xu^p)=0, \quad p\geq 2. 
 \label{eq:S.psinn0}
\end{equation}

\begin{lemma}
Set $\xi =-\hat{b}'\hat{D} b$ and $\eta=-\hat{D}b$. Then 
\begin{gather}
 S_0 =\xi \hat{N}, \qquad S_0(1-\tau)x=(1-T)\xi \hat{x}, \quad x \in C_\bt,
 \label{eq:S.S0-tau}\\
 u^{-2} \nu^\nnatural_{0}- \nu^\nnatural_{0} S_0 = (1-T)\mu^\nnatural \eta \hat{D}b' ,
 \label{eq:S.nunn0-Smap}\\
 S_0 \pi_0^\nnatural -\pi_0^\nnatural u^{-2}= b\psi_0^\nnatural + \psi_0^\nnatural (\dl +\delta) + (1-\tau) \pi_0^\nnatural (u^{-3} \varphi^\nnatural),
 \label{eq:S.pinn0-Smap} \\
S_0 b =bS_0 +(1-\tau) \eta\xi  \hat{N} +(1-T) \xi \hat{b}'\hat{D}.
\label{eq:S.S0-chain} 
\end{gather}
\end{lemma}
\begin{proof}
 As by~(\ref{eq:CnCl.nunn}) we have $\nu^\nnatural_{0}= \sum_{j\geq 0} (1+\eta u) \xi^j \hat{N} u^{2j}$, we get
 \begin{equation}
 u^{-2} \nu^\nnatural_{0} = \sum_{j \geq 1} (1+\eta u) \xi^j \hat{N} u^{2j-2}=  \sum_{j \geq 0} (1+\eta u) \xi^{j+1} \hat{N} u^{2j}.
 \label{eq:S.u2nunn} 
\end{equation}
It then follows that $S_0=\pi^\nnatural_0 (u^{-2} \nu^\nnatural_{0})= \xi \hat{N}$. Furthermore, given any $x\in C_\bt$, we have
\begin{equation*}
 S_0(1-\tau) x =  \xi N(1-\tau) \hat{x} = \xi(1-T) \hat{x}=(1-T) \xi \hat{x}. 
\end{equation*}

As $S_0=  \xi \hat{N}$ we have $\nu^\nnatural_{0} S_0=  \sum_{j\geq 0} (1+\eta u) \xi^j \hat{N}\xi \hat{N} u^{2j}$. Combining this with~(\ref{eq:S.u2nunn}) and using the equality 
$\hat{N}+(1-\tau)\hat{D}=1$ gives
\begin{equation*}
\nu^\nnatural_{0} S_0 -u^{-2} \nu^\nnatural_{0}   = \sum_{j \geq 0} (1+\nu u) \xi^j (\hat{N}-1) \xi \hat{N} =  -\sum_{j\geq 0} (1+\eta u) \xi^j \hat{D}(1-\tau) \xi \hat{N}.
\end{equation*}
As by~(\ref{eq:CnCl.munn}) we have $\mu^\nnatural (x)=    \sum_{j\geq 0} (1+\eta u) \xi^j \hat{x}$, $x\in C_\bt$, we get
\begin{equation*}
\nu^\nnatural_{0} S_0 -u^{-2} \nu^\nnatural_{0}  = -\mu^\nnatural D(1-\tau) \xi \hat{N}= \mu^\nnatural D(1-\tau) \hat{b}'\hat{D} b \hat{N}.
\end{equation*}
Thanks to the equality $(1-\tau)b=(1-\tau)b'$ we have 
\begin{equation*}
 D(1-\tau) \hat{b}'\hat{D} b= \hat{D}b(1-\tau)\hat{D} b= \hat{D}b\hat{D}(1-\tau) b=  \hat{D}b\hat{D} b'(1-\tau). 
\end{equation*}
Thus, 
\begin{equation*}
\nu^\nnatural_{0} S_0 -u^{-2} \nu^\nnatural_{0}  = \mu^\nnatural \hat{D}b\hat{D} b'(1-\tau) \hat{N} = \mu^\nnatural \hat{D}b\hat{D} \hat{b}'(1-T) = (1-T)\mu^\nnatural \eta 
\hat{D} \hat{b}'. 
\end{equation*}

We have $S_0\pi_0^\nnatural -\pi_0^\nnatural u^{-2}= \pi_0^\nnatural (u^{-2} \nu_0^\nnatural)\pi_0^\nnatural -\pi_0^\nnatural u^{-2}=  
\pi_0^\nnatural u^{-2}(\nu_0^\nnatural\pi_0^\nnatural -1)$. Therefore, by using~(\ref{eq:Cn-Cl.pi0nn-nunn-inverses}) and Lemma~\ref{lem:Cn-Cl.pinn-chain-map} we obtain
\begin{align*}
 S_0\pi_0^\nnatural -\pi_0^\nnatural u^{-2} & = \pi_0^\nnatural u^{-2}\left[ (\dl+\delta)\varphi^\nnatural + \varphi^\nnatural (\dl+\delta)\right] \\
 & = b ( \pi_0^\nnatural u^{-2} \varphi^\nnatural)  + ( \pi_0^\nnatural u^{-2} \varphi^\nnatural)(\dl+\delta) + 
 (\pi_0^\nnatural (\dl+\delta) -b \pi_0^\nnatural)  u^{-2} \varphi^\nnatural \\
 &= b\psi_0^\nnatural + \psi_0^\nnatural (\dl +\delta) + (1-\tau) \pi_0^\nnatural (u^{-3} \varphi^\nnatural). 
\end{align*}

It remains to prove~(\ref{eq:S.S0-chain}). By using~(\ref{eq:Cn-Cl.pi0-nu0nn-chain}) we see that $S_0 b$ is equal to
\begin{align}
  \pi^\nnatural_0 \left(u^{-2} \nu^\nnatural_{0} b \right)& =  \pi^\nnatural_0 \left(u^{-2} (\dl+\delta) \nu^\nnatural_{0}\right) + \pi_0^\nnatural \left( u^{-2} \mu^\nnatural (1-T)b'\hat{D}\right) \nonumber \\
  & =
  \pi^\nnatural_0 (\dl+\delta)\left( u^{-2} \nu^\nnatural_{0} \right) + (1-T)\pi_0^\nnatural \left( u^{-2} \mu^\nnatural \right)b'\hat{D}.
  \label{eq:S.S0b} 
\end{align}
In the same way as in~(\ref{eq:S.u2nunn}) we have $u^{-2} \mu^\nnatural(x) = \sum_{j \geq 0} (1+\eta u) \xi^{j+1} \hat{x} u^{2j}$, $x\in C_\bt$, and so  we see that
$\pi_0^\nnatural ( u^{-2} \mu^\nnatural(x))= \xi \hat{x}$. Thus, 
\begin{equation}
  (1-T)\pi_0^\nnatural \left( u^{-2} \mu^\nnatural \right)b'\hat{D}= (1-T) \xi \hat{b}' \hat{D}. 
  \label{eq:S.1-Tu2munn}
\end{equation}
By using~(\ref{eq:Cn-Cl.pi0-nu0nn-chain}) we also see that $ \pi^\nnatural (\dl+\delta)( u^{-2} \nu^\nnatural_{0})$ is equal to
\begin{equation*}
  b\pi_0^\nnatural \left( u^{-2} \nu^\nnatural_{0} \right)+ (1-\tau) \pi_0^\nnatural ( u^{-3} \nu^\nnatural_{0} ) 
  = b S_0 + (1-\tau) \pi_0^\nnatural \left( u^{-3} \nu^\nnatural_{0} \right) . 
\end{equation*}
By using~(\ref{eq:S.u2nunn}) we also get $u^{-3}\nu^\nnatural_{0} = \eta \xi \hat{N} + \sum_{j \geq 1} (1+\eta u) \xi^{j+1} \hat{N} u^{2j-1}$, and so 
$\pi_0^\nnatural ( u^{-3} \nu^\nnatural_{0}) =\eta \xi \hat{N}$. Therefore, we see that $ \pi^\nnatural (\dl+\delta)\left( u^{-2} \nu^\nnatural_{0} \right)=  b S_0 + \eta \xi \hat{N}$. Combining this with~(\ref{eq:S.S0b}) and (\ref{eq:S.1-Tu2munn}) gives~(\ref{eq:S.S0-chain}). The proof is complete.  
\end{proof}

As~(\ref{eq:S.S0-tau}) implies that $S_0$ maps $\ran (1-\tau)$ to $R^T_\bt \subset \ran(1-\tau)$, we see that $S_0$ descends to a unique $k$-linear map,
\begin{equation}
 S:C_\bt^\lambda \longrightarrow C_{\bt-2}, \qquad S\pi^\lambda = \pi^\lambda S_0.
 \label{eq:S}
\end{equation}
We also let $\psi^\nnatural:C_\bt^\nnatural \rightarrow C_{\bt-1}^\lambda$ be the $k$-linear map defined by $\psi^\nnatural=\pi^\lambda \psi_0^\nnatural$. The $T$-compatibility of $\psi_0^\nnatural$ then implies that $\psi^\nnatural$ descends to a unique $k$-linear map 
$\wb{\psi}^\nnatural:  C^\nnatural_{T,\bt} \rightarrow C_{\bt-1}^\lambda$ such that $\wb{\psi}^\nnatural \pi_T= \psi^\nnatural=\pi^\lambda \psi_0^\nnatural$. 

\begin{proposition}\label{prop:S.para-precyclic}
 Let $C$ be a para-precyclic $k$-module with $k\supset \Q$. 
\begin{enumerate}
\item The $k$-linear map $S:C_\bt^\lambda \rightarrow C_{\bt-2}$ is the unique  chain map such that $\wb{\nu}^\nnatural S=u^{-2} \wb{\nu}^\nnatural$. In particular, 
$(C^\lambda, S)$ is an $S$-module and the chain map $\wb{\nu}^\nnatural:C^\lambda_\bt \rightarrow C^\nnatural_{T,\bt}$ is an $S$-map. 

\item The chain maps $\pi^\nnatural: C_\bt^\nnatural \rightarrow C_\bt^\lambda$ and  $\wb{\pi}^\nnatural: C^\nnatural_{T,\bt} \rightarrow C_\bt^\lambda$ are $S$-maps up to homotopy. Namely, 
\begin{equation}
 S \pi^\nnatural -\pi^\nnatural u^{-2}= b\psi^\nnatural + \psi^\nnatural (\dl +\delta), \qquad  
 S \wb{\pi}^\nnatural -\wb{\pi}^\nnatural u^{-2}= b\wb{\psi}^\nnatural + \wb{\psi}^\nnatural (\dl +\delta).
 \label{eq:S.pinn-S-map}
\end{equation}

 \item When $C$ is quasi-precyclic, the chain map $\nu^{T,\nnatural}: C^\lambda_\bt \rightarrow C_{\bt}^\nnatural$ given by~(\ref{eq:Cnn-Cl.nuTnn}) is an $S$-map. 
\end{enumerate}
\end{proposition}
\begin{proof}
As $S\pi^\lambda =\pi^\lambda S_0$, we have $(Sb -bS)\pi^\lambda = S\pi^\lambda b - b\pi^\lambda S_0=\pi^\lambda (S_0b-bS_0)$.  As~(\ref{eq:S.S0-chain}) implies that 
$\ran (S_0b-bS_0)\subset \ran(1-\tau)$, we deduce that $(Sb -bS)\pi^\lambda=0$. This shows that $Sb=bS$ on $C_\bt^\lambda$, i.e., 
$S:C_\bt^\lambda \rightarrow C_{\bt-2}$ is a chain map. In particular, the pair $(C^\lambda, S)$ is an $S$-module.
 
 By using the equality $\wb{\nu}^\nnatural \pi^\lambda = \pi_T \nu^\nnatural_0$ we also get
 \begin{equation*}
 \left(\wb{\nu}^\nnatural S-u^{-2} \wb{\nu}^\nnatural\right) \pi^\lambda = \wb{\nu}^\nnatural \pi^\lambda S_0 - u^{-2}  \pi_T \nu^\nnatural_0 =\pi_T\big(  \nu^\nnatural_0 S_0- u^{-2}  \nu^\nnatural_0 \big).
\end{equation*}
Combining this with~(\ref{eq:S.nunn0-Smap}) gives $(\wb{\nu}^\nnatural S-u^{-2} \wb{\nu}^\nnatural) \pi^\lambda=-\pi_T (1-T)\mu^\nnatural \eta \hat{D}b' =0$. It then follows that $\wb{\nu}^\nnatural S=u^{-2} \wb{\nu}^\nnatural$ on $C^\lambda_\bt$. Moreover, as $\wb{\nu}^\nnatural$ is a right-inverse of $\wb{\pi}^\nnatural$ on $C_\bt^\lambda$ we get 
\begin{equation*}
 S = \wb{\pi}^\nnatural\wb{\nu}^\nnatural S= \wb{\pi}^\nnatural\left(u^{-2} \wb{\nu}^\nnatural\right). 
\end{equation*}
Similarly, if $S':C_\bt^\lambda \rightarrow C_{\bt-2}^\lambda$ is any chain map such that $\wb{\nu}^\nnatural S'=u^{-2} \wb{\nu}^\nnatural$, then 
$S' =\wb{\pi}^\nnatural(u^{-2} \wb{\nu}^\nnatural)=S$. Thus, $S$ is the unique chain map such that $\wb{\nu}^\nnatural S=u^{-2} \wb{\nu}^\nnatural$. In particular, this implies that  the chain map $\wb{\nu}^\nnatural:C^\lambda_\bt \rightarrow C^\nnatural_{T,\bt}$ is an $S$-map. 

We have $S\pi^\nnatural -\pi^\nnatural u^{-2} = S\pi^\lambda \pi_0^\nnatural -\pi^\lambda \pi_0^\nnatural u^{-2} = \pi^\lambda (S_0 \pi_0^\nnatural -\pi_0^\nnatural u^{-2})$. Using~(\ref{eq:S.pinn0-Smap}) we get
\begin{equation*}
 S\pi^\nnatural -\pi^\nnatural u^{-2} = \pi^\lambda b \psi_0^\nnatural + \pi^\lambda  \psi_0^\nnatural (\dl+\delta)=  b\psi^\nnatural + \psi^\nnatural (\dl +\delta). 
\end{equation*}
As $ (S \wb{\pi}^\nnatural -\wb{\pi}^\nnatural u^{-2})\pi_T= S\pi^\nnatural -\pi^\nnatural u^{-2}$ and $\psi^\nnatural=\wb{\psi}^\nnatural \pi_T$, we further get
\begin{equation*}
 (S \wb{\pi}^\nnatural -\wb{\pi}^\nnatural u^{-2})\pi_T=  b(\wb{\psi}^\nnatural \pi_T) + (\wb{\psi}^\nnatural \pi_T(\dl +\delta)= 
 \left[b\wb{\psi}^\nnatural + \wb{\psi}^\nnatural (\dl +\delta)\right]\pi_T. 
\end{equation*}
It then follows that $S \wb{\pi}^\nnatural -\wb{\pi}^\nnatural u^{-2}= b\wb{\psi}^\nnatural + \wb{\psi}^\nnatural (\dl +\delta)$ on $C^\nnatural_{T,\bt}$. Therefore, we see that the chain maps $\pi^\nnatural$ and  $\wb{\pi}^\nnatural$ are $S$-maps up to homotopy, 

Suppose now that $C$ is a quasi-precyclic $k$-module. By its very definition~(\ref{eq:Cnn-Cl.nuTnn}) the chain map $\nu^{T,\nnatural}: C^\lambda_\bt \rightarrow C_{\bt}^\nnatural$ is such that $\nu^{T,\nnatural}\pi^\lambda = \pi^T \nu_0^\natural$, where $\pi^T:C_\bt \rightarrow C_\bt$ is the projection on $C^T_\bt$ defined by the splitting 
$C_\bt=C_\bt^T\oplus R^T_\bt$. Therefore, as above we have  
\begin{equation*}
 \nu^{T,\nnatural} S - u^{-2} \nu^{T,\nnatural}= \pi^T( \nu_0^\nnatural S-u^{-2}\nu_0^\nnatural)=  \pi^T(1-T)\mu^\nnatural\eta \hat{D} b =0.
\end{equation*}
This shows that $\nu^{T,\nnatural}S = u^{-2} \nu^{T,\nnatural}$ on $C^\lambda_\bt$, and so  $\nu^{T,\natural}$ is an $S$-map. The proof is complete. 
\end{proof}

When $C$ is a precyclic $k$-module the chain maps $\wb{\nu}^\nnatural$ and $\nu^{T,\nnatural}$ both agree with the chain map  $\nu^\nnatural: C^\lambda_\bt \rightarrow C_{\bt}^\nnatural$ given by~(\ref{eq:Cnn-Cl.nunn}). Therefore, in this case we obtain the statement. 

\begin{corollary}[compare~\cite{Ka:Crelle90}]\label{cor:S.precyclic} 
 Suppose that $C$ is a precyclic $k$-module with $k\supset \C$. 
 \begin{enumerate}
 \item $S:C_\bt^\lambda \rightarrow C_{\bt-2}$ is the unique  chain map such that $\nu^\nnatural S=u^{-2}\nu^\nnatural$. 
 
 \item The chain map $\nu^\nnatural: C^\lambda_\bt \rightarrow C_{\bt}^\nnatural$ is an $S$-map. 
\end{enumerate}
\end{corollary}

\begin{remark}\label{rmk:S.Kassel-hS0}
 We recover the original version of Corollary~\ref{cor:S.precyclic} in~\cite{Ka:Crelle90} by replacing $\hat{D}$ by $-\check{D}$ in the formula~(\ref{eq:Cn-Cl.vphi}) for $\varphi$ (where $\check{D}$ is given by~(\ref{eq:CnCl.D})), converting $\varphi$ into a special chain homotopy $\check{\varphi}$ (\emph{cf}.~Remark~\ref{rmk:Perturbation.special}), and then substituting $\check{\varphi}$ for $\varphi$ in the definitions of $\nu^\nnatural$ and $S_0$. This amounts to replace $S_0$ by the operator,
\begin{equation}
 \check{S}_0:= \hat{b}'(1-\tau) \check{D}^2 b\hat{N}.
 \label{eq:S.hS0} 
\end{equation}
In fact, when $C$ is precyclic it can be shown that $S_0$ and $\check{S}_0$ descend to the same operator on $C^\lambda_\bt$ (see Remark~\ref{rmk:S.S0-cS0}), and so in the precyclic  case we recover the $S$-operator of~\cite{Ka:Crelle90}. 
\end{remark}

Let $C=(C_\bt, d, s, t)$ be an $H$-unital para-precyclic $k$-module. We shall now re-interpret the operator $S$ in terms of the maps $\pi^\natural$ and $\nu_0^\natural$. 

In the following, we let $\psi^\natural :C_\bt^\natural \rightarrow C_{\bt-1}^\lambda$ and $\wb{\psi}^\natural :C_{T,\bt}^\natural \rightarrow C_{\bt-1}$ be the $k$-linear maps defined by $\psi^\natural= \psi^\nnatural I$ and $\wb{\psi}^\natural= \wb{\psi}^\nnatural I$. In fact, it follows from~(\ref{eq:S.psinn0}) that $\psi^\natural=\wb{\psi} \pi_T=\pi^\lambda \psi_0^\natural$, where $\psi_0^\natural: C_\bt^\natural \rightarrow C_{\bt-1}$ is the $k$-linear map given by
\begin{equation}
 \psi_0^\natural(xu^0)=-b'\hat{D}x, \qquad  \psi_0^\natural(xu)=-\hat{x}-b'\hat{D}s'Nx, \qquad \psi_0^\natural(xu^p)=0, \ p\geq 2. 
 \label{eq:S.psin0}
\end{equation}

\begin{proposition}\label{prop:S.nun}
Let $C$ be an $H$-unital para-precyclic $k$-module with $k\supset \Q$. 
 \begin{enumerate}
\item For all $x\in C_\bt$, we have
\begin{equation}
 S(x^\lambda) = \pi^\natural \big( u^{-1}\nu_0^\natural(x)\big). 
 \label{eq:S.nun-S}
\end{equation}

 \item The chain map $\wb{\nu}^\natural: C^\lambda_\bt \rightarrow C_{T,\bt}^\natural$ is an $S$-map.
 
 \item The chain maps $\pi^\natural:C_\bt^\natural \rightarrow C_\bt^\lambda$ and $\wb{\pi}^\natural:C_{T,\bt}^\natural \rightarrow C_\bt^\lambda$ are $S$-maps up to homotopy. Namely,
\begin{equation*}
 S \pi^\natural -\pi^\natural u^{-1}= b\psi^\natural + \psi^\natural (b+Bu^{-1}), \qquad  
 S \wb{\pi}^\natural -\wb{\pi}^\natural u^{-1}= b\wb{\psi}^\natural + \wb{\psi}^\natural (b+Bu^{-1}).
\end{equation*}
 
 \item When $C$ is quasi-precyclic, the chain map $\nu^{T,\natural}: C^\lambda_\bt \rightarrow C_{\bt}^\natural$ given by~(\ref{eq:Cn-Cl.nuTn}) is an $S$-map. 
 
  \item When $C$ is precyclic, the chain map $\nu^\natural: C^\lambda_\bt \rightarrow C_{\bt}^\natural$ given by~(\ref{eq:Cn-Cl.nun}) is an $S$-map.
\end{enumerate}
\end{proposition}
\begin{proof}
 Let $x\in C_\bt$. We have $S(x^\lambda)=\pi^\lambda S_0x =\pi^\lambda \pi_0^\nnatural (u^{-2}\nu_0^\nnatural(x))$. It follows from~(\ref{eq:Cn-Cl.pinn0J-pinn0}) that  
$\pi^\nnatural \pi_0^\nnatural = \pi^\lambda (\pi^\natural J-(1-\tau) s'\pi_0^\nnatural u^{-1})=\pi^\natural J$. As $J$ is an $S$-map we deduce that 
 \begin{equation*}
 S(x^\lambda)= \pi^\natural J (u^{-2}\nu_0^\nnatural(x))=  \pi^\natural \big( u^{-1}J\nu_0^\nnatural(x)\big)= \pi^\natural \big( u^{-1}\nu_0^\natural(x)\big). 
\end{equation*}

As $\wb{\nu}^\natural = J \wb{\nu}^\nnatural$ and $J$ and $\wb{\nu}^\nnatural$ are both $S$-maps we see that $\wb{\nu}^\natural$ is an $S$-map. When $C$ is quasi-cyclic we have $\nu^{T,\natural}= J \nu^{T,\nnatural}$. As $\nu^{T,\nnatural}$ is an $S$-map we also see that $\nu^{T,\natural}$. When $C$ is precyclic $\wb{\nu}^\natural$ and $\nu^{T,\natural}$ both agree with $\nu^\natural$, and so $\nu^\natural$ is an $S$-map. 

We have $S\pi^\natural - \pi^\natural u^{-1}= S\pi^\nnatural I - \pi^\nnatural I u^{-1}= (S\pi^\nnatural  - \pi^\nnatural u^{-2})I$. Thus, by using~(\ref{eq:S.pinn-S-map}) we get 
\begin{equation*}
 S\pi^\natural - \pi^\natural u^{-1}= b \psi^\nnatural I + \psi^\nnatural (\dl+\delta)I=  b\psi^\natural + \psi^\natural (b+Bu^{-1}). 
\end{equation*}
Likewise, we have $S \wb{\pi}^\natural -\wb{\pi}^\natural u^{-1}= b\wb{\psi}^\natural + \wb{\psi}^\natural (b+Bu^{-1})$. The proof is complete. 
\end{proof}

\subsection{Explicit formulas for $S$}  
We shall now give a simple formula for the operator $S$ and relate it to Connes' original periodicity operator~(\ref{eq:S.Connes-Periodicity}).

\begin{lemma}\label{lem:S.S-d}
 We have $\pi^\lambda b'\hat{D}bN=\pi^\lambda d\hat{D}dN$. 
\end{lemma}
\begin{proof}
Set $d'= b-b'$, i.e., $d'=(-1)^m d$ on $C_m$. We then have
 \begin{align}
 \pi^\lambda b'\hat{D}bN & = \pi^\lambda b\hat{D}bN - \pi^\lambda d'\hat{D}bN \nonumber \\
 & =  \pi^\lambda b\hat{D}bN - \pi^\lambda d'\hat{D}b'N-\pi^\lambda d'\hat{D}d'N 
 \label{eq:S.pilb'DbN} \\
 &=  \pi^\lambda b\hat{D}bN - \pi^\lambda d'N\hat{D}b+\pi^\lambda d\hat{D}dN, \nonumber
\end{align}
where we have used the fact that $\hat{D}b'N=\hat{D}Nb=N\hat{D}b$. 

The equality $b(1-\tau)=(1-\tau)b'$ implies that $\pi^\lambda b\tau =\pi^\lambda b$. More generally, for any polynomial $P(X)\in k[X]$, we have $\pi^\lambda bP(\tau)=P(1)\pi^\lambda b$. By definition $\hat{D}=m^{-1}D_{m-1}(\tau)$ on $C_{m-1}$, where $D_m(X)$ is given by~(\ref{eq:CnCl.Dm}). Therefore, on $C_m$ we have 
\begin{equation}
 \pi^\lambda b\hat{D}b= m^{-1} \pi^\lambda bD_{m-1}(\tau)b= m^{-1}D_{m-1}(1) \pi^\lambda b^2=0.
 \label{eq:S.pilbDb} 
\end{equation}
In particular, we see that $ \pi^\lambda b\hat{D}bN=0$. 

The relations~(\ref{eq:paracyclic.td}) imply that, on $C_m$ and for $j=0,\ldots,m$,  we have
\begin{equation*}
 d'\tau^j=(-1)^{m+mj}dt^j=(-1)^{m+mj}t^jd_{m-j}=(-1)^{m-j}\tau^j d_{m-j}. 
\end{equation*}
Thus, 
\begin{equation}
 d'N =\sum_{0\leq j \leq m} d'\tau^j= \sum_{0\leq j \leq m}(-1)^{m-j} \tau^j d_{m-j}= \sum_{0\leq j \leq m}(-1)^{j} \tau^{m-j} d_{j}.
 \label{eq:S.d'N}
\end{equation}
Therefore, we have $\pi^\lambda d'N = \sum_{j=0}^m(-1)^{j} \pi^\lambda \tau^{m-j} d_{j}= \sum_{j=0}^m(-1)^{j} \pi^\lambda d_{j}=\pi^\lambda b$. 
Using~(\ref{eq:S.pilbDb}) we then deduce that $\pi^\lambda d'\hat{D}Nb =  \pi^\lambda d'N\hat{D}b = \pi^\lambda b \hat{D}b=0$. Combining this with~(\ref{eq:S.pilb'DbN}) and the equality $ \pi^\lambda b\hat{D}bN=0$ gives $\pi^\lambda b'\hat{D}bN=\pi^\lambda d\hat{D}dN$. The proof is complete.  
\end{proof} 

\begin{remark}\label{rmk:S.S0-cS0}
 By using~(\ref{eq:S.pilbDb}) and the equalities $\check{D}=-\hat{D}+N$ and $(1-\tau)\check{D}=\hat{N}-1$, it can be shown that, when $C$ is precyclic, $\pi^\lambda \check{S}_0= \pi^\lambda S_0$, where $\check{S}_0$ is given by~(\ref{eq:S.hS0}). 
\end{remark}

We are now in a position to prove the following simple formulas for the operator $S$. 

\begin{proposition}\label{prop:S.formulas}
Let $C$ be a para-precyclic $k$-module, and assume that $k\supset \Q$. Then, for all $x\in C_m$, $m\geq 2$, we have 
 \begin{align}
 S(x^\lambda ) & = \frac{-1}{(m-1)}\left(d\hat{D}d\hat{N}x\right)^\lambda \nonumber \\
  &= \frac{1}{(m-1)m} \sum_{0\leq i<j\leq m} (-1)^{i+j} (d_id_jx)^\lambda. 
  \label{eq:S.S-periodicity}
\end{align}
In particular, we recover Connes' periodicity operator when $C$ is a precyclic $k$-module . 
\end{proposition}
\begin{proof}
Let $x\in C_m$, $m\geq 2$. By using~(\ref{eq:S.S0-tau}) and Lemma~\ref{lem:S.S-d} we obtain
\begin{equation*}
S(x^\lambda) = \pi^\lambda S_0x= \pi^\lambda \xi\hat{N}x=-(m-1)^{-1} \pi^\lambda b'\hat{D} b\hat{N} x = -(m-1)^{-1}
 =(d\hat{D}d\hat{N}\hat{x})^\lambda.
\end{equation*}
This gives the first equality in~(\ref{eq:S.S-periodicity}). Recall that $\hat{D}=m^{-1} D_{m-1}(\tau)$, where $D_{m-1}(X)$ is given by~(\ref{eq:CnCl.Dm}). Thus, 
\begin{equation*}
 S(x^\lambda) = \frac{-1}{(m+1)m(m-1)} \left(d D_{m-1}(\tau) dNx\right)^\lambda. 
 \end{equation*}

It is convenient to introduce the following notation. Given $k$-linear maps $f_1:C_m\rightarrow C_{m'}$ and $f_2:C_m\rightarrow C_{m'}$ we shall write $f_1\equiv f_2$ when $f_1-f_2=(1-\tau)g$ for some $k$-linear map $g:C_m\rightarrow C_{m'}$. In particular, this implies that $\pi^\lambda f_1=\pi^\lambda f_2$. Using this notation, we see that in order to prove the 2nd equality in~(\ref{eq:S.S-periodicity}) it is enough to show that on $C_m$ we have
\begin{equation}
 d D_{m-1}(\tau) dN \equiv -(m+1)  \sum_{0\leq i<j\leq m} (-1)^{i+j} d_id_j. 
 \label{eq:S.dDdN-didj2}
\end{equation}
 
It follows from~(\ref{eq:CnCl.Dm}) and~(\ref{eq:S.d'N}) that on $C_m$ the operator $D_{m-1}(\tau)dN$ is equal to 
\begin{equation*}
 \sum_{\substack{0\leq i\leq m-1\\ 0\leq j \leq m}} (-1)^{m+j}(m-1-i) \tau^{i+m-j} d_j =  
 \sum_{\substack{0\leq i\leq m-1\\ 0\leq j \leq m}} (-1)^{m+j}(m-i-1)T \tau^{i-j} d_j.
\end{equation*}
Thus, 
\begin{equation}
 dD_{m-1}(\tau)dN =  -\sum_{\substack{0\leq i\leq m-1\\ 0\leq j \leq m}} (-1)^{m-1+j}(m-i-1)Td_{m-1} \tau^{i-j} d_j \equiv - \Delta^{(1)}- \Delta^{(2)},
 \label{eq:S.dDdN-Delta12}
\end{equation}
 where we have set
 \begin{gather*}
  \Delta^{(1)} = \sum_{0\leq i < j\leq m} (-1)^{m-1+j}(m-i-1)d_{m-1} \tau^{i-j} d_j,\\
   \Delta^{(2)} = \sum_{0\leq j \leq i\leq m-1} (-1)^{m-1+j}(m-i-1)d_{m-1} \tau^{i-j} d_j. \\
\end{gather*}

If $0\leq i<j\leq m$, then on $C_{m-1}$ we have 
\begin{equation*}
 d_{m-1}\tau^{i-j}=d_{m-1}T^{-1} \tau^{m-(j-i)}=(-1)^{m-(i-j)}T^{-1} \tau^{m-(j-i)}d_{j-i-1} \equiv (-1)^{m-i+j}d_{j-i-1}. 
\end{equation*}
Combining this with the change of index $i\rightarrow j-i-1$ gives 
\begin{equation}
 \Delta^{(1)}  \equiv \sum_{0\leq i < j\leq m} (-1)^{i-1}(m-i-1)d_{j-i-1} d_j 
  \equiv \sum_{0\leq i < j\leq m-1} (-1)^{i+j}(m+j-i)d_{i} d_j. 
  \label{eq:S.Delta1}
\end{equation}

If $0\leq j \leq i\leq m-1$, then $d_{m-1}\tau^{i-j}=(-1)^{i-j} \tau^{i-j} d_{m-i+j-1} \equiv (-1)^{i-j} d_{m-i+j-1}$. Thus, 
\begin{align*}
 \Delta^{(2)} & \equiv \sum_{0\leq j \leq i\leq m-1} (-1)^{m+j-1}(m-i-1)d_{m-i+j-1}d_j  \\ 
 & \equiv \sum_{0\leq j \leq i\leq m-1} (-1)^{i+j} (i-j) d_i d_j,
\end{align*}
 where we have used the change of index $i\rightarrow m-i+j-1$ to get the 2nd line. As $d_id_j=d_jd_{i+1}$ for $i\geq j$, we further obtain 
 \begin{equation*}
 \Delta^{(2)}\equiv \sum_{0\leq j \leq i\leq m-1} (-1)^{i+j} (i-j) d_j d_{i+1} \equiv \sum_{0\leq j < i\leq m-1} (-1)^{i+j-1} (i-j-1) d_j d_{i}.
\end{equation*}
 Upon interchanging the indices $i$ and $j$ we then get
\begin{equation*}
 \Delta^{(2)}\equiv \sum_{0\leq j < i\leq m-1} (-1)^{i+j-1} (j-i-1) d_i d_{j}\equiv \sum_{0\leq j < i\leq m-1} (-1)^{i+j} (i-j+1) d_i d_{j}. 
\end{equation*}
 Combining this with~(\ref{eq:S.dDdN-Delta12}) and~(\ref{eq:S.Delta1}) gives~(\ref{eq:S.dDdN-didj2}). As mentioned above this proves the 2nd equality in~(\ref{eq:S.S-periodicity}). The proof is complete. 
\end{proof}

\subsection{Applications to periodic cyclic homology} 
In Connes gave 
There is a well known formula expressing the periodicity operator $S$ in terms of the $(b, B)$-operators in cyclic cohomology due to Connes~\cite[Lemma~II.34]{Co:IHES85}. As an application of Proposition~\ref{prop:S.nun} we shall  obtain a dual version of Connes' formula for arbitrary $H$-unital para-precyclic $k$-modules.

Let $C=(C_\bt, d,s,t)$ be an $H$-unital para-precyclic $k$-module. We have
\begin{equation*}
 B(1-\tau)=(1-\tau)s'N(1-\tau) =(1-\tau)s(1-T)=(1-T)(1-\tau)s'. 
\end{equation*}
Therefore, the operator $B$ maps $\ran(1-\tau)$ to $R^T_\bt$, and so it descends to a unique $k$-linear map,
\begin{equation*}
 B:C_\bt^\lambda \longrightarrow C_{T,\bt}, \qquad B \pi^\lambda =\pi_T B. 
\end{equation*}
Thus, given any $x\in C_\bt$, we have $B(x^\lambda)=\overline{Bx}=B\wb{x}$ (where, as above, $\wb{\cdot}$ denotes the class in $C_{T,\bt}$). We then have the following dual version of Connes' formula. 

\begin{proposition}\label{prop:S.Connes-formula}
 Suppose that $k\supset \Q$. Let $x\in C_\bt$ be such that $bx\in \ran (1-\tau)$. Then, in $C_{T,\bt}$ we have
\begin{equation*}
 B\circ S(x^\lambda)  = -b\wb{x} \qquad \bmod \ran \big[b(1-\tau)\big]. 
\end{equation*}
 \end{proposition}
\begin{proof}
 Let $x\in C_\bt$ be such that $bx\in \ran (1-\tau)$. Thanks to~(\ref{eq:S.nun-S}) we have
  \begin{equation*}
 B\circ S(x^\lambda)= B\pi^\lambda \pi^\natural_0(u^{-1} \nu^\natural_{0}(x)) =  \pi_T B\pi_0^\natural (u^{-1} \nu^\natural_{0}(x))= \pi_T \pi_0^\natural (Bu^{-1} \nu^\natural_{0}(x)). 
\end{equation*}
 By~(\ref{eq:Cn-Cl.nun0-chain-map}) we have $Bu^{-1}\nu^\natural_{0}(x) = -b\nu^\natural_{0}(x)+ \nu^\natural_{0}(bx) \bmod R^T_\bt$. Thus, 
\begin{equation}
 B\circ S(x^\lambda)=- \pi_T\pi_0^\natural b\nu^\natural_{0}(x) + \pi_T\pi_0^\natural  \nu^\natural_{0}(bx) 
 = - b\pi_T\big[\pi_0^\natural \nu^\natural_{0}(x)\big] + \pi_T\big[\pi_0^\natural \nu^\natural_{0}(bx)\big].   
 \label{eq:Cn-Cl.BS0-b}
\end{equation}
 As~(\ref{eq:Cn-Cl.pio-nun0}) implies that $\pi_0^\natural \nu^\natural_{0}(x)-x \in \ran(1-\tau)$, we see that $ b\pi_T[\pi_0^\natural \nu^\natural_{0}(x)] =-b\wb{x} \bmod \ran [b(1-\tau)]$. 
 By~(\ref{eq:Cn-Cl.pio-nun0-hN}) we also have  $\pi_0^\natural \nu^\natural_{0}= [1-(1-\tau)s'\hat{D} b]\hat{N}$. By assumption $bx=(1-\tau)y$ for some $y\in C_{\bt-1}$. Thus,   \begin{equation*}
\pi_T  \big[\pi_0^\natural \nu^\natural_{0} (bx)\big] =  \pi_T \big\{[1-(1-\tau)s'\hat{D} b]\hat{N}(1-\tau)y\big\}=[1-(1-\tau)s'\hat{D} b]\pi_T\big[(1-T)\hat{y}\big]=0. 
\end{equation*}
 Combining this with~(\ref{eq:Cn-Cl.BS0-b}) gives $ B \circ S(x^\lambda) = -b\wb{x} \bmod \ran b(1-\tau)$. The proof is complete. 
 \end{proof}

We still assume that $C$ is an $H$-unital para-precyclic $k$-module. We then can form its periodic para-complex $C^\sharp=(C^\sharp_\bt, b+B, T)$, where $C^\sharp_\bt$ is defined as in~(\ref{eq:Parachain.periodic}).  When $C$ is precyclic we actually get a chain complex whose homology is denoted by $\HP_\bt(C)$. In particular, $C_T^\sharp$ is a chain complex, and so we can define the periodic cyclic homology $\HP_\bt(C_T)$. We have a natural chain map  $\wb{\pi}^\sharp: C^\sharp_{T,i} \rightarrow \prod_{q\geq 0} C^\lambda_{2q+i}$ given by
\begin{equation}
\wb{\pi}^\sharp\big[ (\wb{x}_{2q+i})_{q\geq 0}\big] =\big(x^\lambda_{2q+i}\big)_{q\geq 0}, \qquad x_{2q+i}\in C_{2q+i}. 
 \label{eq:periodic-homology.pisharp} 
\end{equation}
 
 Suppose that $k\supset \Q$. As $(C^\lambda, S)$ is an $S$-module, we may also form its periodic chain complex $C^{\lambda,\sharp}:=(C^{\lambda,\sharp}_\bt, b)$, where $C^{\lambda,\sharp}_\bt$ is defined as in~(\ref{eq:Para-S-Mod.Csharp2}). We denote by $H^{\lambda,\sharp}_\bt(C)$ its homology. We have an inclusion of chain complexes $C^{\lambda,\sharp}_\bt \subset \prod_{q\geq 0} C_{2q+\bt}^\lambda$. At the homology level this gives rise to a canonical ($\Z_2$-graded) $k$-linear map $H^{\lambda,\sharp}_\bt\rightarrow \prod_{q\geq 0} H_{2q+\bt}^\lambda(C)$, which fits into the exact sequence~(\ref{eq:Para-S-Mod.exact-sequence-periodicity}). In particular its range is precisely the inverse limit $\varprojlim_S H_{2q+\bt}^\lambda(C)$. 

As $k\supset \Q$ we have the chain map $\wb{\nu}^\natural:C^\lambda_\bt \rightarrow C^\natural_{T,\bt}$ given by~(\ref{eq:Cn-Cl.wbnun}). By Proposition~\ref{prop:CnCl.CTn} this is a right-inverse and chain homotopy left-inverse of the canonical chain map $\wb{\pi}^\natural:C_{T,\bt}^\natural \rightarrow C^\lambda_\bt$. By Proposition~\ref{prop:S.nun} this is an $S$-map, and so it gives rise to a chain map $\wb{\nu}^\sharp: C^{\lambda,\sharp}_\bt \rightarrow C^\sharp_{T,\bt}$. This uses the identification $\varprojlim_{u^{-1}} C^\natural_{T,\bt}\simeq  C^\sharp_{T,\bt}$ given by $(\wb{x}_{2q+i})_{q\geq 0} \rightarrow (\pi_0^\natural (\wb{x}_{2q+i}))_{q\geq 0}$. As~(\ref{eq:CnCl.nun}) implies that $\pi_0^\natural \nu_0^\natural= [ 1-(1-\tau)s'\hat{D}b]\hat{N}$, we have 
\begin{equation}
  \wb{\nu}^\sharp \big[(x_{2q+i}^\lambda)_{q\geq 0}\big] = \left(  [ 1-(1-\tau)s'\hat{D}b]\hat{N}\wb{x}_{2q+i}\right)_{q\geq 0} , \qquad (x_{2q+i}^\lambda)_{q\geq 0} \in 
  C^{\lambda,\sharp}_i. 
  \label{eq:periodic-homology.nusharp}
\end{equation}
As $\wb{\pi}^\sharp \wb{\nu}^\sharp =1$ on $C^{\lambda,\bt}_\bt$was  we actually get the following commutative diagram of chain maps, 
\begin{equation}
 \begin{tikzcd}[column sep=tiny]
   & C^\sharp_{T,\bt} \arrow[dr, "\wb{\pi}^\sharp"] & \\
 C^{\lambda,\sharp}_\bt \arrow[ur, "\wb{\nu}^\sharp"]  \arrow[rr, hook]   &  & {\displaystyle \prod_{q\geq 0} C^\lambda_{2q+\bt}}. 
\end{tikzcd}
\label{eq:periodic-homology.chain-diagram}
\end{equation}

\begin{proposition}\label{prop:periodic-homology.quasi-isom}
 Suppose that $k\supset \Q$. 
 \begin{enumerate}
 \item The chain map $\wb{\nu}^\sharp: C^{\lambda,\sharp}_\bt \rightarrow C^\sharp_{T,\bt}$ given by~(\ref{eq:periodic-homology.nusharp}) is a quasi-isomorphism. 
 
 \item The diagram~(\ref{eq:periodic-homology.chain-diagram}) gives rise to the following commutative diagram, 
 \begin{equation}
 \begin{tikzcd}[column sep=tiny]
   & \HP_\bt(C_{T}) \arrow[dr, "\wb{\pi}^\sharp"] & \\
 H^{\lambda,\sharp}_\bt(C) \arrow[ur, "\wb{\nu}^\sharp"]  \arrow[rr, two heads]   & & \varprojlim_S H^{\lambda}_{2q+\bt} (C). 
\end{tikzcd}
\label{eq:periodic-homology.homology-diagram}
\end{equation}

\item  The downward arrow $\wb{\pi}^\sharp: \HP_\bt(C_{T}) \rightarrow \varprojlim_S H^{\lambda}_{2q+\bt} (C)$ is onto. 
\end{enumerate}
 \end{proposition}
\begin{proof}
By Proposition~\ref{prop:Para-S-Mod.quasi-isomorphism-periodic} the very fact that the $S$-map $\wb{\nu}^\sharp:C_\bt^\lambda \rightarrow C^\sharp_{T,\bt}$ is an $S$-map ensures us that at the periodic level the corresponding chain map $\wb{\nu}^\sharp: C^{\lambda,\sharp}_\bt \rightarrow C^\sharp_{T,\bt}$ is a quasi-isomorphism. 

At the homology level the commutative diagram~(\ref{eq:periodic-homology.chain-diagram}) gives rise to the following commutative diagram, 
 \begin{equation*}
 \begin{tikzcd}[column sep=tiny]
   & \HP_\bt(C_{T}) \arrow[dr, "\wb{\pi}^\sharp"] & \\
 H^{\lambda,\sharp}_\bt(C) \arrow[ur, "\wb{\nu}^\sharp"]  \arrow[rr]   & &\prod_{q\geq 0}  H^{\lambda}_{2q+\bt} (C). 
\end{tikzcd}
\end{equation*}
As mentioned above the range of the bottom horizontal arrow is precisely $\varprojlim_S H^{\lambda}_{2q+\bt} (C)$. Therefore, in order to get the diagram~(\ref{eq:periodic-homology.homology-diagram}) it only remains to show that at the homology level $\wb{\pi}^\sharp$ maps $ \HP_\bt(C_{T})$ to $\varprojlim_S H^{\lambda}_{2q+\bt} (C)$. 

Let $\wb{\psi}^\sharp: C_{T,\bt} \rightarrow \prod_{q\geq 0} C^\lambda_{2q+\bt}$ be the $k$-linear map such that, for any $\wb{x}=(\wb{x}_{2q+i})_{q\geq 0}$ in $C_{T,i}^\sharp$ with $x_{2q+i}\in C_{2q+i}$, we have 
\begin{equation*}
 \psi^\sharp(\wb{x})= \big( \psi^\natural(x_{2q+i}^\natural)\big)_{q\geq 0},
\end{equation*}
where the map $\psi^\natural: C_{\bt}^\natural \rightarrow C_{\bt-1}^\lambda$ is defined as in Section~\ref{sec:S} (see Eq.~(\ref{eq:S.psin0})), 
 and we have set $x_{2q+i}^\natural=x_{2q+i}u^0+ x_{2q-2+i}u + \cdots + x_i u^q \in C^\natural_{2q+i}$.  
\begin{claim*}
 On $C^\sharp_{T,\bt}$ we have
 \begin{equation}
 S\wb{\pi}^\sharp - \wb{\pi}^\sharp = b \psi^\sharp + \wb{\psi}^\sharp  (b+B).
 \label{eq:S.Spi-pi-homotopy} 
\end{equation}
\end{claim*}
\begin{proof}[Proof of the Claim] 
 Let $\wb{x}=(\wb{x}_{2q+i})_{q\geq 0}\in C_{T,i}^\sharp$, $x_{2q+i}\in C_{2q+i}$. We have 
 \begin{equation*}
     \pi^\sharp (\wb{x})= (x^\lambda_{2q+i})_{q\geq 0}= (\pi^\natural(x_{2q+i}))_{q\geq 0}, 
\end{equation*}
where $x_{2q+i}^\natural$ is defined as above. Note that $u^{-1} x_{2q+2+i}^\natural=x_{2q+i}^\natural$. Thus, 
\begin{align*}
 S\wb{\pi}^\sharp( \wb{x})- \wb{\pi}^\sharp( \wb{x}) & = \big( S\pi^\natural(x_{2q+2+i}^\natural)\big)_{q\geq 0} -  \big(\pi^\natural(x_{2q+i}^\natural)\big)_{q\geq 0} \\
 & =  \big( (S\pi^\natural-\pi^\natural u^{-1})(x_{2q+2+i}^\natural)\big)_{q\geq 0}. 
\end{align*}
 Combining this with Proposition~\ref{prop:S.nun} we get
\begin{equation}
S\wb{\pi}^\sharp( \wb{x})- \wb{\pi}^\sharp( \wb{x}) =  \big( b\psi^\natural(x^\natural_{2q+2+i}) \big)_{q\geq 0} +  \big( \psi^\natural[(b+Bu^{-1})x^\natural_{2q+2+i}] \big)_{q\geq 
0} 
 \label{prop:periodic-homology.Spi-piu}
\end{equation}

As $b\psi^\natural(x_i^\natural)=0$, we have $ ( b\psi^\natural(x^\natural_{2q+2+i}))_{q\geq 0} = ( b\psi^\natural(x^\natural_{2q+i}))_{q\geq 0}= b\wb{\psi}^\sharp  (\wb{x})$. 
Likewise, as $\psi^\natural(bx_i^\natural)=0$ we also have  $( \psi^\natural(bx^\natural_{2q+2+i}))_{q\geq 0} =\wb{\psi}^\sharp  (b\wb{x})$. In addition, we have 
\begin{equation*}
  \big( \psi^\natural(Bu^{-1}x^\natural_{2q+2+i}) \big)_{q\geq 0} =  \big( \psi^\natural(Bx^\natural_{2q+i}) \big)_{q\geq 0}= \wb{\psi}^\sharp  (B\wb{x}). 
\end{equation*}
Combining all this with~(\ref{prop:periodic-homology.Spi-piu}) then gives 
\begin{equation*}
 S\wb{\pi}^\sharp( \wb{x})- \wb{\pi}^\sharp( \wb{x}) 
 = b\wb{\psi}^\sharp  (\wb{x}) + \wb{\psi}^\sharp  [(b+B)\wb{x}] . 
\end{equation*}
This proves the claim. 
\end{proof}

The above claim implies that $S\wb{\pi}^\natural -\wb{\pi}^\natural =0$ on $\HP_\bt(C_T)$, and so $\wb{\pi}^\sharp$ maps $ \HP_\bt(C_{T})$ to $\ker(1-S)=
\varprojlim_S H^{\lambda}_{2q+\bt} (C)$. We thus obtain the commutative diagram~(\ref{eq:periodic-homology.homology-diagram}). In that diagram
 the bottom horizontal arrow  is onto and the upward arrow $\wb{\nu}^\sharp: H^{\lambda,\sharp}_\bt(C) \rightarrow \HP_\bt(C_{T})$ is an isomorphism. It then follows that the downward arrow $\wb{\pi}^\sharp: \HP_\bt(C_{T}) \rightarrow  \varprojlim_S H^{\lambda}_{2q+\bt} (C)$ is onto as well. The proof is complete. 
\end{proof}

\section{Applications in Cyclic Cohomology}\label{sec:cohomology} 
In this section, we explain the counterparts of Proposition~\ref{prop:S.nun} and Proposition~\ref{prop:periodic-homology.quasi-isom} in cyclic cohomology and periodic cyclic cohomology. As we shall see this will provide us with explicit ways to convert $(b,B)$-cocycles into cohomologous periodic cyclic cocycles.

\subsection{Para-$S$-comodules} For the sake of exposition's clarity, we  give a brief overview of the dual version of para-$S$-modules. We shall call the corresponding objects \emph{para-$S$-comodules}. As the dual of a left module is a right module we shall work in the category of right $k$-modules. Thus, a  \emph{para-$S$-comodule} is given by a system $(C^\bt, d,S,T)$, where $C^m$, $m\geq 0$, are \emph{right} $k$-modules and $d:C^\bt \rightarrow C^{\bt+1}$ and $S:C^\bt \rightarrow C^{\bt+2}$ are $k$-linear maps and $T:C^\bt\rightarrow C^\bt$ is a $k$-linear automorphism satisfying the relations~(\ref{eq:Para-S-modules-relations}). In particular, when $T=1$ we get a cochain complex $(C^\bt, d)$ and $S$ is a true cochain map. In this case we shall denote by $H^\bt(C)$ the corresponding cohomology. 

In analogy with para-$S$-modules, we shall call \emph{cochain map} any map between para-$S$-comodules that is compatible with the $d$-operators. A \emph{cochain $S$-ma}p is a cochain map which is compatible with the $(S,T)$-operators. We also have natural notions of \emph{cochain homotopies} of cochain maps and 
\emph{cochain $S$-homotopies} of cochain $S$-maps. This provides us with notions of \emph{cochain homotopy equivalence} and \emph{cochain $S$-homotopy equivalence} between para-$S$-comodules, as well as the corresponding notions of deformation retracts and $S$-deformation retract of para-$S$-comodules. 

Our main example of a para-$S$-module is the dual of a para-$S$-module. Given any left $k$-module $\sE$ we denote by $\sE^\dagger$ the dual right $k$-module 
$\Hom_k(\sE, k)$. By duality any $k$-linear map $f:\sE_1\rightarrow \sE_2$ gives rise to a $k$-linear map $f^\dagger: \sE_2^\dagger \rightarrow  \sE_1^\dagger$ such that $f^\dagger(\phi)= \phi \circ f\ \forall \phi \in \sE_2^\dagger$. 
Any para-$S$-module $C=(C_\bt, d, S,T)$ then gives rise to a dual para-$S$-comodule $(C^\bt, d, S,T)$, where $C^m=C_m^\dagger$ and we denote by $(d,S, T)$ the dual maps $(d^\dagger,S^\dagger, T^\dagger)$. 
Any chain map (resp., $S$-map) between para-$S$-modules gives rise to a cochain map (resp., cochain $S$-map) between the corresponding dual para-$S$-comodules. 
Moreover, any chain homotopy equivalence (resp., $S$-homotopy equivalence) of para-$S$-modules gives rise to a cochain homotopy equivalence (resp., cochain $S$-homotopy equivalence) of the dual para-$S$-comodules.

Let $(C^\bt, d, S,T)$ be a para-$S$-comodule. The operator $S:C^\bt \rightarrow C^{\bt+2}$ gives rise to directed systems of right $k$-modules $\{C^{2q+m}\}_{q\geq 0}$, $m\geq 0$, which are compatible with the operators $d$ and $T$. For each $m$, let $C^m_\sharp = \varinjlim_S C^{2q+i}$ be the corresponding direct limit, i.e., 
\begin{equation}
C^m_\sharp = \biggl( \bigoplus _{q\geq 0} C^{2q+i}\biggr) \slash \ran (1-S). 
\label{eq:Para-S-comod.Csharp}
\end{equation}
We have a natural identification $C_\sharp^{m+2} \simeq C_\sharp^{m}$. Under this identification, the operators $d$ and $T$ give rise to operators on $C^0_\sharp \oplus C^1_\sharp$ which are odd and even, respectively. Moreover, we have $d^2=(1-T)$ on $C^0_\sharp \oplus C^1_\sharp$. We thus obtain a $\Z_2$-graded cochain paracomplex $C^\sharp:= (C^0_\sharp \oplus C^1_\sharp, d,T)$, which we shall call the \emph{periodic cochain paracomplex} of $C$. 

Any cochain $S$-map between para-$S$-comodules gives rise to a $T$-compatible cochain map between the corresponding periodic cochain paracomplexes. Furthermore, any cochain $S$-homotopy equivalence of para-$S$-comodules gives rise to a $T$-compatible cochain homotopy equivalence of the corresponding periodic cochain paracomplexes. 

If $C=(C_\bt, d,S, T)$ is a para-$S$-module, then we denote by $C_\sharp$ the periodic cochain paracomplex of its dual para-$S$-comodule. There is a natural duality between $C_\sharp$  and the periodic chain complex $C^\sharp$. Thus, any $S$-map (resp., $S$-homotopy equivalence) between para-$S$-comodules gives rise to a $T$-compatible cochain map (resp., cochain $S$-homotopy equivalence) between the corresponding periodic cochain paracomplexes. 

When $T=1$ we obtain a periodic cochain complex $(C_\sharp^\bt, d)$. We denote by $H_\sharp^\bt(C)$ its cohomology. As the operator $S:C^\bt \rightarrow C^{\bt+2}$ is a cochain map, it descends to a $k$-linear map $S:H_\sharp^\bt(C)\rightarrow H_\sharp^{\bt+2}(C)$, and so it defines a directed system of cohomology modules. The canonical projection $\bigoplus_{q\geq 0} C^{2q+\bt} \rightarrow C_\sharp^\bt$ is an $S$-invariant cochain map, and so it descends to a $k$-linear map from $\varinjlim_S  H^{2q+\bt}(C) \rightarrow H_\sharp^\bt(C)$. As the direct limit functor is exact (see, e.g., \cite{We:IHA}), we actually get a canonical isomorphism, 
\begin{equation}
 {\varinjlim}_S  H^{2q+\bt}(C) \simeq H_\sharp^\bt(C). 
 \label{eq:para-S-comodules.isom-cohomology}
\end{equation}

\subsection{Cyclic cohomology} 
Let $C=(C_\bt, d,s,t)$ be an $H$-unital para-precyclic $k$-module. By duality the $H$-unital precyclic $k$-module $C_T=(C_{T,\bt},d,s,t)$ gives rise to an $H$-unital precyclic $k$-comodule $(C_T^\bt, d,s,t)$, where 
\begin{equation*}
 C_T^m:= \left\{ \phi \in \Hom_k(C_m,k); \ \phi \circ T=\phi\right\}, \qquad m\geq 0. 
\end{equation*}
Here the structural operators $(d,s,t)$ on $C_T^\bt$ are the dual versions of the structural operators $(d,s,t)$ on  $C_{T,\bt}$. We let
 $C_{T,\natural}:=(C_{T,\natural}^\bt, b+B)$ be the corresponding cyclic cochain complex, where 
 \begin{equation}
 C_{T,\natural}^m=C_T^m \oplus C^T_{m-2} \oplus \cdots, \qquad m\geq 0. 
\end{equation}
We denote by $\HC^\bt(C_T)$ the cohomology of $C_{T,\natural}$. The natural duality between the chain complex $C^\natural_{T}$ and the cochain complex $C_{T,\natural}$ is given by
\begin{equation}
 \acou{\phi}{\wb{x}u^p}:= \acou{\phi_{m-2p}}{x}, \qquad x\in C_{m-2p}, \quad \phi=(\phi_{m}, \phi_{m-2},\ldots)\in C^m_{T,\natural}. 
 \label{eq:cohomology.duality-Cn}
\end{equation}

We also let $C_\lambda=(C_\lambda^\bt, b)$ be the Connes cochain complex of $C$, where 
\begin{equation*}
 C^m_\lambda:= \left\{ \phi \in \Hom_k(C_m,k); \ \phi \circ \tau=\phi\right\}, \qquad m\geq 0.
\end{equation*}
We denote by $H_\lambda(C)$ the cohomology of $C_\lambda$. As the $B$-operator on cochains is annihilated by cyclic cochains, we have an inclusion of cochain complexes, 
\begin{equation}
 \wb{\iota}_\natural: C^\bt_\lambda \longrightarrow C_{T,\natural}^\bt, \qquad \wb{\iota}_\natural(\phi)=(\phi,0,\ldots), \quad \phi\in C^m_\lambda. 
 \label{eq:cohomology.canonical-inclusion}
\end{equation}
This is the dual version of the canonical chain map $\wb{\pi}^\natural: C_{T,\bt}^\natural \rightarrow C^\lambda_\bt$. 

Suppose that $k\supset \Q$. By duality the chain map $\wb{\nu}^\natural: C_\bt^\lambda \rightarrow C_{T,\bt}^\natural$ gives rise to a cochain map $\wb{\nu}_\natural:  
C_{T,\natural}^\bt \rightarrow C^\bt_\lambda$. In fact, in view of~(\ref{eq:Cn-Cl.wbnun}), given any  $\phi=(\phi_m, \phi_{m-2},\ldots)\in C^m_{T,\natural}$, $m\geq 0$, we have 
\begin{equation}
 \wb{\nu}_\natural(\phi)= \sum_{m-2j\geq 0} (-1)^j \phi_{m-2j} \circ \left\{ [1-(1-\tau)s'\hat{D}b](\hat{b}'\hat{D}b)^j \hat{N}\right\}. 
 \label{eq:cohomology.nun}
\end{equation}
By Proposition~\ref{prop:CnCl.CTn} the chain map $\wb{\nu}^\natural$ is a right-inverse and chain homotopy left-inverse of the canonical chain map $\wb{\pi}^\natural$. Therefore, by duality we obtain the following result. 

\begin{proposition}\label{prop:cohomology.cyclic}
Assume that $k\supset \Q$. Then the cochain map $\wb{\nu}_\natural$ given by~(\ref{eq:cohomology.nun}) is a left-inverse and cochain homotopy right-inverse of the canonical inclusion~(\ref{eq:cohomology.canonical-inclusion}). 
\end{proposition}

\begin{remark}
 When $C$ is a cyclic module the fact that the canonical inclusion~(\ref{eq:cohomology.canonical-inclusion}) is a quasi-isomorphism goes back to Connes~\cite{Co:CRAS83, Co:IHES85}. 
\end{remark}

The upshot of Proposition~\ref{prop:cohomology.cyclic} is an explicit way to convert any $(b,B)$-cocycle into a cohomologous periodic cocycle. Namely, we have the following statement. 

\begin{corollary}\label{cor:cohomology.cyclic}
Assume that $k\supset \Q$. Then, for every cocycle $\phi=(\phi_m, \phi_{m-2},\ldots)\in C^m_{T,\natural}$, $m\geq 0$, the cyclic cochain $\wb{\nu}_\natural(\phi)$ given by~(\ref{eq:cohomology.nun}) is a cocycle in the same cohomology class as $\phi$. 
\end{corollary}

\subsection{Periodic cyclic cohomology} 
Let $C=(C_\bt, d,s,t)$ be an $H$-unital para-precyclic $k$-module. Under~(\ref{eq:cohomology.duality-Cn}) we have a natural duality between the chain complex $C_T^\natural =(C_{T,\bt}^\natural, b+Bu^{-1})$ and the cochain complex $C_{T,\natural}=(C_{T,\natural}^\bt, b+B)$. Under this duality the dual of the projection 
$u^{-1}:C_{T,\bt}^\natural\rightarrow C_{T,\bt-2}^\natural$ is simply the inclusion of $C^{\bt}_{T,\natural}=  C^{\bt}_{T} \oplus C^{\bt-2}_{T}\oplus \cdots $ into $C^{\bt+2}_{T,\natural}=  C^{\bt+2}_{T} \oplus C^{\bt}_{T}\oplus \cdots $. This turns $C_{T,\natural}$ into an $S$-comodule. Its periodic cochain complex is naturally identified with the $\Z_2$-graded cochain complex $C_{T,\sharp}= (C_{T,\sharp}^\bt, b+B)$, where 
\begin{equation*}
 C_{T,\sharp}^i= \bigoplus_{q\geq 0} C^{2q+i}_T, \qquad i=0,1.
\end{equation*}
More precisely, the identification between $C_{T,\sharp}^i$ and the direct limit $\varinjlim_S C^{2q+i}_T$ is induced from
\begin{equation}
 (\phi_i, \phi_{2+i},\ldots, \phi_{2q_0+i}, 0,\ldots) \longrightarrow (\phi_i^\natural, \phi_{2+i}^\natural,\ldots, \phi_{2q_0+i}^\natural, 0,\ldots),
 \label{eq:cohomology.Cns-Cs}
\end{equation}
where we have set $ \phi_{2q+i}^\natural:=( \phi_{2q+i},\phi_{2q-2+i},\ldots)$.  This identification is the dual version of the identification of $\varprojlim_{u^{-1}} C_{2q+i}$ with $C_i^\sharp$ described in Section~\ref{sec:parachain}. We denote by $\HP^\bt(C_T)$ the cohomology of the cochain complex $C_{T,\sharp}$. This is the cyclic cohomology of the $H$-unital precyclic module $C_T$. 

Suppose that $k\supset \Q$. By duality the chain map $S:C^\lambda_\bt \rightarrow C_{\bt-2}^\lambda$ gives rise to a degree~$2$ cochain map $S:C^{\bt}_\lambda\rightarrow C^{\bt+2}_\lambda$, so that we obtain an $S$-comodule $(C^\lambda_\bt, b, S)$ which is the dual $S$-comodule of the $S$-module $(C^\lambda_\bt, b, S)$. Let $C_{\lambda, \sharp}:=(C_{\lambda, \sharp}^\bt, b, S)$ be its periodic cochain complex, where $C_{\lambda, \sharp}^i$, $i=0,1$, is defined as in~(\ref{eq:Para-S-comod.Csharp}). We denote by $H^\bt_{\lambda, \sharp}(C)$ its cohomology. There is a natural duality between $C_{\lambda, \sharp}$ and the periodic chain complex $C^{\lambda,\sharp}$ introduced in Section~\ref{sec:S}.  

By duality the chain map $\wb{\pi}^\sharp: C_{T,\bt}^\sharp \rightarrow \prod_{q\geq 0} C_{2q+\bt}^\lambda$ in~(\ref{eq:periodic-homology.pisharp}) corresponds to the natural cochain inclusion $\wb{\iota}_\sharp: \bigoplus_{q\geq 0} C^{2q+\bt}_\lambda \hookrightarrow C^\bt_{T,\sharp}$. At the cohomology level this gives rise to a ($\Z_2$-graded) $k$-linear map $\wb{\iota}_\sharp: \bigoplus_{q\geq 0} H^{2q+\bt}_\lambda (C) \rightarrow \HP^\bt(C_{T})$. Although at the level of cochains the inclusion $\wb{\iota}_\sharp$ is not $S$-invariant, at the cohomology level we get an $S$-invariant map. 
This can be seen by using the dual version of Proposition~\ref{prop:S.Connes-formula} (\emph{cf}.\ \cite{Co:IHES85}). 
Alternatively, we know from~(\ref{eq:S.Spi-pi-homotopy}) that  the chain map $(1-S)\wb{\pi}^\sharp$ is chain homotopic to $0$. By duality $\wb{\iota}_\natural(1-S)$ is cochain homotopic to $0$, and hence $\wb{\iota}_\natural(1-S)=0$ on $\bigoplus_{q\geq 0} H^{2q+\bt}_\lambda (C)$. 
Anyway, we see that at the cohomology level the map $\wb{\iota}_\sharp$ descends to a $k$-linear map, 
\begin{equation*}
 \wb{\iota}^\sharp: {\varinjlim}_S H^\lambda_{2q+\bt} \longrightarrow \HP^\bt(C_T). 
\end{equation*}
This map was shown by Connes~\cite{Co:CRAS83, Co:IHES85} (at least when $k=\C$ and $C$ is the cyclic space of a unital $\C$-algebra). 

We shall now use Proposition~\ref{prop:periodic-homology.quasi-isom} and Proposition~\ref{prop:cohomology.cyclic} to reinterpret Connes' isomorphism theorem and exhibit an explicit inverse of $ \wb{\iota}^\sharp$ for $H$-unital precyclic modules.

Let $C=(C,d,s,t)$ be an $H$-unital precyclic $k$-module with $k\supset \Q$. It follows from Proposition~\ref{prop:S.nun} that the cochain map  $\wb{\nu}_\natural: C_{T,\natural}^\bt \rightarrow C^\bt_\lambda$ is a cochain $S$-map. At the periodic level the corresponding cochain map $\wb{\nu}_\sharp: C_{T,\sharp}^{\bt}  
\rightarrow C_{\lambda,\sharp}^\bt$ is the dual version of the chain map $\wb{\nu}^\sharp: C^{\lambda,\sharp}_\bt \rightarrow C^\sharp_{T,\bt}$.
Thus, given any periodic cochain $\phi=(\phi_i, \phi_{2+i}, \ldots )\in C_{T,\sharp}^{i}$,  we have 
\begin{equation}
 \wb{\nu}_\sharp(\phi)= \text{class of} \ \big( \phi_{2q+i} \circ  \big\{[1-(1-\tau)s'\hat{D}b] \hat{N}\big\}\big)_{q\geq 0} \ \text{mod}\ \ran(1-S).
 \label{eq:cohomology.nusharp} 
\end{equation}
By duality the commutative diagram of chain maps~(\ref{eq:periodic-homology.chain-diagram}) then gives rise to the following commutative diagram of cochain maps, 
\begin{equation}
 \begin{tikzcd}[column sep=small]
  & C_{T,\sharp}^{\bt}  \arrow[dl, "\wb{\nu}_\sharp" ']  & \\
 C_{\lambda,\sharp}^\bt     & &  {\displaystyle \arrow[ll, two heads]   \arrow[ul, "\wb{\iota}_\sharp" ']  \bigoplus_{q\geq 0} C^\lambda_{2q+\bt} .} 
\end{tikzcd}
\label{eq:periodic-cohomology.cochain-diagram}
\end{equation}

\begin{proposition}\label{prop:periodic-cohomology.HP}
Let $C=(C,d,s,t)$ be an $H$-unital precyclic $k$-module with $k\supset \Q$.
 \begin{enumerate}
\item The cochain map $\wb{\nu}_\sharp: C_{T,\sharp}^{\bt}  \rightarrow C_{\lambda,\sharp}^\bt$ given by~(\ref{eq:cohomology.nusharp}) is a quasi-isomorphism. 

 \item The following diagram is commutative, 
 \begin{equation}
 \begin{tikzcd}[column sep=tiny]
  & \HP^\bt(C_{T})  \arrow[dl, "\wb{\nu}_\sharp" ']  & \\
 H_{\lambda,\sharp}^\bt (C) &  & \arrow[ll, "\sim" '] \arrow[ul, "\wb{\iota}_\sharp" ']   \varinjlim_S H^\lambda_{2q+\bt}(C) ,  
\end{tikzcd}
\label{eq:periodic-cohomology.cohomology-diagram}
\end{equation}
where the bottom horizontal arrow is the canonical isomorphism~(\ref{eq:para-S-comodules.isom-cohomology}). 

 \item The Connes map $\wb{\iota}_\sharp: {\varinjlim}_S H^\lambda_{2q+\bt} \rightarrow \HP^\bt(C_T)$ is an isomorphism. 
\end{enumerate}
In other words, the cochain map $\wb{\nu}_\sharp$ allows us to invert Connes' map $\wb{\iota}_\sharp: {\varinjlim}_S H^\lambda_{2q+\bt} \rightarrow \HP^\bt(C_T)$ under the canonical identification ${\varinjlim}_S H^\lambda_{2q+\bt} (C)\simeq   H_{\lambda,\sharp}^\bt (C)$. In particular, this allows us to realize Connes' isomorphism by means of an explicit quasi-isomorphism. 
\end{proposition}
\begin{proof}
As the cochain map $\wb{\nu}_\sharp: C_{T,\sharp}^{\bt}  \rightarrow C_{\lambda,\sharp}^\bt$ arises from the cochain $S$-map $\wb{\nu}_\natural: C_{T,\natural}^\bt \rightarrow C^\bt_\lambda$, we have a commutative diagram, 
\begin{equation*}
     \begin{tikzcd}
               \varinjlim  \HC^\bt(C_T)  \arrow[r, "\sim" ]  \arrow[d, "\wb{\nu}_\natural" ']   & \HP^\bt(C_{T})  \arrow[d, "\wb{\nu}_\sharp" ']  \\
                \varinjlim_S H^\lambda_{2q+\bt}(C)  \arrow[r, "\sim" ]  &  H_{\lambda,\sharp}^\bt (C).   
     \end{tikzcd}
\end{equation*}
Here the bottom horizontal arrow is the canonical isomorphism~(\ref{eq:para-S-comodules.isom-cohomology}) and the top horizontal arrow arises from this isomorphism under the identification of $C_{T,\sharp}^\bt$ and $\varinjlim C_{T,\natural}^\bt$ given by~(\ref{eq:cohomology.Cns-Cs}).
We know by Proposition~\ref{prop:cohomology.cyclic} 
 that $\wb{\nu}_\natural: \HC^\bt(C_{T}) 
 \rightarrow H_\lambda^\bt(C)$ is an isomorphism. This isomorphism is compatible with the periodicity operators in cohomology. 
 Therefore, by functoriality of the direct limit, the left vertical arrow $\wb{\nu}_\natural: \varinjlim  \HC^\bt(C_T) \rightarrow  \varinjlim_S H^\lambda_{2q+\bt}(C)$  is an isomorphism. As the horizontal arrows are both isomorphisms, it then follows that the right vertical arrow is an isomorphism as well. That is, the cochain map $\wb{\nu}_\sharp: C_{T,\sharp}^{\bt}  \rightarrow C_{\lambda,\sharp}^\bt$ is a quasi-isomorphism.
 
At the cohomology level the commutative diagram of cochain maps~(\ref{eq:periodic-cohomology.cochain-diagram}) gives rise to the following commutative diagram, 
\begin{equation*}
   \begin{tikzcd}[column sep=tiny]
          & \HP^\bt(C_{T})  \arrow[dl, "\wb{\nu}_\sharp" ']  & \\
          H_{\lambda,\sharp}^\bt (C) &  & \arrow[ll] \arrow[ul, "\wb{\iota}_\sharp" ']  \bigoplus_{q\geq 0}  H^\lambda_{2q+\bt}(C).  
  \end{tikzcd}
\end{equation*}
As the bottom horizontal arrow and the map $\wb{\iota}_\sharp$ both descend to maps on $\varinjlim_S H^\lambda_{2q+\bt}(C)$, we obtain the commutative diagram~(\ref{eq:periodic-cohomology.cohomology-diagram}). In the diagram~(\ref{eq:periodic-cohomology.cohomology-diagram}) the bottom vertical arrow is an isomorphism. As shown above, the downward arrow $\wb{\nu}_\sharp:\HP^\bt(C_{T})  \rightarrow  H_{\lambda,\sharp}^\bt (C)$ is an isomorphism. It then follows that the upward arrow $\wb{\iota}_\sharp: {\varinjlim}_S H^\lambda_{2q+\bt}(C) \rightarrow \HP^\bt(C_T)$ is an isomorphism.  The proof is complete. 
 \end{proof}

\begin{remark}\label{rmk:S.Connes-Chern}
Cyclic cohomology is the natural receptacle of the Chern character in $K$-homology, a.k.a.~Connes-Chern of character~\cite{Co:IHES85}. The original definition of the Connes-Chern character by Connes~\cite{Co:IHES85} is in terms of cyclic cocycles. However, we often get better information from its representations by $(b,B)$-cocycles~\cite{CM:CMP93, CM:GAFA95, CM:CMP98}. It would be interesting to see what light on these representations could be shed from the convertion results mentioned above.  
\end{remark}

\end{document}